\documentclass[preprint]{imsart}

\RequirePackage{amsthm,amsmath}
\RequirePackage[round]{natbib}
\RequirePackage[colorlinks,citecolor=blue,urlcolor=blue]{hyperref}
\usepackage{amssymb}
\usepackage{graphicx}
\usepackage{amsfonts}
\usepackage{amsbsy}
\usepackage{mathrsfs}
\usepackage{lscape}
\usepackage{float}
\theoremstyle{plain}
\DeclareMathOperator{\sign}{sign}
\usepackage{multirow}
\usepackage{tablefootnote}
\usepackage{threeparttable}
\usepackage{pdflscape}
\usepackage{enumerate}
\usepackage{pdfpages}
\usepackage{soul}
\usepackage{verbatim}

\startlocaldefs
\numberwithin{equation}{section}
\theoremstyle{plain}
\newtheorem{theorem}{Theorem}
\usepackage[english]{babel}
\newtheorem{lemma}{Lemma}

\endlocaldefs

\newcommand{\bx}{\mathbf{x}}
\newcommand{\bX}{\mathbf{X}}
\newcommand{\bY}{\mathbf{Y}}
\newcommand{\by}{\mathbf{y}}
\newcommand{\bZ}{\mathbf{z}}
\newcommand{\bS}{\mathbf{s}}
\newcommand{\bt}{\mathbf{t}}
\newcommand{\bd}{\mathbf{d}}
\newcommand{\bH}{\mathbf{H}}
\newcommand{\bV}{\mathbf{V}}
\newcommand{\bbeta}{\boldsymbol{\beta}}
\newcommand{\bpsi}{\boldsymbol{\psi}}
\newcommand{\btheta}{\boldsymbol{\theta}}
\newcommand{\bTheta}{\boldsymbol{\Theta}}

\newcommand{\bSigma}{\boldsymbol{\Sigma}}

\newcommand{\EE}{\mathbb{E}}
\newcommand{\Var}{\mathrm{Var}}

\graphicspath{{../Figures/}}

\usepackage{xcolor}

\begin{document}

\begin{frontmatter}

\title{Inference for possibly high-dimensional inhomogeneous Gibbs point processes}
\runtitle{High-dimensional inference for Gibbs models}

\begin{aug}
\author[A]{\fnms{Isma\"ila} \snm{Ba}\thanksref{t1}\ead[label=e1]{ba.ismaila@courrier.uqam.ca}}
\and
\author[B,C]{\fnms{Jean-Fran\c cois} \snm{Coeurjolly}\thanksref{t1}}
\ead[label=e2]{jean-francois.coeurjolly@univ-grenoble-alpes.fr}
\ead[label=e3]{coeurjolly.jean-francois@uqam.ca}


\address[A]{Department of Mathematics,
Universit\'e du Qu\'ebec \`a Montr\'eal (UQAM), \\ Canada,
\printead{e1}}

\address[B]{Laboratory Jean Kuntzmann, Department DATA,
Université Grenoble Alpes, \\ France,
\printead{e2}}

\address[C]{Department of Mathematics,
Universit\'e du Qu\'ebec \`a Montr\'eal (UQAM), \\ Canada,
\printead{e3}}

\thankstext{t1}{Natural Sciences and Engineering Research Council of Canada}
\runauthor{Ba and Coeurjolly}

\end{aug}

\begin{abstract}
\; Gibbs point processes (GPPs) constitute a large and flexible class of spatial point processes with explicit dependence between the points. They can model attractive as well as repulsive point patterns. Feature selection procedures are an important topic in high-dimensional statistical modeling. In this paper, composite likelihood approach regularized with convex and non-convex penalty functions is proposed to handle statistical inference for possibly high-dimensional inhomogeneous GPPs. The composite likelihood incorporates both the pseudo-likelihood and the logistic composite likelihood. We particularly investigate the setting where the number of covariates diverges as the domain of observation increases. Under some conditions provided on the spatial GPP and on the penalty functions, we show that the oracle property, the consistency and the asymptotic normality hold. Our results also cover the low-dimensional case which fills a large gap in the literature. Through simulation experiments, we validate our theoretical results and finally, an application to a tropical forestry dataset illustrates the use of the proposed approach.
\end{abstract}





\end{frontmatter}

\section{Introduction}
Spatial point patterns are datasets containing the random locations of events or objects within a spatial domain. These datasets arise in a broad range of applications, for instance in modeling locations of trees in a forest, locations of disease cases in a region, etc \cite[see e.g.][]{moller2003statistical,illian2008statistical,baddeley2015spatial}. The stochastic models generating such datasets are called spatial point processes, with the spatial Poisson point process as the reference model. This model is the natural candidate for modeling independent random structures, i.e. independent locations of points in space with no interaction. In practice, we may observe random patterns where the points interact between them. There exists several models which are suitable for modeling these dependent random structures. Gibbs point processes (GPPs) are for instance one of them. 
In a nutshell, GPPs are defined, in a bounded domain, by a density with respect to the Poisson process. 
This makes this class particularly large, flexible and attractive. It can model homogeneous or inhomogeneous, clustered  or regular  patterns~\citep{jensen:stougaard:2001,dereudre2017consistency,dereudre2019introduction}.

One way of characterizing spatial point processes is through intensity functions. The $k$th order intensity function at $u_1,\dots,u_k$ can be interpreted as the local probability to observe a point at pairwise distinct points $u_1,\dots,u_k$ \citep{coeurjolly2017tutorial}. Given a Gibbs model, such intensity functions are (in general) not available in a closed form, essentially due to the normalizing constant involved in its density expression (see Section~\ref{sec:background} for more details). Thus, even for simple models, the expected number of points in a bounded domain has no explicit expression.

A reasonable way to overcome this issue is to define a quantity which does not involve the normalizing constant and which characterizes in some ways the GPP. A suitable candidate which presents these two features is the Papangelou conditional intensity \citep{dereudre2019introduction}, denoted here by $\lambda$. For a GPP $\bX$ and a location $u \in \mathbb{R}^d$, $\lambda(u,\bX) \mathrm{d}u$ can be interpreted as the conditional probability to observe a point in a ball with volume $\mathrm{d}u$ around $u$ given the rest of the configuration agrees with $\bX$ \citep{coeurjolly2017tutorial}. From a practical point of view,  we may suspect that the Papangelou conditional intensity depends on spatial covariates, which can be environmental conditions, topological attributes, soil characteristics, etc. In the present study, we assume that the Papangelou conditional intensity is a loglinear form of the parameters  \cite[see][]{jensen:stougaard:2001,daniel2018penalized}:
\begin{align}
\label{intensity function}
\lambda_{\btheta}(u,\bX)&=\exp(\bbeta^\top \bZ(u) + \bpsi^\top \bS(u,\bX)), \quad u \in W \subseteq \mathbb{R}^d
\end{align}
where $\btheta=(\bpsi^\top, \bbeta^\top)^\top \in \mathbb{R}^{p} $ is a parameter vector to be estimated, $d$ represents the state space of the spatial GPP $\bX$ (usually $d=2,3$) and $W$ is the observation window. We assume that $\bpsi$ is a real $l$-dimensional parameter and $\bbeta$ a real $q$-dimensional parameter, so that $l+q=p$. In this setting, the $q$ spatial covariates measured at coordinate $u$, $\bZ(u) = \{ z_1(u),\cdots,z_{q}(u) \}^\top$, describe the spatial inhomogeneity and the covariates effects; and $\bS(u,\bX) = \{ s_1(u,\bX),\cdots,s_l(u,\bX) \}^\top$ correspond to interaction terms (see Section~\ref{sec:examples} for specific examples). 

Likelihood inference is a standard method in parametric estimation. Due to the normalizing constant defining the density of a GPP, the likelihood function is intractable, which makes difficult, from a computational point of view, to fit Gibbs models via maximum likelihood. 
From a theoretical point of view, similar complexities appear. Even for simple stationary models, \citet{dereudre2017consistency} show that very little is known. Consistency of the maximum likelihood can be established only for some stationary models and no central limit theorem is available. 

Alternatives to the likelihood method include the pseudo-likelihood \citep{besag1978some,jensen1991pseudolikelihood,jensen1994asymptotic} and the logistic composite likelihood~\citep{baddeley2014logistic}, two methods which can be implemented very efficiently~\citep{baddeley2015spatial} when the number of parameters is moderate, i.e. when $p$ is small. Asymptotic results for parameters estimates from these methods are also well-known in the homogeneous case, that is for models which cannot include spatial covariates (see Section~\ref{sec:pseudo-likelihood}). No asymptotic result is available when $q \geq 1$ i.e. when the model is inhomogeneous.

With the large number of covariates that spatial pattern data may contain, an important question concerns the spatial features to include in the final model for estimating the parameters. It then becomes inevitable to develop a technique that can perform covariate selection and parameter estimation simultaneously which can be done in adding a penalty function to the objective function i.e., the pseudo-likelihood or the logistic composite likelihood in our setting. Regularization technique is a recent topic in the context of spatial point processes \cite[see e.g.][]{thurman2014variable,thurman2015regularized,yue2015variable,choiruddin2018convex} and has not much been investigated for spatial GPPs in particular. \cite{rajala2018detecting} and  \cite{daniel2018penalized} consider this problem. However, only convex penalties are considered and no theoretical result is provided to guarantee these procedures. 

The aim of this paper is to address regularization method for inhomogeneous GPPs via penalized composite likelihood (and in particular for the pseudo-likelihood) in a more complete fashion.
We provide conditions on the inhomogeneous Gibbs model (assumptions on covariates, form of the interaction terms, etc) on the penalty function to obtain sparsity, consistency and asymptotic normality for the regularized pseudo-likelihood estimator. The results are established in the increasing domain setting.
When $p$ is small and no regularization is used, consistency and asymptotic normality also hold, which fills an important gap in the literature see Theorem~\ref{thm:the4} in Section~\ref{sec:result}. \cite{billiot2008maximum,dereudre2009campbell,baddeley2014logistic,dereudre2017consistency} constitute the main references in the unregularized setting (see also  Section~\ref{subsec2} for more details) but all results from these references  are restricted to stationary/homogeneous Gibbs models. 
The present paper can be seen as an extension of all the previous references in the stationary case, of~\cite{choiruddin2018convex} to spatial GPPs, that is to the estimation of conditional intensities, of~\cite{rajala2018detecting} and \cite{daniel2018penalized} to non-convex penalties. Our results are in the same vein as the ones obtained by~\citet{fan2001variable} and~\citet{fan2004nonconcave}. We consider an increasing domain asymptotic framework~\cite[e.g.][]{book:981816} and we also consider the setting when the number of parameters, and in particular the number of spatial covariates grows with the sample size, which is in our situation the size of the observation domain.

GPPs are complex dependent processes. For instance even for very simple examples, we do not know if these processes are $\alpha$-mixing or not. Concentration inequalities for functional of GPP are therefore not deeply studied in the literature: \citet{reynaud,picard} propose such inequalities
 for inhomogeneous Poisson point processes while \citet{coeurjolly2015almost} considers stationary Gibbs models. Deriving concentration inequalities for inhomogeneous Gibbs point processes is, to our opinion, an open challenging probabilistic question. Since such inequalities are the starting point to derive finite-sample properties for the regularized pseudo-likelihood estimator, we have decided to focus in this paper on asymptotic properties instead of finite-sample ones.

The remainder of the paper is structured as follows. In Section~\ref{sec:gpps}, we present a brief formalism to define and characterize spatial GPPs. We also present a few examples. Methodologies to infer parametric GPPs are presented in Section~\ref{subsec2} where we, in particular, review unregularized and regularized pseudo-likelihood. Asymptotic properties are given in Section~\ref{sec:asy}. Section~\ref{sec:num} details the methodology while providing its numerical aspects and Section~\ref{sec:sim} presents a simulation study while Section~\ref{sec:data} is devoted to the application of the proposed method to a tropical forestry dataset. Discussion and conclusion follow in Section~\ref{sec:discussion}. Finally, proofs are postponed to Appendices~\ref{sec:auxLemma}-\ref{proof3}.

\section{Gibbs point processes (GPPs)}
\label{sec:gpps}

\subsection{Background and definitions} \label{sec:background}

We consider spatial point processes $\bX$ on $S \subseteq \mathbb R^d$ in this paper. We view $\bX$ as a locally finite subset of $S$. Thus two points cannot occur at the same location. We denote by $\bX_B= \bX\cap B$ the restriction of $\bX$ to a set $B \subseteq S$ and by $|B|$ the volume of a bounded Borel set of $\mathbb R^d$. Local finiteness of $\bX$ means that $\bX_B$ is finite almost surely, that is the number of points $N(B)=|\bX_B|$ is finite almost surely.
Let ${N}_{lf}$  be the space of locally finite configurations of  $\mathbb{R}^{d}$, that is
\[
{N}_{lf}=\{ \bx, |\bx_B| < \infty \; \mbox{for any bounded domain} \; B \subset \mathbb{R}^{d} \}.
\]

We now briefly remind the definition of GPPs. For a comprehensive presentation of GPPs as well as a detailed list of references, we refer to~\citet{dereudre2019introduction}. In this paper, we are interested in increasing domain asymptotic properties for some statistical inference to be detailed in Section~\ref{subsec2}. Thus, the point process must be defined on $\mathbb R^d$ and we assume to observe it on a sequence of observation domains that grow to $\mathbb R^d$. It is therefore important to properly define GPPs on $S=\mathbb R^d$. 

However, let us first consider the case $|S|<\infty$, as it is easier to understand and interpret. 
GPPs are characterized by an energy function $H$ (or Hamiltonian) that maps any finite point configuration to $\mathbb R \cup \{\infty\}$. Specifically, if $|S|<\infty$, a GPP on $S$ associated to $H$ and with activity $z>0$ admits the following density with respect to the unit rate Poisson process:
\begin{equation}\label{densityGibbs}
 f(\bx)  \propto  z^{|\bx|} e^{-H(\bx)}
\end{equation} 
where $\propto$ means ``proportional to''. By this, we highlight that, even for simple models, the constant is intractable. This definition makes sense under some regularity conditions on $H$, typically non degeneracy ($H(\emptyset)<\infty$) and stability (there exists $A\in\mathbb R$ such that $H(\bx)\geq A |\bx|$ for any $\bx \in N_{lf}$). Consequently, configurations $\bx$ having a small energy $H(\bx)$ are more likely to be generated from a  GPP than from a Poisson point process, and conversely for configurations having  a high energy. In the extreme case where $H(\bx)=\infty$, then $\bx$ cannot, almost surely, be the realization of a GPP associated to $H$.

Suppose that $f$ is hereditary, that is, $f(\bx) > 0 \Rightarrow f(\by) > 0$ for $\by \subset \bx$. In other words, the hereditary property means that an authorized configuration $\bx$ (in the sense that $f(\bx) > 0$) remains allowed if we take away one point. Then, one can define the Papangelou conditional intensity $\lambda(u,\bx)$ at any location $u \in S$ as follows:
\begin{equation}
\label{condint}
\lambda(u,\bx)=
\left\lbrace
\begin{array}{ccc}
f(\bx \cup u)/f(\bx)  & \mbox{for } & u \notin \bx \\
f(\bx)/f(\bx \setminus u) & \mbox{for} & u \in \bx 
\end{array}\right.
\end{equation}
with $a/0:=0$ for $a \geq 0$. Heuristically, the quantity $\lambda(u,\bx) \mathrm{d}u$ may be interpreted as the probability that has the process $\bX$ to insert a point in a region  $\mathrm{d}u$ around $u$ given the rest outside this infinitesimal set is $\bx$.

When $|S|=\infty$, the above definition \eqref{densityGibbs} of the density does not make sense in general since $H(\bx)$ can be infinite or even undefined if $|\bx|=\infty$. In this case a GPP is defined through its local specifications, which are the conditional densities on any bounded set $\Delta$, given the outside configuration on $\Delta^c$, with respect to the unit rate Poisson process on $\Delta$. These conditional densities take a similar form as in \eqref{densityGibbs}, where now the Hamiltonian $H$ becomes a family of Hamiltonian functions $H_\Delta$ that quantify the energy of $\bx_{\Delta}$ given the outside configuration $\bx_{\Delta^c}$. 
The interpretation nonetheless remains similar: a (infinite) GPP associated to $H_\Delta$ tends to favor configurations $\bx_\Delta$ on $\Delta$ having a small value $H_\Delta(\bx)$.
The interest of the Papangelou conditional intensity concept is that it still makes sense when $S=\mathbb{R}^{d}$. For more details, we refer the reader to \citet{daley2007introduction} and \citet{moller2003statistical}. 

Campbell theorem is a fundamental tool for general point processes which in particular is used to define intensity functions. The analog for GPPs is called the GNZ (for Georgii, Nguyen and Zessin) formula~\citep{xanh1979integral,georgii1979canonical,georgii2011gibbs}. GNZ equation is a way of characterizing the GPP and also highlights the interest of the Papangelou conditional intensity. It states that, for any measurable function $h : \mathbb R^d \times {N}_{lf} \to \mathbb{R}^{+}$
\begin{equation}
\label{gnz}
\EE \sum_{u \in \bX} h(u, \bX \setminus \{u\})  =  
 \EE \int    h(u ,\bX) \lambda(u, \bX)  \mathrm{d}u.
\end{equation}
Iterated versions GNZ formula are available (see e.g. \citet{dereudre2019introduction}).

We end this section with two definitions related to the Papangelou conditional intensity which often correspond to important expected properties. We say that a GPP has a finite range (FR) for some $R<\infty$, if for any $u\in S$, $\bx \in N_{lf}$
\begin{equation}
\label{fr}
 \lambda(u,\bx)=\lambda(u,\bx \cap B(u,R)) \tag{FR}.
\end{equation}
In other words, the probability to insert a point $u$ in $\bx$ only depends on the $R$-neighbors of $u$ to $\bx$. Finally, a GPP is said to be locally stable if the GPP is stochastically dominated by a Poisson point process, that is if there exists $\bar \lambda<\infty$ such that for any $u\in S$ and $\bx \in N_{lf}$
\begin{equation}
\label{ls}
 \lambda(u,\bx) \leq \bar \lambda \tag{LS}.
\end{equation}

\subsection{Examples of homogeneous models} \label{sec:examples}

As specified earlier, the interaction between points is encoded, when $|S|<\infty$, in the density or when $|S|\le \infty$ in the Papangelou conditional intensity. Parametric models can therefore be easily defined. We consider first parametric models with density of the form
\begin{equation}
	\label{eq:densityhomogtheta}
	f_{\boldsymbol \theta}(\bx) \propto \exp\left( \beta |\bx| + \boldsymbol\psi^\top \bS(\bx) \right)
\end{equation}
where $\btheta=(\boldsymbol\psi^\top,\beta)^\top$, $\beta\in \mathbb R$, $\boldsymbol\psi \in \mathbb R^l$ and $\bS: N_{lf}^l\to\mathbb R$ is a vector of interaction terms. The Papangelou conditional intensity writes
\begin{equation}
	\label{eq:papangelouhomogtheta}
\lambda_{\btheta}(u,\bx) = \exp\left(\beta + \boldsymbol\psi^\top \bS(u,\bx) \right)
\end{equation}
with $\bS(u,\bx)=\bS(\bx \cup u) -\bS(\bx)$.
Note that when $\boldsymbol\psi=0$, the model reduces to a homogeneous Poisson point process with intensity $\exp(\beta)$.
The models below, which extend to $\mathbb R^d$, are stationary, i.e. the distribution of $\bX$ is invariant under translations. In particular,  using the GNZ formula, it means that $\EE N(A) = \int_A \EE(\lambda_{\btheta}(u,\bx)) \mathrm d u = |A| \EE \lambda_{\btheta}(0,\bX)$. The average number of points in a bounded domain with volume $|A|$ does not depend on the location of this set $A$ (nor its shape). 

\begin{itemize}
\item Strauss model: this pairwise interaction model is defined by $l=1$, $\gamma=\exp \psi \in [0,1]$ and $s(\bx)=\sum_{u,v\in \bx}^{\neq} \mathbf 1(\|v-u\|\le R)$ for some $R<\infty$ or using the Papangelou conditional intensity $s(u,\bx)= \sum_{v\in \bx} \mathbf 1(\|v-u\|\le R)$ which represents the number of $R$-neighbors of $u$ in $\bx$.
When $\gamma=0$ the model is well-defined and corresponds to the Hard-core model with hard-core distance $R$ (two points at distance smaller than $R$ are forbidden). Strauss model is exclusively able to model repulsive patterns.
\item Geyer saturation model: it generalizes the Strauss model in the sense that it is defined for any value of the interaction parameter $\psi$ and  can model both clustering or inhibition. It is defined by $l=1$, $\psi\in \mathbb R$,  
$s(u,\bx)=s(\bx\cup u)-s(\bx)$ for any $u\in \mathbb R^d$ and $\bx\in N_{lf}$ with
$s(\bx) = \sum_{ u \in \bx} \mbox{min} \left (\sigma,\tau(u,R,\bx \setminus u) \right ) $
where $\sigma$ is a threshold parameter, $\tau(u,R,\bx \setminus u):= \sum_{ v \in \bx}^{u \neq v} \mathbf{1}(\Vert u - v \Vert \leq R)$ is the number of other points $v$ of $\bx$ lying within a distance $R<\infty$ of the point $u$. 
\end{itemize}
It is to be noticed that both models satisfy \eqref{fr} (with respective finite range parameters $R$ and 
$2R$) and~\eqref{ls}. Many other examples are available in the literature (see e.g. \citet{moller2003statistical,illian2008statistical}): piecewise Strauss models, pairwise interaction models with infinite range such as the Lennard-Jones model, higher order interaction models such as the area interaction model, etc.

\subsection{Examples and existence of inhomogeneous Gibbs models}

The originality of this paper is to focus on inhomogeneous Gibbs models. There are several ways of introducing inhomogeneity, non stationarity, anisotropy, etc (see \citet{jensen:stougaard:2001} for a review).
We consider, here, inhomogeneous models that can be explained using spatial covariates. Namely, we focus on models with Papangelou condition intensities of the loglinear form
\begin{equation}
\label{intensity function}
\lambda_{\btheta}(u,\bx) =\exp \Big \{ \bbeta^\top \bZ(u) + \bpsi^\top \bS(u,\bx) \Big \}
\end{equation}
for $u\in \mathbb R^d$ and $\bx \in N_{lf}$, where now $\btheta=(\bpsi^\top, \bbeta^\top)^\top \in \mathbb{R}^{p} $. We still assume that $\bpsi$ is a real $l$-dimensional parameter and $\bbeta$ a real $q$-dimensional parameter, so that $l+q=p$. In this setting, the $q$ spatial covariates measured at coordinate $u$, $\bZ(u) = \{ z_1(u),\cdots,z_{q}(u) \}^\top$, describe the spatial inhomogeneity, covariates effects and $\bS(u,\bx) = \{ s_1(u,\bx),\cdots,s_l(u,\bx) \}^\top$ remain interaction terms. Inhomogeneous Strauss  or Geyer models for instance can thus be straightforwardly proposed. From the previous section, if the covariates are assumed to be bounded (more formalized later by condition ($\mathcal C$.\ref{C:cov})), then inhomogeneous versions of Strauss and Geyer models for instance also satisfy \eqref{fr}-\eqref{ls}. 
Figure~\ref{fig:models} depicts simulated realizations of homogeneous and inhomogeneous  Strauss and Geyer saturation models. 

\begin{figure}[!ht]
\begin{center}
\setlength{\tabcolsep}{0pt}
\renewcommand{\arraystretch}{0}
\hspace*{-1cm}\begin{tabular}{c c}
\includegraphics[scale=.45]{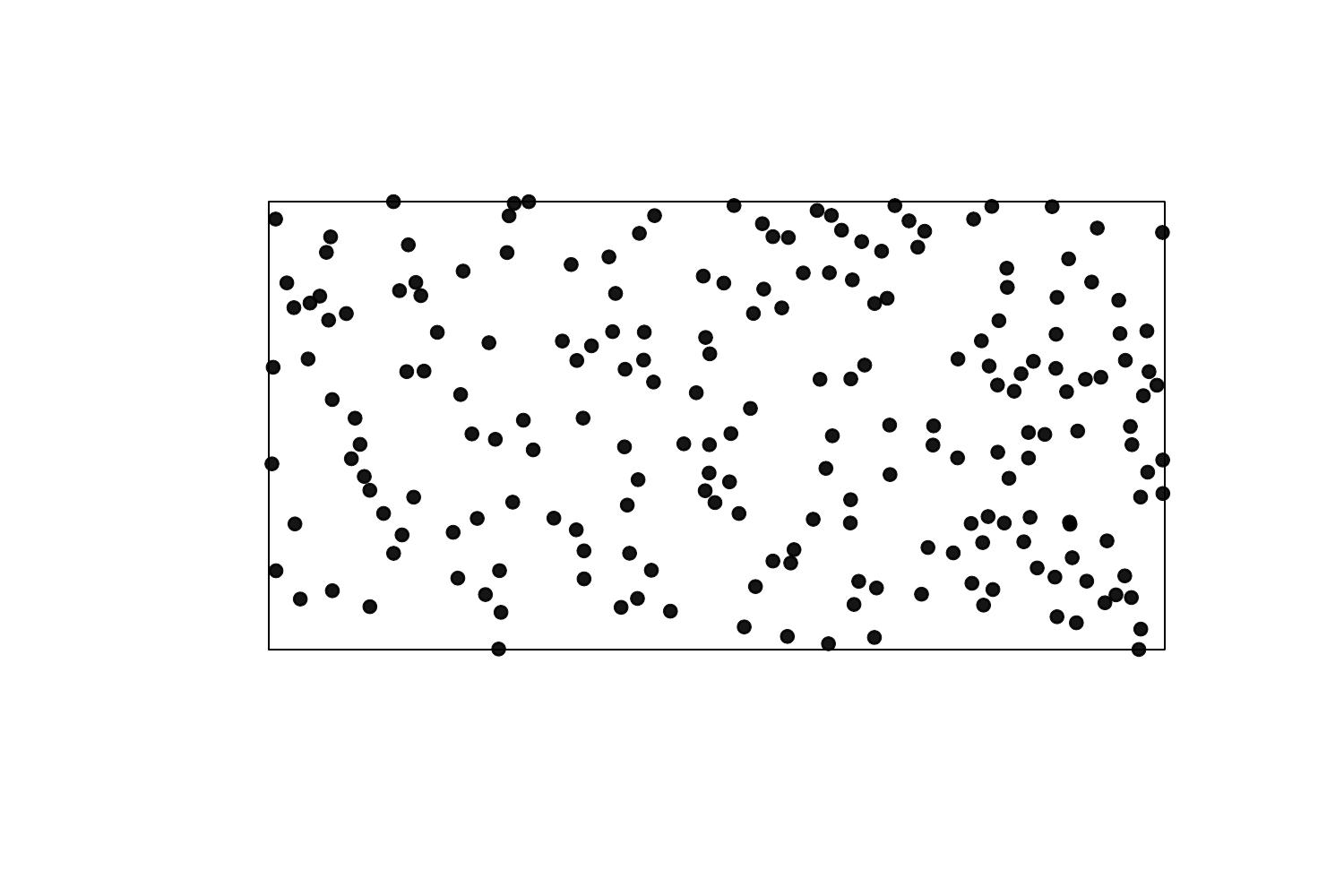} & \includegraphics[scale=.45]{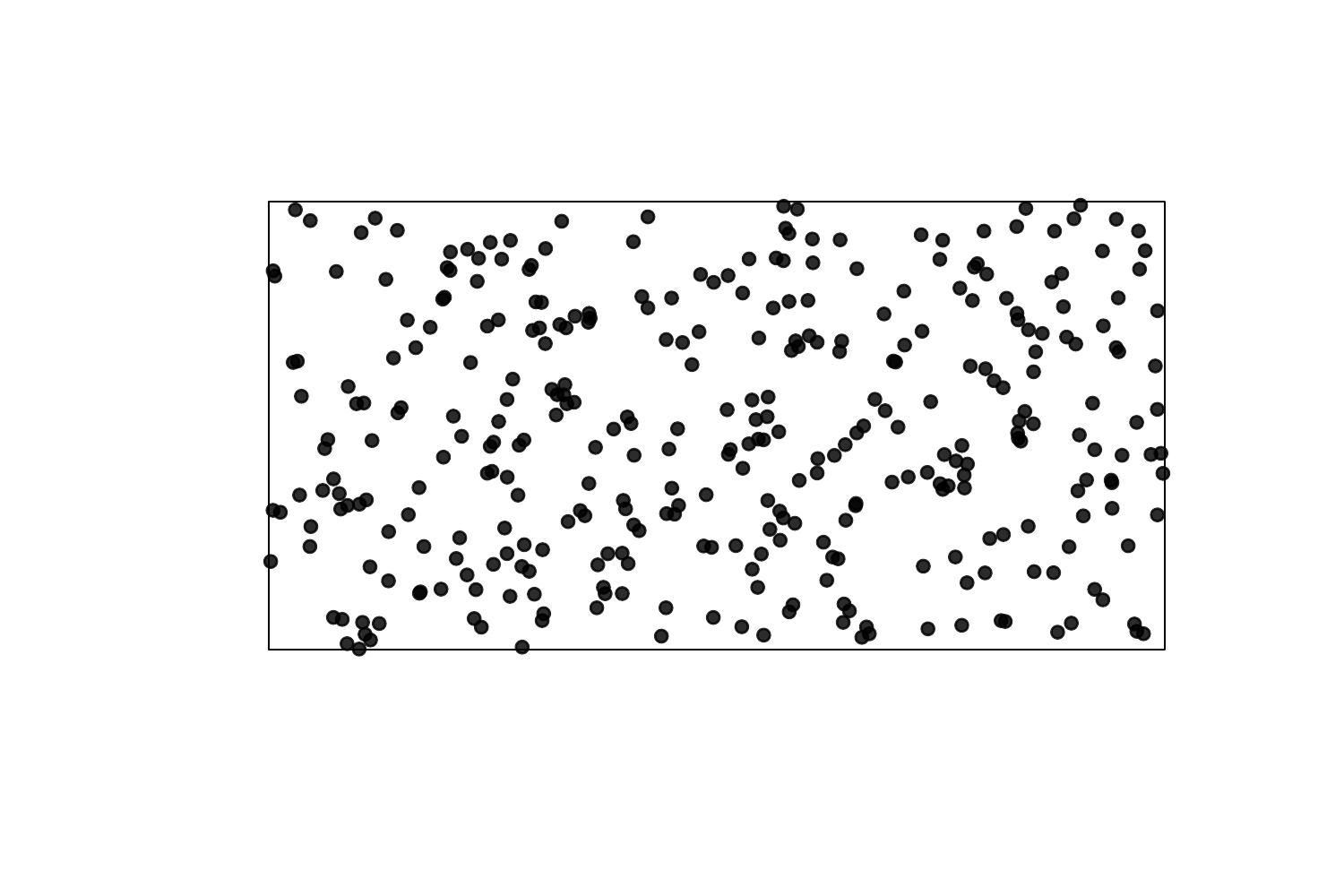}\\[-.9cm]
\includegraphics[scale=.45]{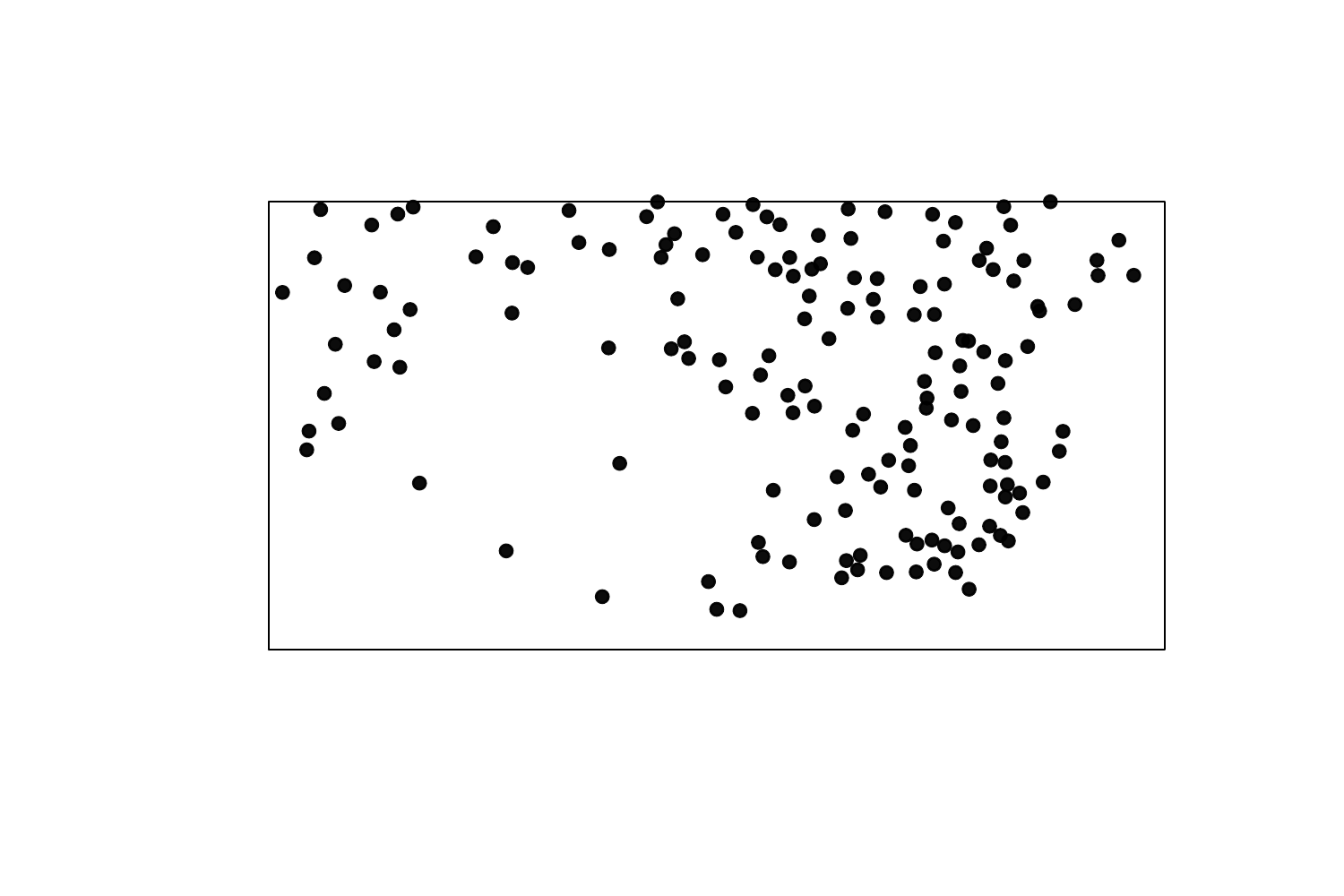} & \includegraphics[scale=.45]{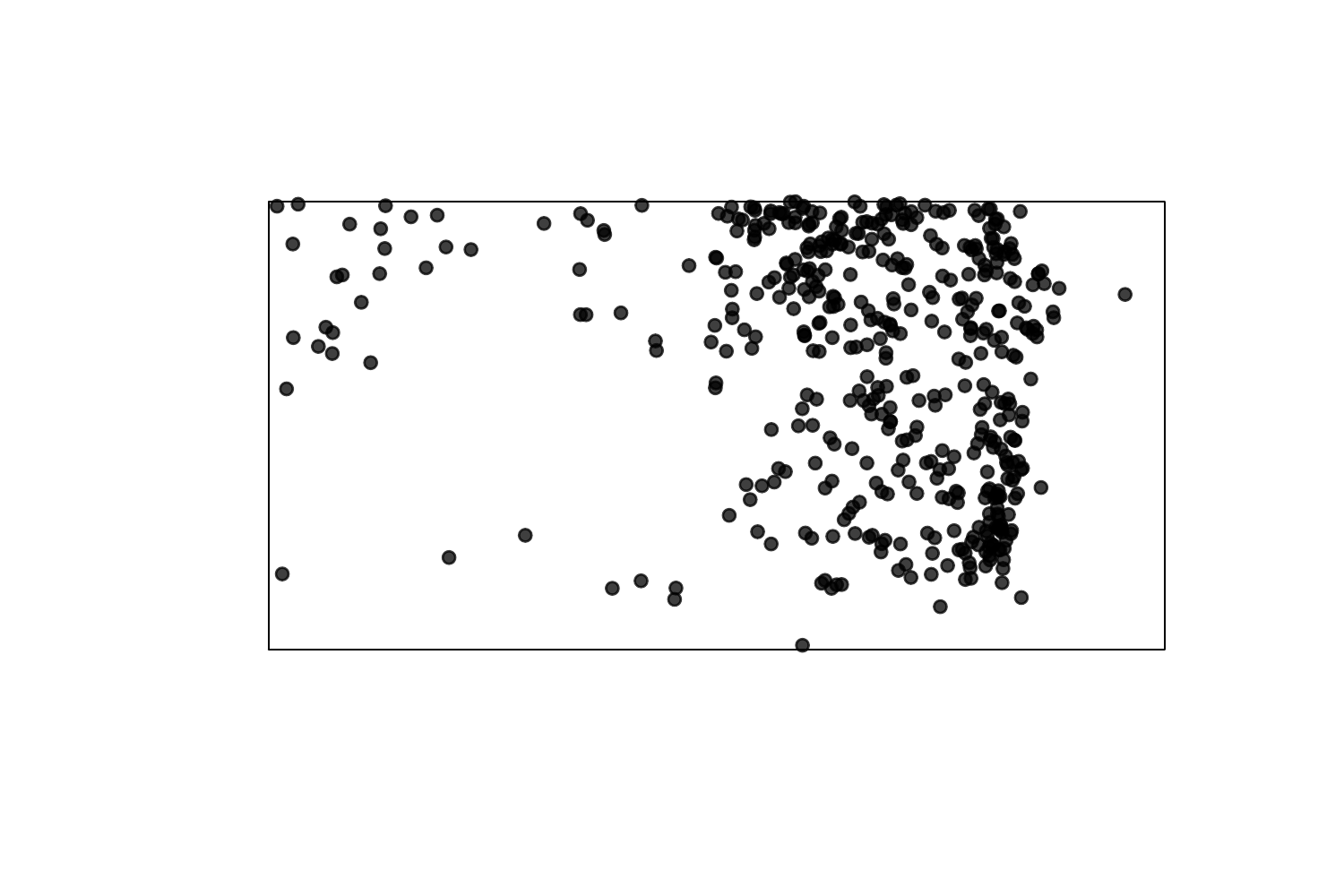}
\end{tabular}
\caption{Realizations of Strauss process with parameters $\gamma=0.1$ and $R=15$ (column 1), 
Geyer saturation process with $\sigma=1$, $\gamma=1.5$ and $R=15$ (column 2),  in the spatial domain $W=[0,1000] \times [0,500]$. Homogeneous models with parameter $\beta_0$ are considered in the top row while bottom row is concerned with inhomogeneous models with $\bbeta^top=(\beta_0,2,.75)$. The two spatial covariates correspond to elevation and slope of elevation (see Section~\ref{sec:sim}); $\beta_0$ is set such that the average number of points under the Poisson case (i.e. when $\gamma=1$) in this domain equals $400$.}
\label{fig:models}
\end{center}
\end{figure}

As mentioned in Section~\ref{sec:background}, our statistical objective raises the probabilistic question: given a function $\lambda:\mathbb R^d\times N_{lf}\to \mathbb R^+$ (eventually parametric), does there exist a Gibbs model with Papangelou conditional intensity $\lambda$? Such a question has generated an important literature in the homogeneous case. The most recent result can be found in~\citet{dereudre2012existence} and proves the existence for many interaction models including homogeneous Strauss  and Geyer models. Surprisingly, this fundamental probabilistic problem has never been considered in the literature in the non-stationary case. A recent draft by~\citet{vasseur:dereudre:20} establishes the following useful result, which provides sufficient conditions to ensure the existence.

 \begin{theorem}{\citet[Theorem~1]{vasseur:dereudre:20}}
\label{existence}
Let $\lambda : \mathbb{R}^{d} \times {N}_{lf} \to  \mathbb{R}^{+}$, assume the assumptions \eqref{fr}-\eqref{ls} hold, then there exists at least one infinite volume Gibbs measure, that is there exists at least one Gibbs model $\bX$, with Papangelou conditional intensity $\lambda$ which in particular satisfies the GNZ equation~\eqref{gnz}.
\end{theorem}

In the rest of the paper, we assume that $\bX$ is a GPP with parametric Papangelou conditional intensity function~\eqref{intensity function} satisfying~\eqref{fr}-\eqref{ls}. By Theorem~\ref{existence}, the model is well-defined and the notation $\mathbb E$ or $\Var$ are expectation and variance with respect to this Gibbs measure.

\section{Parametric estimation of the Papangelou conditional intensity}
\label{subsec2}

In this section, we present methodologies to estimate parametric Gibbs models given by a Papangelou conditional intensity $\lambda_{\btheta}(\cdot,\cdot)$.
We present the pseudo-likelihood, known as the most standard alternative  to the maximum likelihood method. Then, we present regularized or penalized versions which are able to handle high-dimensional problems, that is situations when $p$ is large. We assume to observe a single realization of a point process $\bX$ defined on $\mathbb R^d$ and observed on a bounded domain $W$.

\subsection{Pseudo-likelihood} \label{sec:pseudo-likelihood}

The maximum likelihood method is rarely used in the literature. The density is known upto a constant which is not explicit and which must be estimated at each step of the optimization procedure. This makes such a method very computationally expensive even if $p$ is very small.

The standard alternative is to maximize the pseudo-likelihood function~\citep{besag1978some,jensen1991pseudolikelihood}, which is a ``modified" version of the likelihood function and does not involve the normalizing constant. We may encounter boundary effects problem when computing the conditional intensity at some points. Indeed, $\lambda_{\btheta}(u,\bx)$  may depend on unobserved points of $\bx$ lying outside the observation window $W$ for points $u$ close to the edge of $W$. In such a situation, we need some edge corrections. To handle this problem, 
we take advantage of the finite range property~\eqref{fr} and
use the minus sampling $D=W\ominus R$, i.e. the observation window $W$ eroded by the finite interaction range $R$ of the GPP. Now, for any $u \in D \cup (\bX\cap D)$, $\lambda_{\btheta}(u,\bx)$ can be observed without error using the observation of $\bX$ on $W$.
For a Gibbs model with Papangelou conditional intensity $\lambda_{\btheta}(\cdot,\cdot)$, the log-pseudo-likelihood function in $D$ is defined by (see \citet{jensen1991pseudolikelihood})
\begin{align}
\mbox{LPL}( \bX;\btheta) &=  {\sum_{u \in \bX \cap D} \log\lambda_{\btheta}(u, \bX)} - {\int_D \lambda_{\btheta}(u,\bX)\mathrm{d}u}. \label{ch2:Pois}    
\end{align}
Other edge effect corrections are known in the literature (e.g. isotropic correction, translation correction, etc.) but are not considered in this paper for the sake of simplicity~\cite[see e.g.][]{ripley1991statistical}. \\ 
The first and second derivatives of the log-pseudo-likelihood function, for models given by~\eqref{intensity function}, are respectively
\begin{align*}
\mathbf{LPL}^{(1)}( \bX;\btheta) = & {\sum_{u \in \bX \cap D} \mathbf{t}(u,\bX \setminus u)} - {\int_D \bt(u,\bX)\lambda_{\btheta}(u,\bX)\mathrm{d}u}, \\
\mathbf{LPL}^{(2)}( \bX;\btheta) = & - {\int_D \bt(u,\bX)\bt(u,\bX)^\top \lambda_{\btheta}(u,\bX)\mathrm{d}u}
\end{align*}
where $ \bt(u,\bx)=(\bS(u,\bx)^\top,\bZ(u)^\top)^\top $ for $u \in \mathbb{R}^d$ and $\bx \in \boldsymbol{N}_{lf}$.
The score of the log-pseudo-likelihood function (\ref{ch2:Pois}), $\mathbf{LPL}^{(1)}( \bX;\btheta)$, is an unbiased estimating equation (this ensues from GNZ equation, see e.g. (\ref{gnz})). 
The estimator obtained from the maximization of (\ref{ch2:Pois}) is called the maximum pseudo-likelihood estimator. Well-desired asymptotic properties of this estimator have been shown  theoretically for stationary Gibbs models. In the case of stationary exponential models, \citet{jensen1991pseudolikelihood,mase1995consistency,jensen1994asymptotic,billiot2008maximum} show the consistency and asymptotic normality of the maximum pseudo-likelihood estimator. \citet{dereudre2009campbell} extend the results to non-hereditary GPPs while \citet{coeurjolly2010asymptotic,coeurjolly2017parametric} to finite range stationary non-exponential models and infinite range stationary exponential models. 

A particular case of our main result developed in Section~\ref{sec:asy} is that we obtain asymptotic results for inhomogeneous Gibbs models.

\subsection{Regularization via penalized pseudo-likelihood}
\label{sec4} 

In an attempt to simultaneously select and estimate the components of the parameter vector $\btheta=(\bpsi^\top, \bbeta^\top)^\top$, we regularize the log-pseudo-likelihood function defined in (\ref{ch2:Pois}). That is, we maximize the penalized log-pseudo-likelihood function
\begin{align}
\label{eq:qwee}
	\mbox{Q}(\bX;\btheta)=\mbox{LPL}( \bX;\btheta) - |D|{\sum_{j=1}^{p} p_{\lambda_{j}}(|\theta_{j}|)},
\end{align}
where $|D|$ is the volume of the eroded observation domain, $\lambda_{j}$ is a nonnegative tuning parameter corresponding to $\theta_j$ for $j=1,\ldots,p$ and $p_{\lambda_j}$ is a penalty function which we now describe. For any $\lambda \geq 0$, we say that $p_\lambda(\cdot): \mathbb{R}^+ \to \mathbb{R}$ is a penalty function if $p_\lambda$ is a nonnegative function with $p_\lambda(0)=0$ and $p_{0}(.)=0$. In our study, we do not make penalizations on the parameter $\psi$, i.e. we set $\lambda_1=\cdots=\lambda_l=0$. Therefore, (\ref{eq:qwee}) becomes
\begin{align}
	\mbox{Q}(\bX;\btheta)=\mbox{LPL}( \bX;\btheta) - |D|{\sum_{j=l+1}^{p} p_{\lambda_{j}}(|\theta_{j}|)}. \nonumber
\end{align}
In the present study, we consider convex and non-convex penalty functions. The $\ell_1$ and $\ell_2$ norms are convex penalty functions while the Smoothly Clipped Absolute Deviation (SCAD)~\citep{fan2001variable} and the Minimax Concave (MC+)~\citep{zhang2010nearly} penalty functions are non-convex. Note that ``+" in MC+ means penalized linear unbiased selection (PLUS). From the two convex functions, other penalty functions can be constructed. For example, the Elastic net penalty function, which is a convex combination between the $\ell_1$ and $\ell_2$ norms. Below are the expressions of these penalty functions:
\begin{itemize}
\item $\ell_1$ norm:  $p_{\lambda}(\theta)= \lambda \theta$,
\item $\ell_2$ norm:  $p_{\lambda}(\theta)=\frac12 \lambda \theta^2$,
\item Elastic net: for $0 < \gamma < 1$, $p_{\lambda}(\theta)=\lambda \left \{ \gamma \theta + \frac12 (1-\gamma) \theta^2 \right \}$, 
\item SCAD: for any $\gamma > 2$, $p_{\lambda}(\theta)= \left\lbrace
\begin{array}{ccc}
\lambda \theta  & \mbox{if} & \theta \leq \lambda \\
\frac{\gamma \lambda \theta - \frac12(\theta^2+\lambda^2)}{\gamma - 1} & \mbox{if} & \lambda \leq \theta \leq \gamma \lambda \\
\frac{\lambda^2(\gamma^2-1)}{2(\gamma-1)} & \mbox{if} & \theta \geq \gamma \lambda,
\end{array}\right.$
\item MC+: for any $\gamma > 1$, $p_{\lambda}(\theta)= \left\lbrace
\begin{array}{ccc}
\lambda \theta - \frac{\theta^2}{2 \gamma} & \mbox{if} & \theta \leq \gamma \lambda \\
\frac12 \gamma \lambda^2& \mbox{if} & \lambda \leq \theta \leq \gamma \lambda.
\end{array}\right.$
\end{itemize}
These penalty functions give rise to well-known regularization methods. For example, the Lasso method ~\citep{tibshirani1996regression} is given by the $\ell_1$ norm penalty function. Furthermore, an adaptive version of these penalty functions might be useful in situations where we allow each direction to have a different regularization parameter. For more details on these adaptive techniques, we refer the reader to~\citet{zou2006adaptive} and~\citet{zou2009adaptive} for convex penalty functions.

\section{Asymptotic properties}
\label{sec:asy} 

Asymptotic properties of the regularized pseudo-likelihood estimator are presented in this section. We consider a sequence of observation windows $W=W_n$, $n=1,2,...$ which expands to $\mathbb{R}^{d}$ and define the sequence $D_n:=W_n \ominus R$, $n=1,2,...$ where $R$ is the interaction range of the inhomogeneous GPP. The parameter $\btheta=(\bpsi^\top, \bbeta^\top)^\top$ is now of dimension $p=p_n$ which may diverge as $n \to \infty$. Note that the parameter $\bpsi$ is still of dimension $l$ but the dimension of $\bbeta$ is now considered to be $q=q_n$, that is $p_n=l+q_n$.  In our setting, no selection is done on the parameter $\bpsi$. The results we propose are also valid for the unregularized pseudo-likelihood.  We denote by $\btheta_0=(\btheta_{01}^\top,\btheta_{02}^\top)^\top=(\btheta_{01}^\top,\mathbf{0}^\top)^\top$ the $p_n$-dimensional vector of true coefficients, where $\btheta_{01}\hspace{-.1cm} = \hspace{-.1cm} (\psi_{01},\cdots,\psi_{0l},\beta_{01},\cdots,\beta_{0s})^\top \hspace{-.15cm}= \hspace{-.12cm} (\theta_{01},\cdots,\theta_{0(l+s)})^\top$ is the $(l+s)$-dimensional vector of non zero coefficients and  $\btheta_{02}=(\beta_{0(s+1)},\cdots,\beta_{0q_n})^\top=(\theta_{0(l+s+1)},\cdots,\theta_{0p_n})^\top$ is the $(p_n-l-s)$-dimensional vector of zero coefficients. The number of non zero coefficents $m=l+s$ is assumed to be independent of $n$. Generally for any $\btheta \in \bTheta \subset \mathbb R^{p_n}$, we let $\btheta=(\btheta_{1}^\top,\btheta_{2}^\top)^\top$ where $\btheta_{1}$ and $\btheta_{2}$
are respectively of dimension $m$ and $p_n-m$. In this setting, we let 
\begin{align}
\label{lpln}
\mbox{LPL}_n( \bX;\btheta) &=  {\sum_{u \in \bX \cap D_n} \log\lambda_{\btheta}(u, \bX)} - {\int_{D_n} \lambda_{\btheta}(u,\bX)\mathrm{d}u},  
\end{align}
\begin{align}
\label{qn}
	\mbox{Q}_n(\bX;\btheta)=\mbox{LPL}_n( \bX;\btheta) - |D_n|{\sum_{j=l+1}^{p_n} p_{\lambda_{n,j}}(|\theta_{j}|)}
\end{align}
be respectively the log-pseudo-likelihood function and its penalized version, where now the regularization parameters $\lambda_{n,j}$, $j=l+1,\cdots,p_n$ depend on $n$.

\subsection{Notation and conditions} \label{sec:not}

We define the $p_n \times p_n$ stochastic matrices $\mathbf{A}_n(\bX;\btheta_{0})$ and $\mathbf{B}_n(\bX;\btheta_{0})$ 
by
\begin{align*}
\mathbf{A}_n(\bX;\btheta_{0})&={\int_{D_n} \bt(u,\bX)\bt(u,\bX)^\top \lambda_{\btheta_{0}}(u,\bX) \mathrm{d}u}, \\
\mathbf{B}_n(\bX;\btheta_{0})&= \hspace{-.12cm} {\int_{D_n} \hspace{-.12cm}  \int_{D_n} \hspace{-.25cm}  \bt(u,\bX)\bt(v,\bX)^\top (\lambda_{\btheta_{0}}(u,\bX)\lambda_{\btheta_{0}}(v,\bX)-\lambda_{\btheta_{0}}(\{u,v\},\bX))  \mathrm{d}v \mathrm{d}u} \\
& + {\int_{D_n} \int_{D_n} \Delta_v \bt(u,\bX) \Delta_u \bt(v,\bX)^\top \lambda_{\btheta_{0}}(\{u,v\},\bX) \mathrm{d}v \mathrm{d}u} 
\end{align*}
where the second-order Papangelou conditional intensity $\lambda_{\btheta_{0}}(\{u,v\},\bX)$ and the difference operator $\Delta_v$ are defined, for any $u,v \in \mathbb{R}^{d}$ and $\bx \in \boldsymbol{N}_{lf}$
\begin{align}
\lambda_{\btheta_{0}}(\{u,v\},\bx)&=\lambda_{\btheta_{0}}(u,\bx \cup v)\lambda_{\btheta_{0}}(v,\bx)=\lambda_{\btheta_{0}}(v,\bx \cup u)\lambda_{\btheta_{0}}(u,\bx),  \label{eq:papan2} \\
\Delta_v \bt(u,\bx)&:= \bt(u,\bx \cup v) - \bt(u,\bx).  \label{eq:diffop}
\end{align}
For a $p_n \times p_n$ stochastic matrix $\mathbf{H}_n(\bX;\btheta_{0})$, we let $\mathbf{H}_n(\btheta_{0})=\EE[\mathbf{H}_n(\bX;\btheta_{0})]$ and $\mathbf{H}_{n,11}(\bX;\btheta_{0})$ $(\mbox{resp. } \mathbf{H}_{n,11}(\btheta_{0}))$ be the $m \times m$ top-left corner of $\mathbf{H}_{n}(\bX;\btheta_{0})$ $(\mbox{resp. } \mathbf{H}_{n}(\btheta_{0}))$. The matrices $\mathbf{A}_n(\btheta_{0})$ and $\mathbf{A}_n(\btheta_{0}) + \mathbf{B}_n(\btheta_{0})$ are actually related to the sensitivity matrix and the variance of the score function, that is
\begin{align*}
\mathbf{A}_n(\btheta_{0})&= \EE \left[  -\frac{d}{d \btheta^\top} \; \mathbf{LPL}_n^{(1)}( \bX;\btheta_0)  \right],  \\
\mathbf{A}_n(\btheta_{0}) + \mathbf{B}_n(\btheta_{0})&=   \Var \left[ \mathbf{LPL}_n^{(1)}( \bX;\btheta_0)  \right].
\end{align*}
In what follows, for a squared symmetric matrix $\mathbf{M}_n$, $\nu_{\min}(\mathbf M_n)$ 
denotes  the smallest  eigenvalue of $\mathbf M_n$. Consider the following conditions ($\mathcal C$.\ref{C:Dn})-($\mathcal C$.\ref{C:plambda}) which are required to derive our asymptotic results:

\begin{enumerate}[($\mathcal C$.1)]
\item  $(D_n)_{n \geq 1}$ is an increasing sequence of convex compact sets, such that  $D_n \to \mathbb{R}^d$ as $n \to \infty$. \label{C:Dn}
\item We assume that the Papangelou conditional intensity function has the log-linear specification given by~\eqref{intensity function} where $\btheta \in \bTheta$ and $\bTheta$ is an open convex bounded set of $\mathbb R^{p_n}$, and where the statistics $s_j(u,\bx)$ $j=1,\cdots,l$ are such that for any $u \in \mathbb R^{d}$,  $\bx \in \boldsymbol{N}_{lf}$, there exists $R>0$ such that   $s_j(u,\bx)=s_j(u,\bx \cap B(u,R))$. \label{C:Theta}
\item  The covariates $\bZ$ and the interaction function $\bS$ satisfy
\[
	\sup_{n \geq 1}  \sup_{i=1,\cdots,q_n} \sup_{u \in \mathbb{R}^d} |z_{i}(u)|<\infty 
 	\quad \mbox{ and } \quad 
 	\mathbb{E}[ |s_j(u,\bX)|^{4}]<\infty
\]
for any $u \in D_n$ and $j=1,\ldots,l$. \label{C:cov}
\item  There exists $\bar{\lambda}_n: D_n \to \mathbb{R}^+$  such that  for any $u \in D_n$, $\bx \in \boldsymbol{N}_{lf}$ and $\btheta \in \bTheta$ $\lambda_{\btheta}(u,\bx) \leq \bar{\lambda}_n(u)$  with   $\sup_{n \geq 1}  \sup_{u \in \mathbb{R}^d} \bar{\lambda}_n(u)<\infty$. \label{C:locsta}
\item ${\displaystyle \liminf_{n\to \infty} \;  \nu_{\min}\big(|D_n|^{-1}\{\mathbf{A}_{n,11}(\btheta_0)+\mathbf{B}_{n,11}(\btheta_0)\} \big)>0}$. \label{C:BnCn}
\item $\exists \, \mathscr{N}(\btheta_0)$ a neighborhood of $\btheta_0$ such that ${\forall \, \bx \in \boldsymbol{N}_{lf}, \, \btheta \in  \mathscr{N}(\btheta_0)}$ 
\[
\displaystyle\liminf_{n\to\infty}\; \nu_{\min}\big(|D_n|^{-1}\mathbf{A}_{n}(\bx;\btheta)\big)> 0.
\] \label{C:An}
\item The penalty function $p_{\lambda}(.)$ is nonnegative on $\mathbb R^{+}$, satisfies $p_\lambda(0)=0$, $p_0(.)=0$, and is continuously differentiable on $\mathbb R^+ \setminus\{0\}$ with derivative ${p}_\lambda'$ assumed to be a Lipschitz function on $\mathbb R^+\setminus\{0\}$.
Furthermore, given $(\lambda_{n,j})_{n \geq 1}, \mbox{ for } j=l+1, \ldots, l+s,$ we assume that there exists $(\tilde r_{n,j})_{n \geq 1}$, where $\tilde r_{n,j}\sqrt{|D_n|/p_n} \to \infty$ as $n \to \infty$, such that, for $n$ sufficiently large, $p_{\lambda_{n,j}}$ is thrice continuously differentiable in the ball centered at $|\theta_{0j}|$ with radius $\tilde r_{n,j}$ and we assume that the third derivative is uniformly bounded. \label{C:plambda}

\end{enumerate}

Under condition ($\mathcal C$.\ref{C:plambda}) and $p_n/|D_n|\to 0$ as $n\to\infty$, we define the sequences $a_n$, $b_n$ and $c_n$  by
\begin{align}	
a_n &=\max_{j=l+1,\dots,l+s} | p'_{\lambda_{n,j}}(|\theta_{0j}|)| , \label{eq:an} \\
b_n &=\inf_{j=l+s+1,\ldots,p_n} \inf_{\substack{|\theta| \leq \epsilon_n \\ \theta \neq 0}} p'_{\lambda_{n,j}}(\theta) \label{eq:bn}, \mbox{ for } \epsilon_n=K_1\sqrt{\frac{p_n}{|D_n|}},	\\
c_n &=  \max_{j=l+1,\dots,l+s} |p^{\prime\prime}_{\lambda_{n,j}}(|\theta_{0j}|) |  \label{eq:cn}
\end{align}
where $K_1$ is any positive constant. 

\subsection{Main results} \label{sec:result}
We state our main results here. Proofs are relegated to Appendices~\ref{sec:auxLemma}-\ref{proof3}. We first show in Theorem~\ref{THM:ROOT} that the regularized pseudo-likelihood estimator converges in probability and exhibits its rate of convergence.

\begin{theorem}
\label{THM:ROOT}
Assume the conditions ($\mathcal C$.\ref{C:Dn})-($\mathcal C$.\ref{C:locsta}) and ($\mathcal C$.\ref{C:An})-($\mathcal C$.\ref{C:plambda}) hold. Let $a_n$ and $c_n$ be given by (\ref{eq:an}) and~\eqref{eq:cn}. If $a_n=O(|D_n|^{-1/2})$, $c_n=o(1)$ and $p_n=o(|D_n|)$, then there exists a local maximizer $\hat \btheta$ of $Q_n(\bX;\btheta)$  such that    ${\bf \| \hat \btheta -\btheta_0\|}=O_\mathrm{P}(\sqrt{p_n}(|D_n|^{-1/2}+a_n))$.
\end{theorem}

This implies that, if $a_n=O(|D_n|^{-1/2})$, $c_n=o(1)$ and $p_n=o(|D_n|)$, the regularized pseudo-likelihood estimator is root-$(|D_n|/p_n)$ consistent. Furthermore, we demonstrate in Theorem~\ref{THM:SPARSITYCLT} that such a root-$(|D_n|/p_n)$ consistent estimator ensures the sparsity of $\boldsymbol{\hat \theta}$; that is, the estimate will correctly set $\btheta_2$ to zero with probability tending to 1 as $n \to \infty$, and $\hat \btheta_1$ is asymptotically normal.
\begin{theorem}
\label{THM:SPARSITYCLT}
Assume the conditions ($\mathcal C$.\ref{C:Dn})-($\mathcal C$.\ref{C:plambda}) hold and set $m=l+s$. If $a_n\sqrt{|D_n|}\to 0$, $b_n\sqrt{|D_n|/p_n^2} \to \infty$, $c_n\sqrt{p_n}\to 0$ and $p_n^2/|D_n|\to 0$ as $n\to\infty$, the root-($|D_n|/p_n$) consistent local maximizer ${ \boldsymbol {\hat { \theta}}}=(\boldsymbol{\hat \theta}_1^\top, \boldsymbol{\hat \theta}_2^\top)^\top $ in Theorem 2 satisfies the two following properties:
\begin{enumerate}[(i)]
\item Sparsity: $\mathrm{P}(\boldsymbol{\hat \theta}_2=0) \to 1$ as $n \to \infty$,
\item Asymptotic Normality: $|D_n|^{1/2} \boldsymbol \Sigma_n(\bX;\btheta_{0})^{-1/2}(\boldsymbol{\hat \theta}_1- \boldsymbol{\theta}_{01})\xrightarrow{d} \mathcal{N}(0, \mathbf{I}_{m})$,
\end{enumerate}
where
\begin{align}
\boldsymbol \Sigma_n(\bX;\boldsymbol{\theta}_{0})= & |D_n|\{\mathbf{A}_{n,11}(\bX;\boldsymbol{\theta}_{0})+|D_n| \boldsymbol \Pi_n \}^{-1}\{\mathbf{A}_{n,11}(\bX;\boldsymbol{\theta}_{0})+\mathbf{B}_{n,11}(\bX;\boldsymbol{\theta}_{0})\}\nonumber \\
& \{\mathbf{A}_{n,11}(\bX;\boldsymbol{\theta}_{0})+|D_n| \boldsymbol \Pi_n \}^{-1}, \label{eq:Sigman} \\
\boldsymbol \Pi_n = & \mathrm{diag}\{p''_{\lambda_{n,1}}(|\theta_{01}|),\ldots,p''_{\lambda_{n,l}}(|\theta_{0l}|),p''_{\lambda_{n,l+1}}(|\theta_{0(l+1)}|),\ldots,p''_{\lambda_{n,m}}(|\theta_{0m}|)\}. \label{eq:pi}
\end{align}
\end{theorem}

As a consequence, $\boldsymbol \Sigma_n(\bX;\boldsymbol{\theta}_{0})$ is the asymptotic covariance matrix of $\boldsymbol{\hat \theta}_1$. Note that $\boldsymbol \Sigma_n(\bX;\boldsymbol{\theta}_{0})^{-1/2}$ is the inverse of $\boldsymbol \Sigma_n(\bX;\boldsymbol{\theta}_{0})^{1/2}$, where $\boldsymbol \Sigma_n(\bX;\boldsymbol{\theta}_{0})^{1/2}$ is any square matrix with $\boldsymbol \Sigma_n(\bX;\boldsymbol{\theta}_{0})^{1/2}\big(\boldsymbol \Sigma_n(\bX;\boldsymbol{\theta}_{0})^{1/2}\big)^\top=\boldsymbol \Sigma_n(\bX;\boldsymbol{\theta}_{0})$.

It is worth pointing out that these asymptotic results do not exist in the literature for inhomogeneous Gibbs models in an unregularized setting. The following result fills this gap.
\begin{theorem}
\label{thm:the4}
 Suppose that the conditions ($\mathcal C$.\ref{C:Dn})-($\mathcal C$.\ref{C:locsta}) and ($\mathcal C$.\ref{C:An}) are satisfied. Then, there exists a local maximizer $\hat \btheta$ of $\mbox{LPL}_n(\bX;\btheta)$  such that    ${\bf \| \hat \btheta -\btheta_0\|}=O_\mathrm{P}(|D_n|^{-1/2})$. If in addition, condition ($\mathcal C$.\ref{C:BnCn}) holds, then such a root-$|D_n|$ consistent local maximizer ${ \boldsymbol {\hat { \theta}}}$ satisfies the following property:
\begin{align*}
\{ \mathbf{A}_{n}(\bX;\boldsymbol \theta_{0})+\mathbf{B}_{n}(\bX;\boldsymbol \theta_{0})\}^{-1/2}
\mathbf{A}_{n}(\bX;\boldsymbol \theta_{0})(\boldsymbol{\hat \theta}-\boldsymbol \theta_{0})&\xrightarrow{d} \mathcal{N}(0,\mathbf{I}).
\end{align*} 
\end{theorem}

We claim that Theorem~\ref{thm:the4} ensures from Theorems~\ref{THM:ROOT}-\ref{THM:SPARSITYCLT}. It suffices to set the penalty part to zero and to consider only the pseudo-likelihood function $\mbox{LPL}_n( \bX;.)$ as the objective function in the proofs of Theorems~\ref{THM:ROOT}-\ref{THM:SPARSITYCLT}. Proof of Theorem~\ref{thm:the4} is therefore omitted. 

\subsection{Discussion of the conditions} \label{sec:conditions}

Condition ($\mathcal C$.\ref{C:Theta}) ensures that the spatial GPP has a finite interaction range while condition ($\mathcal C$.\ref{C:locsta}) ensures the local stability property. These two conditions are necessary to guarantee the existence of non-stationary exponential Gibbs models defined on $\mathbb{R}^d$ as mentioned in Theorem \ref{existence}. Note that ($\mathcal C$.\ref{C:Dn}) is needed in order to define the spatial GPP on  $\mathbb{R}^d$, which is in agreement with the investigation of its asymptotic properties. Condition ($\mathcal C$.\ref{C:cov}) combined with ($\mathcal C$.\ref{C:Theta}) and ($\mathcal C$.\ref{C:locsta}) allows us to bound the matrices $\mathbf{A}_n(\btheta_{0})$ and $\mathbf{B}_n(\btheta_{0})$ by $p_n \vert D_n \vert$, which in turn induces the boundedness of the variance of the score function by the same bound. From conditions  ($\mathcal C$.\ref{C:BnCn})-($\mathcal C$.\ref{C:An}), the matrix $\boldsymbol \Sigma_n(\bX;\boldsymbol{\theta}_{0})$ is invertible for sufficiently large $n$. 
The asymptotic normality of $\boldsymbol{\hat \theta}_1$ mainly ensues from a central limit theorem (CLT) for $\mathbf{LPL}_{n,1}^{(1)}( \bX;\btheta_0)$. The latter is based on a general CLT for nonstationary conditionally centered random fields obtained by \citet{coeurjolly2017parametric}. Conditions ($\mathcal C$.\ref{C:Dn})-($\mathcal C$.\ref{C:BnCn}) are necessary to apply such a result. 
Condition ($\mathcal C$.\ref{C:plambda}) concerns only the penalty function and is similar to the one proposed by \citet{choiruddin2018convex}. Such an assumption is fulfilled for $\ell_1, \ell_2$, SCAD and MC+ penalty functions for example. 

\section{Numerical aspects} \label{sec:num}

We describe in this section the numerical aspects of our optimization problem. The main aspect is to approximate numerically the integral part in the expression of the log-pseudo-likelihood function defined in  (\ref{ch2:Pois}). \cite{baddeley2000practical} define the Bermann-Turner approach which consists in using a finite sum approximation   
\[
\int_{D} \lambda_{\btheta}(u,\bx) \mathrm{d}u \approx \sum_{i=1}^{n+m} \nu_i \lambda_{\btheta}(u_i,\bx)
\]
which yields to an approximation of the log-pseudo-likelihood function as follows
\begin{align}
\label{approxpseudo}
\mbox{LPL}( \bx;\btheta) \approx \sum_{i=1}^{n+m} \nu_i (y_i\log\lambda_{\btheta}(u_i,\bx)-\lambda_{\btheta}(u_i,\bx))
\end{align}
where $u_i$, $i=1,\ldots,n+m$ are points in $D$ consisting of the $n$ data points $\bx$ and $m$ dummy points, the $\nu_i$ are quadrature weights summing to the volume of $D$ and finally

\[
y_i = \left \{
    \begin{array}{ll}
        1/\nu_i & \mbox{if} \; u_i \; \mbox{is a data point}, \; u_i \in \bx  \\
        0 & \mbox{if} \; u_i \; \mbox{is a dummy point}, \; u_i \notin \bx. 
    \end{array}
\right. 
\]
Note that this approximation performs well when $m$ is large. 
Including the $n$ data points in the approximation yields a small bias. However, now
the right-hand side of equation (\ref{approxpseudo}) corresponds to the log-likelihood function of a weighted Poisson regression model with responses $y_i$ and weights $\nu_i$, which leads to the use of standard statistical software for this fitting method. This is exploited in the \textsf{spatstat R} package \citep{baddeley2015spatial} by using the \textsf{ppm} function with the default method, \textsf{method=``mpl"}. In the situation where the number of points is quite large, using (\ref{approxpseudo}) to fit Gibbs model can be computationally intensive. To get around this drawback, one can use the logistic composite log-likelihood function proposed by \cite{baddeley2014logistic}
\begin{align}
\label{logistic}
\mbox{LCL}( \bx;\btheta) = \sum_{u \in \bx \cap D} \log \left(\frac{\lambda_{\btheta}(u,\bx)}{\delta(u) + \lambda_{\btheta}(u,\bx)} \right) + \sum_{u \in \bd \cap D} \log \left( \frac{\delta(u)}{\delta(u) +\lambda_{\btheta}(u,\bx)} \right)
\end{align}
where $\bd$ is a realization of  a dummy point process $\mathcal{D}$ independent of $\bX$ with known intensity function $\delta(u)$. Conditional on $\bx \cup \bd$, (\ref{logistic}) corresponds to the logistic regression model with responses $y(u)=1 \{ u \in \bx \}$ and offset term $-\log \delta(u)$. Thus, standard software for generalized linear models may also be used to implement this logistic fitting method; and this is provided in \textsf{R} by calling the \textsf{ppm} function with option \textsf{method=``logi"} of the \textsf{spatstat} package.   \newline
The approximate penalized log-pseudo-likelihood or logistic composite log-likelihood function is then 
\begin{align}
\label{approxQ}
\mbox{Q}(\bx;\btheta) \approx \mbox{CL}(\bx;\btheta) - |D|{\sum_{j=l+1}^{p} p_{\lambda_{j}}(|\theta_{j}|)}
\end{align}
where $\mbox{CL}(\bx;\btheta)$ corresponds to either the approximate log-pseudo-likelihood or the logistic composite log-likelihood function.  \newline
The regularization paths in (\ref{approxQ}) can be efficiently computed using the coordinate descent algorithm \citep{friedman2007pathwise,friedman2010regularization,breheny2011coordinate}, which is generally a fast and popular algorithm for estimation of penalized generalized linear models. More precisely we adopt cyclical coordinate descent methods, which can handle high-dimensional problems and can take advantage of sparsity. This algorithm is well described in \citet{choiruddin2018convex} and \citet{daniel2018penalized} for spatial point processes regularized with convex and non-convex penalties, incorporating both likelihood and composite likelihood. They provide quadratic approximation of the objective function, which for the pseudo-likelihood function defined in (\ref{approxpseudo}) is given by 
\begin{align}
\label{quadapprox}
\mbox{LPL}( \bx;\btheta) \approx \mbox{LPL}_{Q}( \bx;\btheta) = - \frac12 \sum_{i=1}^{N} {\boldsymbol{\nu}}_i (y_i^{\star} - \bt_i^\top \btheta )^2 + C(\widetilde{\btheta}),
\end{align} 
where $N=n+m$, $C(\widetilde{\btheta})$ is a constant, $y_i^{\star}$ are the working response values and ${\boldsymbol{\nu}}_i $ are the weights:
\begin{align*}
{\boldsymbol{\nu}}_i  &= \nu_i \exp( \bt_i^\top \widetilde{\btheta}) \\
y_i^{\star} &= \bt_i^\top \widetilde{\btheta} + \frac{y_i - \exp( \bt_i^\top \widetilde{\btheta})}{\exp( \bt_i^\top \widetilde{\btheta})}.
\end{align*}
The rest of the algorithm in our setting follows the same lines as those described in \citet{choiruddin2018convex} for the convex and non-convex penalties, i.e. (a) identify a decreasing sequences of the tuning parameter $\lambda$; (b) for each value of $\lambda$, compute $\mbox{LPL}_{Q}$ and then use the coordinate descent method to solve the following optimization problem
\begin{align}
\label{Omega}
\underset{\btheta \in \mathbb{R}^p}{\mathrm{min}} \left\{ - \mbox{LPL}_{Q}( \bx;\btheta) + |D|{\sum_{j=l+1}^{p} p_{\lambda_{j}}(|\theta_{j}|)} \right\}.
\end{align}
For full details on the algorithm, we refer the reader to \cite{choiruddin2018convex} and \citet{daniel2018penalized}.
The choice of the tuning parameter $\lambda$ is a challenging task when dealing with regularization techniques. Notice that large (small) values of $\lambda$ produce estimates with high (low) biases and low (high) variances. Therefore, an optimal choice of $\lambda$ is necessary to control the trade-off between the bias and the variance of the estimates. \cite{friedman2010regularization}, \cite{breheny2011coordinate}, among others have described in details the selection of the tuning parameter $\lambda$. First, it consists in identifying a sequence of $\lambda$ ranging from a maximum value of $\lambda$ for which all penalized coefficients are zero to the value of $\lambda$ corresponding to the unregularized parameter estimates, i.e. $\lambda=0$. Secondly, one can define a criterion, say $C$ to select $\lambda$, generally it consists in maximizing or minimizing $C$. For models with tractable likelihood functions like inhomogeneous Poisson point process, an approach for selecting the tuning parameter $\lambda$, is via information criteria, with the most commonly used based on BIC \citep{schwarz1978estimating}. Following \citet{gao2010composite}, a composite likelihood analogue of this criteria may be used for Gibbs models. Let us now define the matrices $\bH$ and $\bV$, which are involved in  the expression of information criteria for composite likelihoods, 
\[
\bH = \EE \left \{ - \mathbf{CL}^{(2)}( \bx;\btheta)  \right \} \; \mbox{and} \; \bV = \Var \left \{  \mathbf{CL}^{(1)}( \bx;\btheta) \right \}.
\]
Note that $\bSigma = \bH^{-1} \bV \bH^{-1}$ is the asymptotic variance-covariance matrix of the estimate $\hat{\btheta}$ and estimates of $\bH$ and $\bSigma$  can be efficiently computed using the \textsf{vcov} function of the \textsf{spatstat R} package \citep{coeurjolly2013fast}. 
The composite BIC \citep{gao2010composite}, which we denote by cBIC is defined as follows:
\begin{equation}
	\label{eq:defcbic}
\text{cBIC}(\lambda) = - 2 \mathrm{CL}(\bx;\hat{\btheta}_\lambda) + \log(n) d(\lambda) 
\end{equation}
where 
\begin{equation}
	d(\lambda) = \text{trace}(\hat{\bV}_\lambda \hat{\bH}^{-1}_\lambda)=\text{trace} (\hat{\bH}_\lambda \hat{\bSigma}_\lambda) \label{eq:dlambda}
\end{equation}
is called the  effective number of parameters in the model with tuning parameter $\lambda$. $\hat{\bSigma}_\lambda$ and $\hat{\bH}_\lambda$ are estimates of $\bSigma$ and $\bH$ at $\hat{\btheta}$ where we remove the lines and columns corresponding to the indices of zero coefficients of $\hat{\btheta}_\lambda$, i.e. $\hat{\bSigma}_\lambda=\hat{\bSigma}\vert_{\btheta=\hat{\btheta}_\lambda}$ and $\hat{\bH}_\lambda=\hat{\bH}\vert_{\btheta=\hat{\btheta}_\lambda}$. It is to be noticed that $\hat{\bSigma}_\lambda$ and $\hat{\bH}_\lambda$ are $a \times a$ matrices where $a = \vert \{j=0,1,\ldots,p: \; \hat{\theta}_j(\lambda) \neq 0 \} \vert$. In~\eqref{eq:defcbic}, the letter $n$ stands for the observed number of points. That choice could be discussed as it is the realization of a random variable. Other choices could be $|W|$ (used by \citet{choiruddin2018convex}) or $n+m$ (used by~\citet{daniel2018penalized}). The recent work by~\citet{choiruddin2020bic} shows that both from a theoretical and practical point of view, the most pertinent choice is $n$. 

Standard information criteria were designed for use in  an unregularized framework with objective function either a likelihood or a pseudo-likelihood (composite) function.
To take into account the effects of the tuning parameter $\lambda$, \cite{hui2015tuning}  consider the extended regularized information criterion (ERIC) for tuning parameter selection involving the likelihood function from a GLM and the adaptive lasso. Following \citet{daniel2018penalized} and according to the definition of our penalized composite likelihood, we propose the composite analogue of ERIC to Gibbs models
\begin{equation}
	\label{eq:defceric}
\text{cERIC}(\lambda) = - 2 \mathrm{CL}(\bx;\hat{\btheta}_\lambda) + \log \left(\frac{n}{ \vert D \vert \lambda} \right) d(\lambda). 
\end{equation}
Note that we choose the tuning parameter $\lambda \geq 0$ which minimizes the criterion of choice, i.e. either cBIC or cERIC in the present study, and this will yield the optimal GPP model. For ease of computation, we fix for SCAD, $\gamma=3.7$ following \citet{fan2001variable} and for MC+, $\gamma=3$ following \citet{breheny2011coordinate}.


\section{Simulation study} \label{sec:sim}

\subsection{Simulation set-up}
\label{sim:setup}
We conduct a simulation study to evaluate the performances (prediction and selection) of the regularized pseudo-likelihood estimator in an increasing domain framework, in which the total number of covariates increases as well as the size of the spatial domain. The spatial domains are $W_1=[0,250] \times [0,125]$, $W_2=[0,500] \times [0,250]$ and $W_3=[0,1000] \times [0,500]$. The setting for generating the spatial covariates is similar to that of \cite{waagepetersen2007estimating}, \cite{thurman2015regularized} and \cite{choiruddin2018convex}. We use the $201 \times 101$ pixel images of elevation ($x_1$) and slope ($x_2$) contained in the \textsf{bei} datasets of \textsf{spatstat} library in \textsf{R} \citep{team2019r} as two true covariates. Note that we standardize these two covariates, i.e. centered and scaled them. We first consider a \textsf{Scenario 0}, in which the covariates are only $x_1$ and $x_2$, i.e. for each domain $W_k \; (k=1,2,3)$ the total number of covariates is set to $p=2$. The aim of this scenario is to assess the asymptotic properties of the pseudo-likelihood estimator in an unregularized setting. Note that the regression coefficients of $x_1$ and $x_2$ are set respectively at $2$ and $0.75$, which means a relatively large effect of elevation and  a relatively small effect of slope. In a second time,  we set the total number of covariates at $p_k= \lfloor 3 \vert W_k  \vert^{1/4} \rfloor$ and define $q_k= p_k - 2$ for each domain $W_k$, $k=1,2,3$ where $\lfloor . \rfloor$ is the floor function and $\vert W_k  \vert$ is the area of $W_k$, this corresponds to $p_1=39$, $p_2=56$ and $p_3=79$; and create two different scenarios to define extra covariates for the regularized setting:
\begin{enumerate}[ \textsf{Scenario}~1.]
\item For each domain $W_k$, we generate $q_k$ $201 \times 101$ pixel images of covariates as standard Gaussian white noise. We denote these covariates by $x_3,\cdots,x_{q_k}$. In order to build multicollinearity, we transform them together with $x_1$ and $x_2$ through $\mathbf{z}(u)= \mathbf{V}^\top \bx(u)$ where $\bx(u)=\{ x_1(u),\cdots,x_{q_k}(u) \}^\top$. Note that $\mathbf{V}$ is such that $\boldsymbol\Omega = \mathbf{V}^\top \mathbf{V}$ where $\boldsymbol\Omega$ has entries $(\boldsymbol\Omega)_{ij}=(\boldsymbol\Omega)_{ji}=0.7^{\vert i - j \vert}$ for $i,j=1,\cdots,q_k$, except $(\boldsymbol\Omega)_{12}=(\boldsymbol\Omega)_{21}=0$, to preserve the correlation between the covariates $x_1$ and $x_2$. It is to be noticed that the regression coefficients for $z_3,\cdots,z_{q_k}$ are zero. \label{sce1} \\

\item We first center and scale the $13$ $50 \times 25$ pixel images of soil nutrients covariates obtained from the study in tropical forest of Barro Colorado Island (BCI) in central Panama  \cite[see][]{condit1998tropical,hubbell1999light,hubbell2005barro}. Then, we convert them to be $201 \times 101$ pixels images as $x_1$ and $x_2$. Note that for each $W_k$, we have that $q_k > 13$. So, in order to have $p_k$ covariates in total we consider the interaction between two soil nutrients. In this setting, the extra covariates are $13$ soil nutrients and $q_k -13$ interactions between them. Together with $x_1$ and $x_2$, we keep the structure of the covariance matrix to preserve the complexity of the situation. We set to zero the regression coefficients of $z_3,\cdots,z_{q_k}$ where $\mathbf{z}(u)=\{1,x_1(u),\cdots,x_{q_k}(u) \}^\top$. \label{sce2}

\end{enumerate}
We consider three Gibbs models: two Strauss processes with $R=9.25$ and $\psi=\log(\gamma)$, $\gamma=0.2,0.5$; and a Geyer saturation process with threshold $\sigma=1$, $R=9.25$, and $\psi=\log(1.5)$. For each model, the true conditional intensity is set to be $\lambda_{\boldsymbol{\theta}}(u,\mathbf{x}) = \exp(\beta_0 \, + \, \beta_1 z_1(u) + \beta_2 z_2(u) + \psi \, s(u,\mathbf{x}))$, where $\beta_1$ and $\beta_2$ are respectively the coefficients of the centered and scaled elevation and slope covariates, i.e. $\beta_1=2$ and $\beta_2=0.75$. The intercept $\beta_0$ is fixed so that we have under the Poisson model ($\psi=0$) in average $500$ points in $W_1$, $2000$ points in $W_2$ and $4000$ points in $W_3$. Note that for the implementation in \textsf{R} of the fitting method via the \textsf{ppm} function of \textsf{spatstat} package involving the pseudo-likelihood function, we consider two choices regarding the  number of dummy points $\textsf{nd}^2$ in term of the observed number of points $n$ to control efficiently the bias induced by the Bermann-Turner approximation. In the presence of regular patterns like the two Strauss models, we take $\textsf{nd}^2 \approx 256n$. For the clustered pattern, we consider the choice $\textsf{nd}^2 \approx 4n$, as suggested in the \textsf{spatstat R} package. With these scenarios, we simulate $500$ point patterns  from each model within each spatial domain using the Metropolis-Hastings algorithm implemented in the \textsf{rmh} function of the \textsf{spatstat} package. For each pattern in \textsf{Scenario 0}, we fit the conditional intensity via pseudo-likelihood and assess performance by looking at prediction properties of the estimation method: the bias, the standard deviation (SD) and the square root of mean squared errors (RMSE), which we define by
\[
\text{Bias} = \hspace{-.1cm} \left( \sum_{j=1}^{p_k} \left \{ \hat{\EE}(\hat{\theta}_j) - \theta_j \right \}^2 \right)^{\frac12},  \text{SD}= \hspace{-.1cm} \left ( \sum_{j=1}^{p_k} \hat{\sigma}^2_j \right )^{\frac12},  \text{RMSE}= \hspace{-.1cm} \left ( \sum_{j=1}^{p_k}   \hat{\EE}(\hat{\theta}_j -\theta_j)^2 \right )^{\frac12}
\]      
where $\hat{\mathbb{E}}(\hat{\theta}_j)$ and  $\hat{\sigma}^2_j$ are respectively the empirical mean and variance of the estimates $\hat{\theta}_j$, for $j=1,\cdots,p_k$ where $k=1,2,3$. For \textsf{Scenarios}~\ref{sce1} and~\ref{sce2}, we fit a regularization path to the simulated point patterns using modified internal function in \textsf{spatstat}, \textsf{glmnet} and \textsf{ncvreg}. Note that in our setting, a modification of the \textsf{ncvreg R} package is required to include option for weights. The optimal model from each path as stated at the end of Section~\ref{sec:num}, is chosen on  the basis of the minimizer of either cBIC or cERIC respectively defined by~\eqref{eq:defcbic} and~\eqref{eq:defceric}. 
To evaluate $d(\lambda)$ given by~\eqref{eq:dlambda}, we need to estimate matrices $\mathbf H$ and $\boldsymbol\Sigma$, which can be obtained using the \textsf{vcov} function of the \textsf{spatstat} package. The computation of $\boldsymbol\Sigma$ using the \textsf{vcov} function involves the number of quadrature points $N$ and the number of covariates and it fails when both of the latter parameters are large. This is the case for the clustered pattern we considered (Geyer model) in the spatial domains $W_2$ and $W_3$. So, to get around this drawback, we conduct a parametric bootstrap approach, which we now describe. We generate $100$ point patterns from the Geyer model within each spatial domain ($W_2$, $W_3$) and estimate the coefficients of the parameter via the \textsf{ppm} function. The estimate of $\boldsymbol\Sigma$ is then obtained by computing the empirical covariance matrix of parameter estimates.

 We also look at the prediction performance of the estimation method for these two scenarios in the similar way we do for \textsf{Scenario 0}. In addition, we evaluate the selection performance of the two information criteria by looking at the true positive rate (TPR) and the false positive rate (FPR). TPR is defined as the ratio of the selected true covariates over the number of true covariates while FPR corresponds to the ratio of the selected noisy covariates over the number of noisy covariates. In our setting, TPR explains how the model can correctly select the true covariates, elevation $x_1$ and slope $x_2$; and FPR investigates how the model incorrectly select among the noise covariates $x_3$ to $x_{p_k}$, for $k=1,2,3$. As a rule of thumb, methods with TPRs close to $100 \%$ and FPRs close to $0$ present a good selection performance.

\subsection{Simulation results} 
\label{sim:results}
Table~\ref{table:scen0} presents prediction properties of the pseudo-likelihood estimator in an increasing spatial domain. Note that the results summarized in Table~\ref{table:scen0} correspond also to that of oracle model in \textsf{Scenarios} \ref{sce1} and \ref{sce2}. For all three models, the biases, the SDs and the RMSEs diminish as the spatial domain grows and the number of points increases, whereas the biases in spatial domains $W_2$ and $W_3$ do not vary much for each model. Overall, the bias tends towards 0 and the RMSE is improved as the spatial domain grew for all three models. The results illustrate the consistency of the pseudo-likelihood estimator for  Gibbs models. \newline
Prediction and selection properties of the regularized pseudo-likelihood estimator under different penalty functions are reported in Figures~\ref{fig:sc1.cbic}, \ref{fig:sc1.ceric} and \ref{fig:sc1.cbic.ceric} for \textsf{Scenario} \ref{sce1} and in Figures~\ref{fig:sc2.cbic}, \ref{fig:sc2.ceric} and \ref{fig:sc2.cbic.ceric} for \textsf{Scenario} \ref{sce2}. For more details on the figures, see the tables  in Appendix~\ref{tab:sim}. Generally, predictive performance and variable selection are more challenging for the Strauss models than for the Geyer model. The RMSEs are worse for the Strauss models in $W_1$, improve gradually for the Strauss models in spatial domains $W_2$ and $W_3$ and reach the best values for the Geyer model under some regularization methods by approaching the RMSEs of the oracle method. For all three models, the variable selection indices significantly improve with increasing spatial domain, whereas the improvement is more stable for the Geyer model. This improvement is largely driven by the number of points in the model and the size of the spatial domain. \newline
From Figures~\ref{fig:sc1.cbic} and \ref{fig:sc1.ceric}, one observes that the adaptive lasso and adaptive elastic net methods significantly outperform other regularization methods in terms of predictive performance and variable selection. Firstly, their RMSEs are very competitive as they frequently agree with (or close to) that of oracle method for the clustered model across all sizes of spatial domain. Secondly, they have the highest TPRs and lowest FPRs, yielding good overall selection performance. The latter means that the true covariates are selected and the noise covariates are eliminated more frequently from these two models. Moreover, these performances in prediction and selection are slightly amplified by considering a more complex design in \textsf{Scenario} \ref{sce2}, the results of which are reported in Figures~\ref{fig:sc2.cbic} and \ref{fig:sc2.ceric}. It is a known fact that ridge regularization method always selects all covariates, which implies that the rates remain unchanged (TPR and FPR equal to $100\%$). Lasso and elastic net present quite large values of FPRs in \textsf{Scenario} \ref{sce1} (see Figures~\ref{fig:sc1.cbic} and \ref{fig:sc1.ceric}), meaning that they wrongly select the noise covariates more frequently. As the design is getting more complex for \textsf{Scenario} \ref{sce2} (see Figures~\ref{fig:sc2.cbic} and \ref{fig:sc2.ceric}), we gain smaller FPRs for lasso (not for elastic net) but suffer at the same time from smaller TPRs. These smaller TPRs are explained by the non-selection of the slope covariate ($x_2$), which has smaller coefficient than that of the elevation covariate ($x_1$). This yields poor overall selection performance for lasso and elastic net. 
The SCAD and MC+ methods perform comparably, outperform lasso, ridge and elastic net. Considering \textsf{Scenario}~\ref{sce1}, SCAD and MC+ have the highest FPRs (after ridge) in a small spatial domain and TPRs comparable to that of adaptive lasso across all sizes of spatial domain, yielding a poor overall selection performance. As the setting is getting more complex for \textsf{Scenario} \ref{sce2}, SCAD and MC+ present a good overall selection performance as they have quite large TPRs and small FPRs, close to that of adaptive lasso. In terms of predictive performance, RMSEs of SCAD and MC+ are small and close to that of adaptive lasso with an exception that the RMSEs are quite large in a small spatial domain with a small number of points. Overall, we recommend adaptive lasso and adaptive elastic net as methods of regularization for spatial Gibbs models. Moreover, in situations where the observed point pattern has a large number of points (clustered point pattern for e.g.) and the covariance matrix of the covariates has a complex structure (as in \textsf{Scenario} \ref{sce2}), we also recommend the non-convex penalties SCAD and MC+ as methods of regularization in addition to those recommended above.\newline
RMSEs, TPRs and FPRs of the adaptive lasso and adaptive elastic net are reported in Figures \ref{fig:sc1.cbic.ceric} and \ref{fig:sc2.cbic.ceric} in order to compare the performance of cBIC against cERIC. Considering \textsf{Scenario} \ref{sce1}, cERIC has the lowest TPRs (with adaptive elastic net) in a small spatial domain but also the lowest FPRs across all sizes of spatial domain. This is consistent with ERIC's aggressive shrinkage properties which greatly reduce the number of incorrectly-selected covariates at the risk of potentially not selecting some true covariates \citep{hui2015tuning}. In a moderate and large spatial domain, this disadvantage of cERIC with respect to TPR disappear due to high values of TPRs, which tend towards $100\%$ at a faster rate. For the more complex design in \textsf{Scenario} \ref{sce2}, cERIC significantly outperforms cBIC across all sizes of spatial domain in terms of variable selection. This improvement is largely driven by low FPRs and high TPRs for cERIC. In our setting, cBIC and cERIC perform comparably in terms of predictive performance as they have very competitive RMSEs across all \textsf{Scenarios} and sizes of spatial domain. Because of its good overall prediction and selection performances with adaptive lasso and adaptive elastic net, we recommend the use of cERIC for the tuning of the regularization parameter.    

\setlength{\tabcolsep}{5pt}
\renewcommand{\arraystretch}{1.5}
\begin{table}[!ht]
\caption{Empirical prediction properties (Bias, SD, and RMSE) based on 500 replications of Strauss and Geyer models. The average number of points $n$ under each model is provided in the last column.}
\label{table:scen0} 
\centering
\begin{tabular}{c l c  ccc  c}
\hline
\hline 
& &  & Scenario 0 (unregularized)\\
\hline
 \multicolumn{1}{c}{Spatial} & \multicolumn{1}{c}{Model} & \multicolumn{1}{c}{Interaction} &  \multicolumn{3}{c}{Prediction properties} & \multicolumn{1}{c}{Av. number} \\ 
 \cline{4-6}
domain &  & parameter & Bias & SD & RMSE & of points (n) \\ 
  \hline
  \hline
  \multirow{3}{*}{$W_1$} & Strauss & $\gamma=0.2$ & 0.09 & 0.49 & 0.5 & 101 \\ 
   & Strauss & $\gamma=0.5$ & 0.11 & 0.36 & 0.38 &  138  \\ 
   & Geyer  & $\gamma=1.5$ & 0.07 & 0.46 & 0.47 &  750  \\ 
\hline
  \multirow{3}{*}{$W_2$} & Strauss & $\gamma=0.2$ & 0.01 & 0.24 & 0.24 &  395 \\ 
   & Strauss & $\gamma=0.5$ & 0 & 0.17 & 0.17 &  540  \\ 
  & Geyer  & $\gamma=1.5$ & 0.01 & 0.19 & 0.19 &  2968 \\ 
   \hline
  \multirow{3}{*}{$W_3$} & Strauss & $\gamma=0.2$ & 0.02 & 0.13 & 0.13 & 1137 \\ 
   & Strauss & $\gamma=0.5$ & 0.01 & 0.1 & 0.10 & 1484  \\ 
  & Geyer  & $\gamma=1.5$ & 0.03 & 0.09 & 0.09 &  5630 \\ 
   \hline
 \end{tabular}
\end{table}

\begin{figure}[!ht]
\begin{center}
\renewcommand{\arraystretch}{0}
\includegraphics[width=1\textwidth]{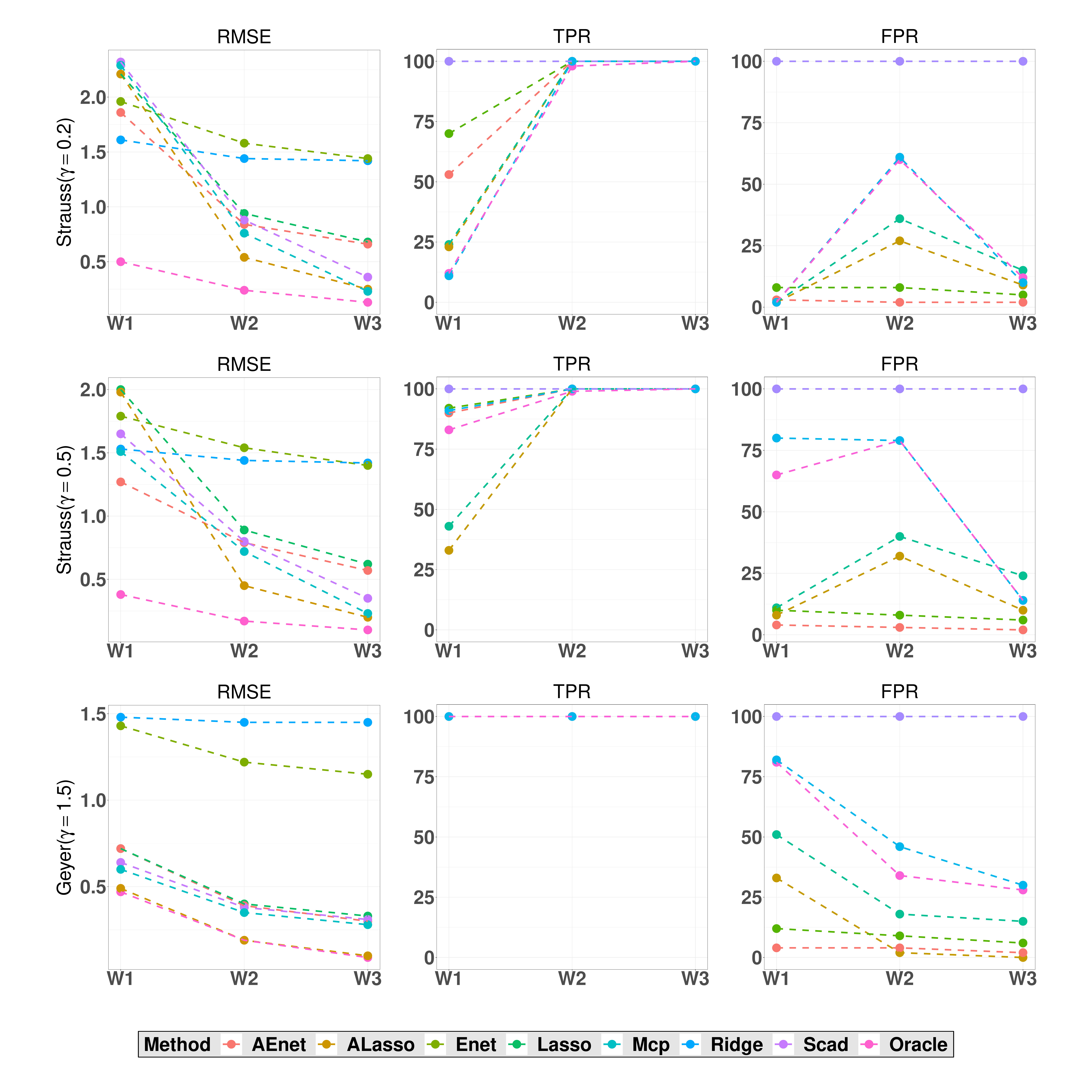} 
\caption{Empirical RMSE, TPR and FPR in terms of the observation window from Scenario \ref{sce1} (left to right) based on 500 replications of inhomogeneous Strauss models with $\gamma=0.2$ (first row) and $\gamma=0.5$ (second row), and inhomogeneous Geyer model with $\gamma=1.5$ (third row) using composite BIC.}
\label{fig:sc1.cbic}
\end{center}
\end{figure}

\begin{figure}[!ht]
\begin{center}
\renewcommand{\arraystretch}{0}
\includegraphics[width=1\textwidth]{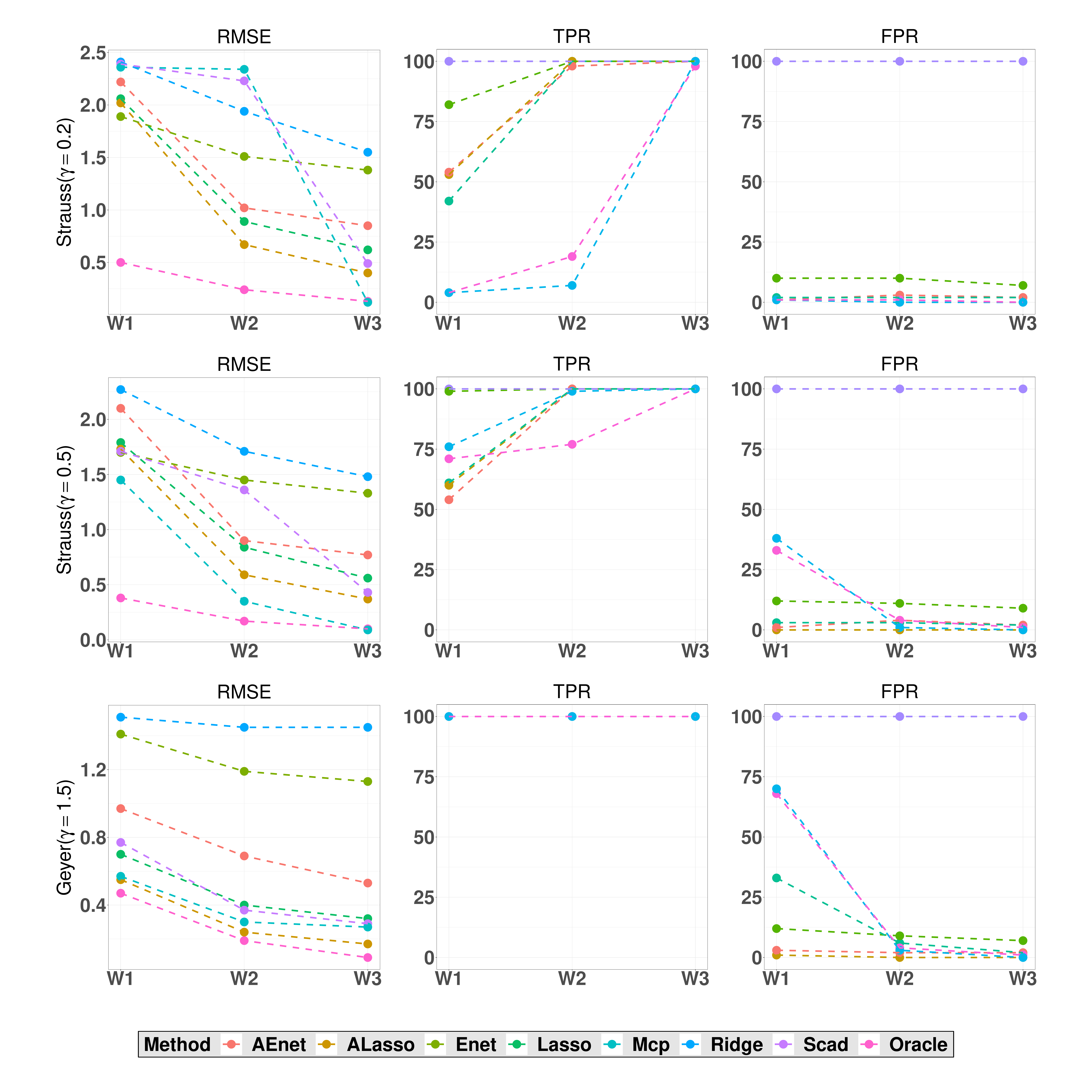} 
\caption{Empirical RMSE, TPR and FPR in terms of the observation window  from Scenario \ref{sce1} (left to right) based on 500 replications of inhomogeneous Strauss models with $\gamma=0.2$ (first row) and $\gamma=0.5$ (second row), and inhomogeneous Geyer model with $\gamma=1.5$ (third row) using composite ERIC.}
\label{fig:sc1.ceric}
\end{center}
\end{figure}

\begin{figure}[!ht]
\begin{center}
\renewcommand{\arraystretch}{0}
\includegraphics[width=1\textwidth]{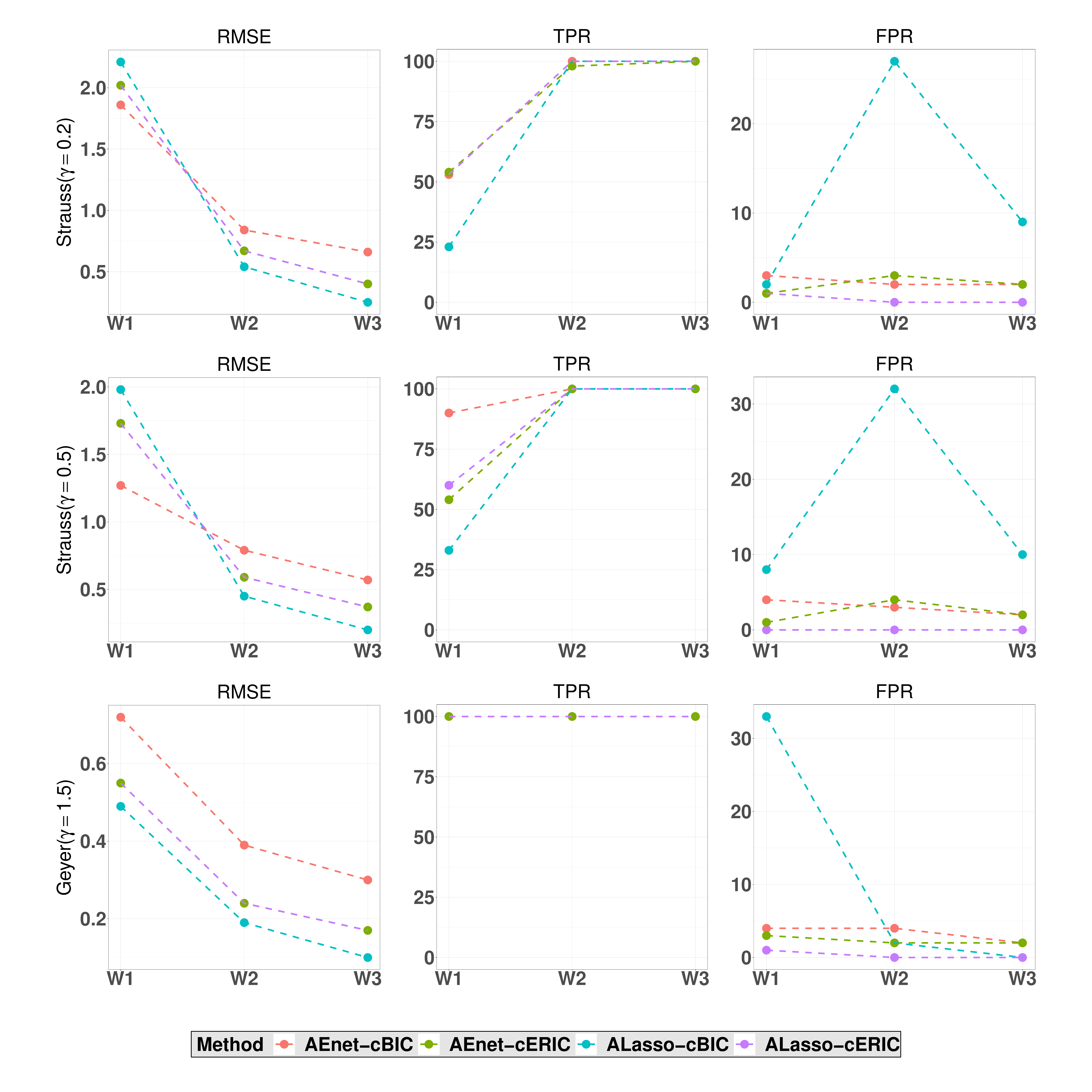} 
\caption{Empirical RMSE, TPR and FPR in terms of the observation window  from Scenario \ref{sce1} (left to right) for the ALasso and AEnet regularization methods. Results are based on 500 replications of inhomogeneous Strauss models with $\gamma=0.2$ (first row) and $\gamma=0.5$ (second row), and inhomogeneous Geyer model with $\gamma=1.5$ (third row) using composite BIC and ERIC.}
\label{fig:sc1.cbic.ceric}
\end{center}
\end{figure}

\begin{figure}[!ht]
\begin{center}
\renewcommand{\arraystretch}{0}
\includegraphics[width=1\textwidth]{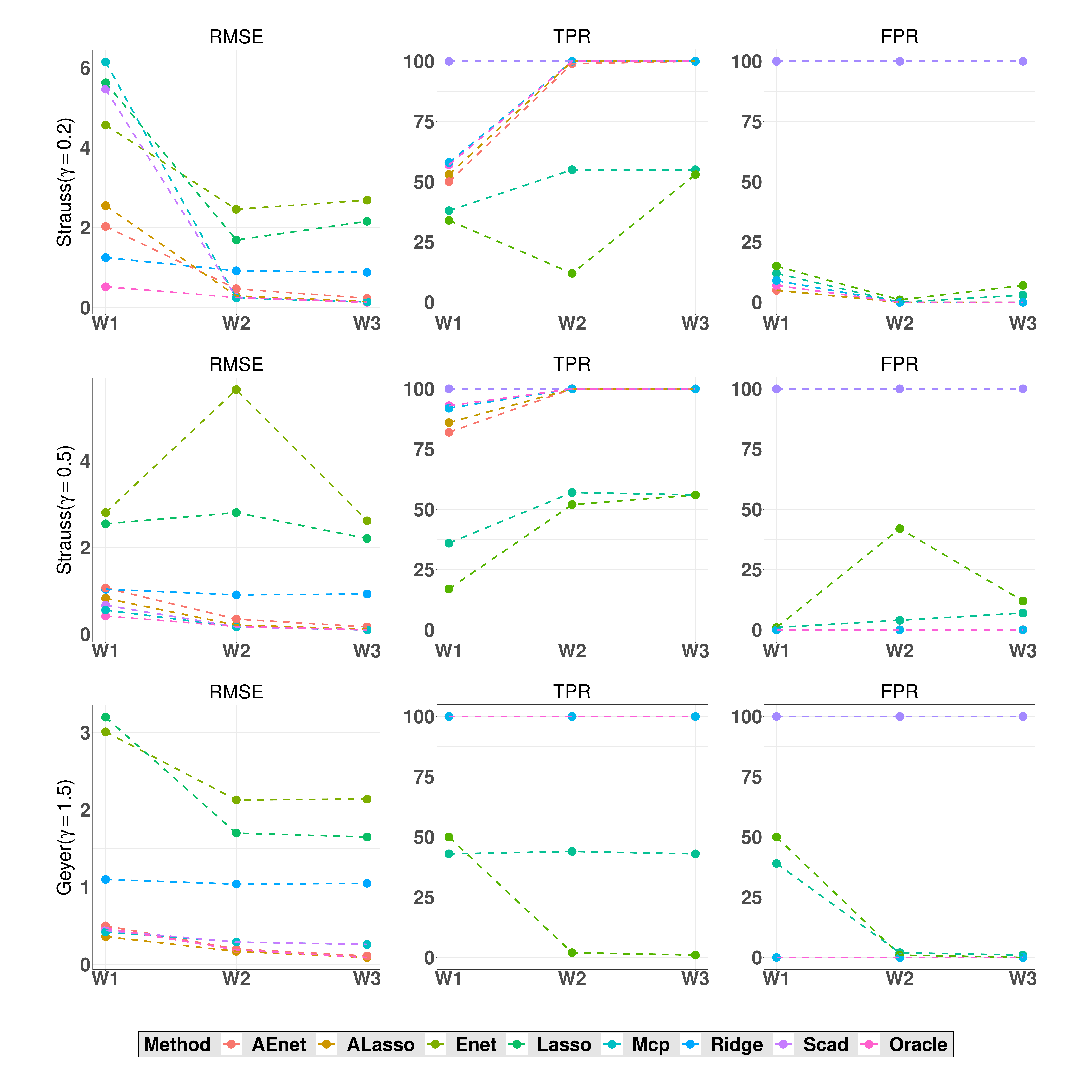} 
\caption{Empirical RMSE, TPR and FPR in terms of the observation window  from Scenario \ref{sce2} (left to right) based on 500 replications of inhomogeneous Strauss models with $\gamma=0.2$ (first row) and $\gamma=0.5$ (second row), and inhomogeneous Geyer model with $\gamma=1.5$ (third row) using composite BIC.}
\label{fig:sc2.cbic}
\end{center}
\end{figure}

\begin{figure}[!ht]
\begin{center}
\renewcommand{\arraystretch}{0}
\includegraphics[width=1\textwidth]{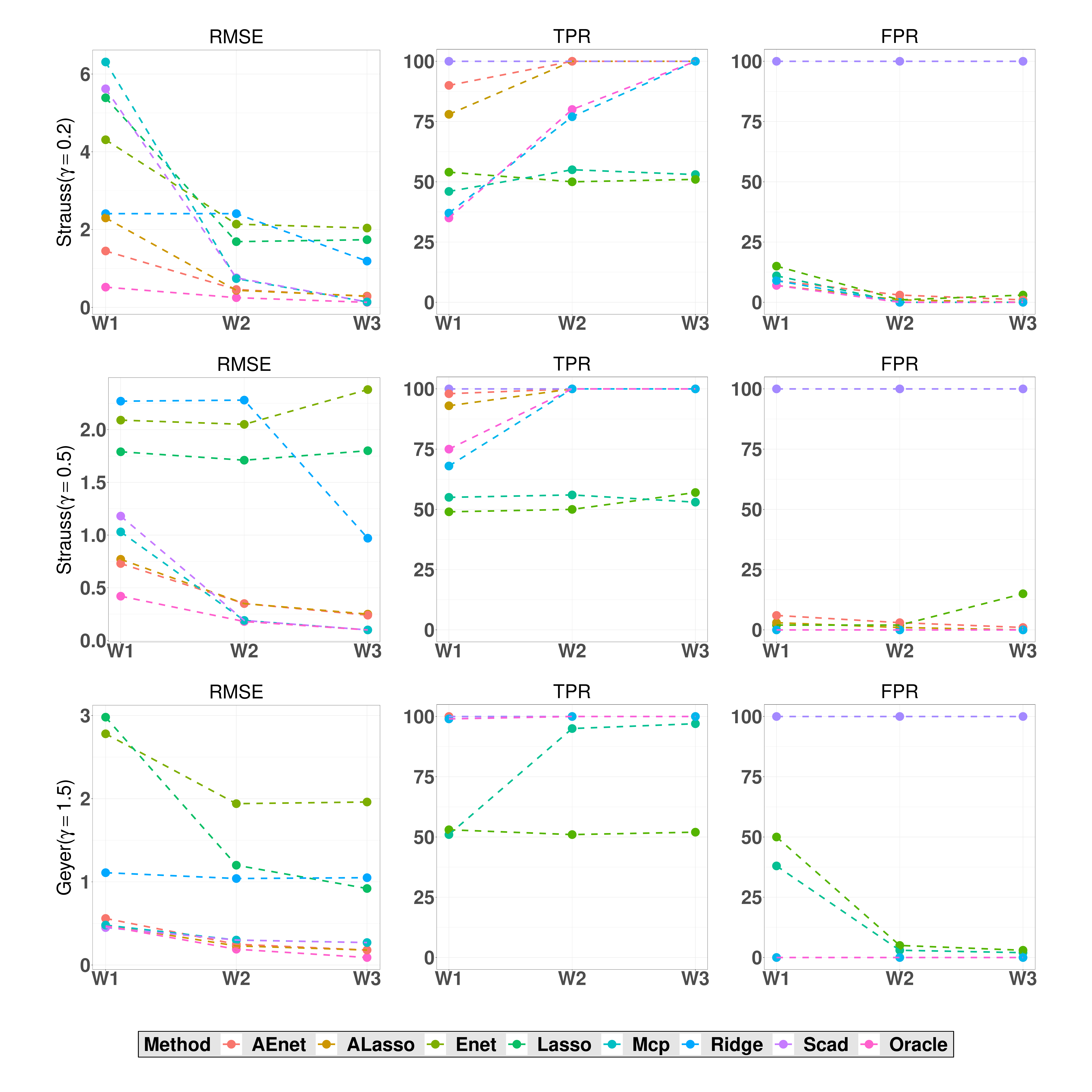} 
\caption{Empirical RMSE, TPR and FPR in terms of the observation window  from Scenario \ref{sce2} (left to right) based on 500 replications of inhomogeneous Strauss models with $\gamma=0.2$ (first row) and $\gamma=0.5$ (second row), and inhomogeneous Geyer model with $\gamma=1.5$ (third row) using composite ERIC.}
\label{fig:sc2.ceric}
\end{center}
\end{figure}

\begin{figure}[!ht]
\begin{center}
\renewcommand{\arraystretch}{0}
\includegraphics[width=1\textwidth]{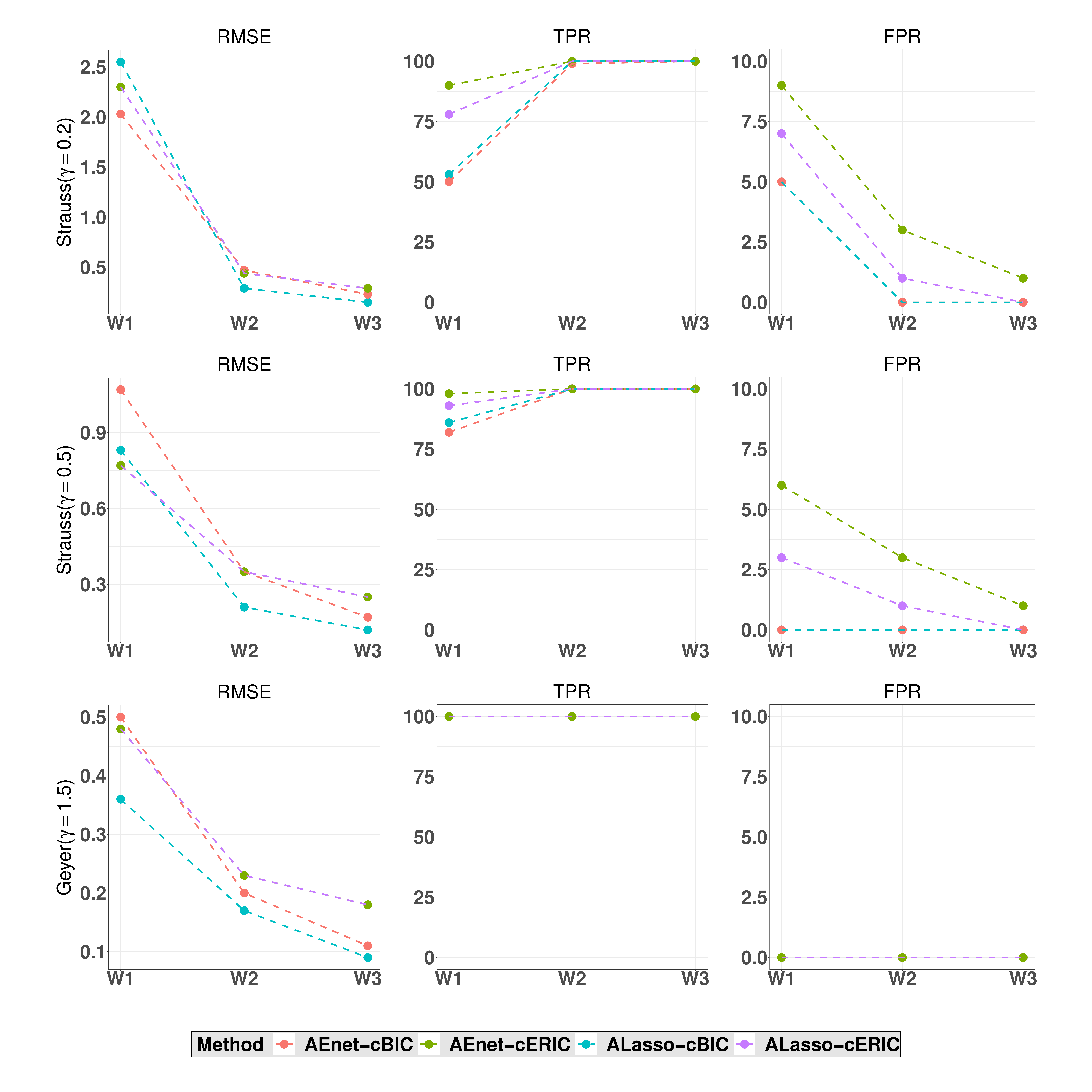} 
\caption{Empirical RMSE, TPR and FPR in terms of the observation window  from Scenario \ref{sce2} (left to right) for the ALasso and AEnet regularization methods. Results are based on 500 replications of inhomogeneous Strauss models with $\gamma=0.2$ (first row) and $\gamma=0.5$ (second row), and inhomogeneous Geyer model with $\gamma=1.5$ (third row) using composite BIC and ERIC.}
\label{fig:sc2.cbic.ceric}
\end{center}
\end{figure}

\subsection{Penalized logistic composite likelihood}

The focus here is to compare the regularized pseudo-likelihood estimator (RPE) and the regularized logistic estimator (RLE) in a fixed spatial domain with a different number of dummy points. More precisely, we compare their predictive performance and selection properties.  \newline
The simulation setting is a particular case of that of \textsf{Scenario} \ref{sce2} in the sense that we consider only the spatial domain $W = [0,250] \times [0,125]$. The number of covariates is set at $p= \lfloor 3 \vert W  \vert^{1/4} \rfloor$, that is $p=39$. We consider the same Gibbs models presented in Section \ref{sim:setup}, all with true Papangelou conditional intensity at the form $\lambda_{\boldsymbol{\theta}}(u,\mathbf{x}) = \exp(\beta_0 \, + 2 z_1(u) + 0.75 z_2(u) + \psi \,s(u,\mathbf{x}))$. The intercept $\beta_0$ is fixed so that  the Poisson point patterns have $900$ points on average. Note that $\psi=\log(\gamma)$ where $\gamma=0.2,0.5,1.5$; this corresponds respectively to two Strauss models and a Geyer model. We simulate $500$ point patterns from each model within spatial domain $W$ and consider three different choices for the number of dummy points $\mathsf{nd}^2$ in order to fit the conditional intensity to the simulated point patterns using both the pseudo-likelihood and the logistic composite likelihood functions. More precisely, we choose $\mathsf{nd} =15,30,60$ since the average number of points under the Poisson model is $900=30^2$. These choices correspond respectively to $\mathsf{nd}^2 < \pi$, $\mathsf{nd}^2 \approx \pi$ and $\mathsf{nd}^2 > \pi$, where $\pi$ is the expected number of data points (under the Poisson case). The latter point allows to compare the performance of the two fitting methods when the number of dummy points is less than, equal to and greater than the number of data points in an unregularized setting. Indeed, \citet{baddeley2014logistic} have shown that the logistic (composite) likelihood method is preferable to the (pseudo)-likelihood method in situations where the number of data points is very large and in presence of Gibbs model generating strong interactions, as it requires a small number of dummy points to perform quickly and efficiently. The aim of this section is to investigate these results in a regularized setting. \newline
From each simulated pattern, a regularization path is fitted using both the penalized pseudo-likelihood (PPL) and the penalized logistic composite likelihood (PLCL) with adaptive lasso penalty. The choice of the optimal model from each path is based on cERIC and finally, we use the same selection and prediction indices examined in Section \ref{sim:setup} to evaluate the selection and the prediction performances of the two regularized fitting methods. \newline             
Table~\ref{table:logis} displays the results of the selection and prediction indices for the RPE and the RLE. It is worth noticing that across all models, PPL and PLCL have respectively the lowest FPRs and the highest TPRs for all values of $\mathsf{nd}$. For the Geyer model, PPL and PLCL perform comparably for all values of $\mathsf{nd}$ in terms of selection performance. The FPRs of PLCL and PPL are close while their TPRs are very different across all values of $\mathsf{nd}$ for the Strauss models. This yields a good overall selection performance for PLCL. The selection performance of both methods improves as the number of dummy points increases. This improvement is largely driven by the values of TPRs, which increases at a faster rate for PLCL and at a lower rate for PPL. In terms of predictive performance, PLCL outperforms PPL across the two Strauss models for all values of $\mathsf{nd}$. PLCL has smaller bias and lower RMSE overall, despite their larger SD. For the Geyer model, PLCL slightly outperforms PPL as they have competitive RMSEs across all values of $\mathsf{nd}$ even if PLCL still has smaller bias and larger SD. Across all three models, estimates of PLCL are less biased than that of PPL. In situations where the number of observed data points is very large, we would recommend to apply  both PPL and PLCL methods with a preference for PLCL method when (a) $\mathsf{nd}^2 \approx \pi$ or $\mathsf{nd}^2 > \pi$ and PLCL method only when (b) $\mathsf{nd}^2 < \pi$. Note that (a) requires intensive computation and it seems more reasonable to choose (b). Otherwise, we recommend the use of PLCL for the regularization of spatial GPPs.             

\setlength{\tabcolsep}{3pt}
\renewcommand{\arraystretch}{1.5}
\begin{table}[!ht]
\caption{Empirical prediction properties (Bias, SD, and RMSE) and empirical selection properties (TPR, and FPR in $\%$) based on 500 replications of Strauss and Geyer models on spatial domain $W=[0,250] \times [0,125]$ using cERIC for $nd=15,\,30,\,60$. Penalized pseudo and logistic composite likelihoods are considered with adaptive lasso.}
\label{table:logis} 
\centering
\begin{tabular}{@{\extracolsep{1pt}}l c  c  | ccccc | ccccc @{}}
\hline
\hline 
 \multicolumn{1}{l}{nd} & \multicolumn{1}{c}{Model} & \multicolumn{1}{c}{Interaction}  & \multicolumn{5}{c}{PPL}  & \multicolumn{5}{c}{PLCL} \\ 
 \cline{4-8} \cline{8-13}
 &  & parameter  & Bias & SD & RMSE & FPR & TPR &  Bias & SD & RMSE & FPR & TPR  \\ 
  \hline
  \hline
  \multirow{3}{*}{$15$} & Strauss & $\gamma=0.2$  & 2.2 & 0.29 & 2.22 & 0 & 36 & 0.95 & 0.65 & 1.15 & 3 & 64\\ 
   & Strauss & $\gamma=0.5$  & 1.87 & 0.16 &  1.88 & 0 & 47 & 1.23 & 1.13& 1.67& 3& 54\\ 
   & Geyer  & $\gamma=1.5$  & 0.5& 0.17 & 0.53  & 0& 100 &0.11 &0.39  & 0.41&0 & 100\\ 
   \hline
  \multirow{3}{*}{$30$} & Strauss & $\gamma=0.2$  &1.76 & 0.63 & 1.87& 0 & 48 & 0.26 & 0.63 & 0.68 & 3 & 99\\ 
   & Strauss & $\gamma=0.5$ &1.61 & 0.2& 1.62& 0&50&0.22 & 0.33& 0.4& 2& 100\\ 
  & Geyer  & $\gamma=1.5$  & 0.34&0.25& 0.42&0 & 100& 0.08& 0.39&0.4 & 0&100 \\ 
  \hline
  \multirow{3}{*}{$60$} & Strauss & $\gamma=0.2$  & 1.58&0.66&1.71& 1&50& 0.22&0.44& 0.49& 3& 99\\ 
   & Strauss & $\gamma=0.5$ &1.49 & 0.27& 1.51& 1& 51&0.17 & 0.36& 0.4& 3& 100\\ 
  & Geyer  & $\gamma=1.5$ & 0.28&0.3 &0.41 & 0&100 &0.05 &0.37 & 0.37& 0& 100\\ 
 \hline
 \end{tabular}
\end{table}


\subsection{Large and small signals} 
The purpose of this section is to consider large as well as small values of the coefficients.
We consider a simulation study simpler than that of the previous sections and simulate a Strauss process with $R=9.25$ and $\psi=\log(0.2)$, which corresponds of a very repulsive model with true conditional intensity of the form $\lambda_{\mathbf{\theta}}=\exp(\beta_0 \, + \mathbf{\bbeta}^\top \mathbf{z}(u)  + \psi \, s(u,\mathbf{x}))$, where the length of $\mathbf{\bbeta}$ is fixed as in Section 6.1. We write $\mathbf{\bbeta}=(\mathbf{\bbeta}_1^\top,\mathbf{\bbeta}_2^\top)^\top$, where $\mathbf{\bbeta}_1$ and $\mathbf{\bbeta}_2$ are vectors of length $s$ and $q_k$ respectively. Recall that $q_k=p_k-s$ where $p_k= \lfloor 3 \vert W_k  \vert^{1/4} \rfloor$ for $k=1,2,3$ and $W_1=[0,250] \times [0,125]$, $W_2=[0,500] \times [0,250]$ and $W_3=[0,1000] \times [0,500]$. The $s \times 1$ vector $\mathbf{\bbeta}_1$ is set as $\mathbf{\bbeta}_1=(b,\cdots,b)^\top$, where $b \in \mathbb{R}$ and the $q_k \times 1$ vector $\mathbf{\bbeta}_2$ has all its  components zero, i.e. $\mathbf{\bbeta}_2=\mathbf{0}$. Now with this parametrization of $\mathbf{\bbeta}$, a small effect means a small value of $b$ while a large value of $b$ yields a large effect. It is worth mentioning that the setting for the coefficients in the previous simulation setups corresponds to $s=2$ and $\mathbf{\bbeta}_1=(2,0.75)^\top$. In order to address the aim of this section, we consider for each choice of $s$ the following cases: i) $b_1=0.25$,  ii) $b_2=0.5$, iii) $b_3=1$, and iv) $b_4=1.5$; which will result in eight examples for choices of $s=2$ and $s=5$. Note that in theses examples, the first two covariates are the elevation and slope covariates; the remaining covariates are generated using Scenario~\ref{sce1}. For a spatial domain $W$, we define $I_i = \int_W \exp(b_i  \, \mathbf{1}^\top z(u)) \mathrm{d}u$ and fix the intercept $\beta_0$ so that under the Poisson model, we have on average $500 \times \frac{I_i}{I_2}$ points in $W_1$, $2000 \times \frac{I_i}{I_2}$ points in $W_2$ and $4000 \times \frac{I_i}{I_2}$ points in $W_3$ where $i=1,2,3,4$. For each example, we simulate 500 point patterns from each model within spatial domains $W_1$, $W_2$ and $W_3$. We only consider the penalized pseudo-likelihood method with the adaptive lasso penalty to fit the regularization path for each simulated pattern. We use cERIC to select the tuning parameter as recommended earlier in the present paper. The number of dummy points is chosen as described in Section 6.1. Finally, we use the selection and prediction indices defined in Section 6.1 to measure the estimation accuracy and gauge the variable selection performance. \newline
Figure  \ref{fig:ls.signal} reports RMSEs and TPRs of the regularized pseudo-likelihood estimator under the ALasso-cERIC method for large as well as small values of the coefficients. It can be seen that when the spatial domain grows and the number of points increases, RMSEs tend to diminish, which is expected according to the consistency of the regularized pseudo-likelihood estimator for Gibbs point process models. Moreover, TPRs are improved as the magnitude of the coefficients and the size of the spatial domain are getting larger. That is also expected and it means that our procedure tends to recover strong signals more frequently than weak signals. Note that FPRs are also computed; they are low all the time with a maximum value of $4\%$.

\begin{figure}[!ht]
\begin{center}
\renewcommand{\arraystretch}{0}
\includegraphics[width=1\textwidth]{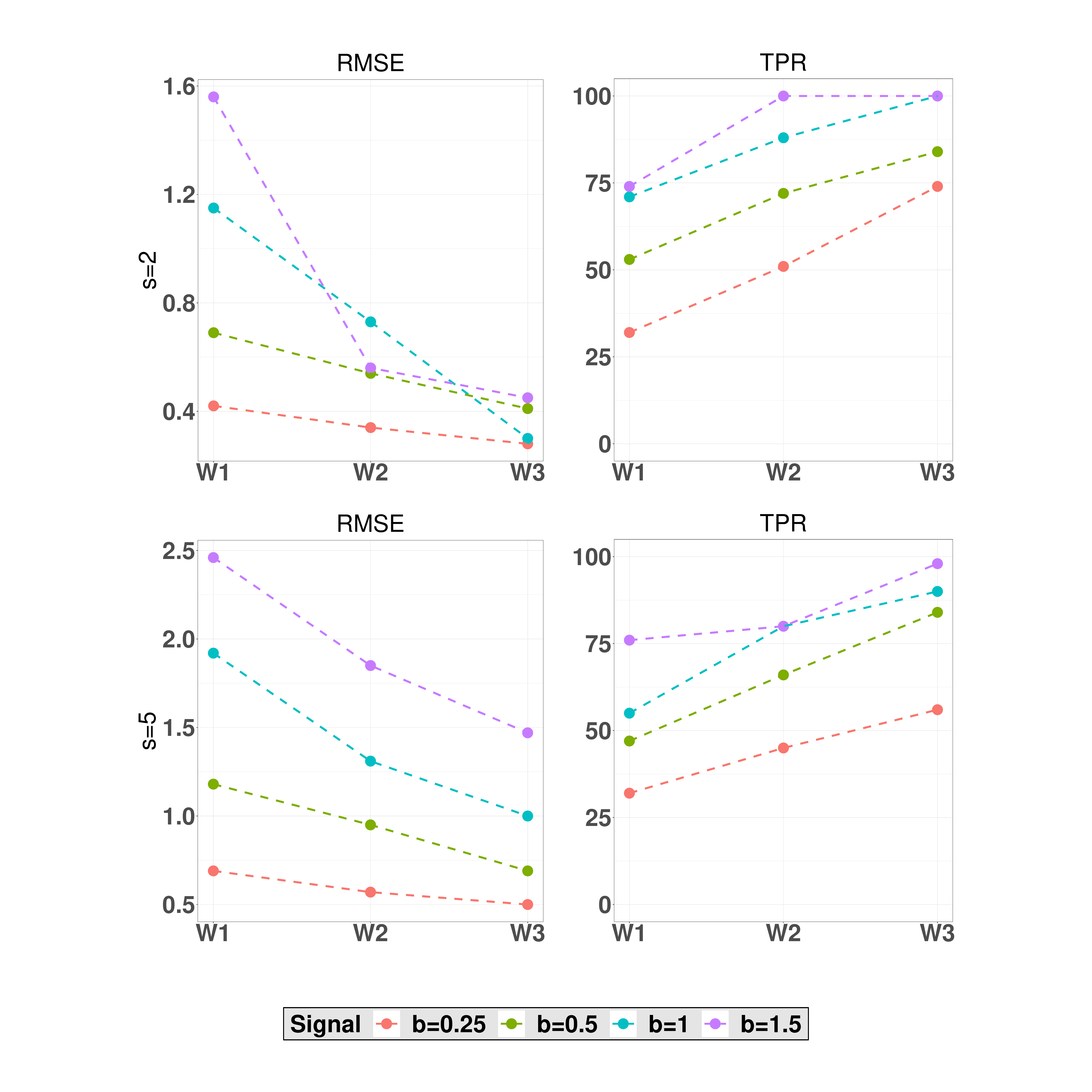} 
\caption{Empirical RMSE, TPR in terms of the observation window  for different magnitude of the coefficients: $b=0.25$, $b=0.5$, $b=1$, and $b=1.5$. Results are based on 500 replications of inhomogeneous Strauss model with $\gamma=0.2$ using ALasso-cERIC.}
\label{fig:ls.signal}
\end{center}
\end{figure}

\section{Application to a real data set} \label{sec:data}

In order to illustrate our approach, we consider the well-known data obtained from comprehensive census program conducted on the tropical rainforest of the Barro Colorado Island (BCI), Panama. Censuses are carried out on a site of $50$ hectares ($W=1000 m \, \times \, 500m$) where all freestanding woody stems at least $1 \, cm$ diameter at breast height were identified, measured, and mapped. This results to a very rich dataset consisting of records of more than $350,000$ individual trees of over $300$ species \cite[see][]{condit1998tropical,hubbell1999light,hubbell2005barro}. The focus here is to analyze the locations of $3604$ individual trees of the species Beilschmiedia pendula (Lauraceae). So, it seems interesting to know how this high number of trees profits from environmental factors such as topography, soil properties, etc. Also, regarding the relatively large number of environmental factors, an important question should concern the selection of some covariates among them as well as the estimation of their coefficients. The spatial distribution of B. pendula has been extensively analysed using spatial point processes in both the unregularized~\cite[][]{moller2007modern,waagepetersen2007estimating,hengl2009spatial,guan2010weighted} and regularized settings~\cite[][]{thurman2014variable,thurman2015regularized,choiruddin2018convex,daniel2018penalized}. Inhomogeneous Poisson process models, Cox point process models and area-interaction point process model were used in these analyses to describe spatial dependence among trees. In our analysis, we use a Geyer saturation point process model to describe spatial dependence among trees and model the conditional intensity of B. pendula trees as a log-linear function of $93$ covariates consisting of $2$ topological attributes, $13$ soil nutrients, and $78$ interactions between two soil nutrients.  \newline
We precisely fit a regularized Geyer saturation model with threshold $\sigma=1$ and interaction radius $r=10\,m$ to the B. pendula data using the PLCL method with adaptive lasso and adaptive elastic net penalties, hereafter referred to respectively as the PLCL-ALASSO and PLCL-AENET methods. The tuning parameter is chosen using cERIC and the number of dummy points is set at the default number of that of $\mathsf{ppm}$ function in the $\mathsf{spatstat}$ package, i.e. approximately $14400$ dummy points. All covariates are centered and scaled in order to identify those that have a significant effect on the conditional intensity. \newline
Table~\ref{table:app} reports the parameter estimates for the covariates selected by the PLCL-ALASSO  and PLCL-AENET methods. Out of $93$ covariates, $14$ covariates are selected using the PLCL-AENET method while only $8$ using the PLCL-ALASSO method. Note that the 8 covariates selected under ALASSO penalty are among the 14 selected under AENET penalty. Although the magnitudes of the parameter estimates for the common selected covariates are slightly different, the signs all agree with each other. Both models select a positive interaction term, which indicates clustering among trees. \newline
These results suggest that B. pendula conditional intensity is positively associated with elevation and slope: trees appear to grow more densely in  areas of higher elevation and slope, and negatively associated with soil pH: trees are more likely to live in areas with low concentration of pH in the soil. Furthermore, the appearance of B. pendula trees is positively associated with the interaction between aluminum and copper concentration. The rest of the common selected features are interactions between covariates which present a negative association with the conditional intensity of the B. pendula, for e.g. the interaction between aluminum and mineralized nitrogen, etc. We find some differences in estimation and selection between the PLCL-ALASSO and PLCL-AENET methods. First, the PLCL-AENET method generally has larger estimators, except for the interaction between phosphorus concentration and soil pH where the estimate obtained with the PLCL-ALASSO method has a higher value. Second, we lose some non-zero covariates under the PLCL-AENET method when we use the PLCL-ALASSO method, even though for most of these covariates the estimates are relatively small. Note that we use the absolute value for the comparison of the estimates. The maps of B. pendula trees as well as the fourteen covariates selected using the PLCL-AENET method are reported in Figure~\ref{fig:app}.\newline
In summary, we prefer the PLCL-AENET method to the PLCL-ALASSO method in the basis of the selection of soil pH covariate. From Figure~\ref{fig:app}, one can easily see that soil pH is higher in the centre of the plot and slightly lower on the left slope of the plot, which means a negative association with the appearance of B. pendula trees as noticed with the PLCL-AENET method.     

\setlength{\tabcolsep}{6pt}
\renewcommand{\arraystretch}{1.5}
\begin{table}[!ht]
\caption{List of covariates selected by Geyer model fitted via penalized logistic composite likelihood with adaptive lasso and adaptive elastic net regularization using cERIC.}
\label{table:app}
\begin{center}
\begin{tabular}{ l @{\hskip 1.5in} c @{\hskip 1in}c  }
\hline
\hline
  & \multicolumn{2}{c}{\mbox{PLCL - Geyer(r=10, sat=1)}} \\
    \cline{2-3}
Covariates & ALASSO & AENET \\
  \hline
\hline
Elev  & 0.177  & 0.342  \\ 
  Slope  & 0.237  & 0.316   \\ 
  Al  & 0  & 0.01    \\ 
  pH   & 0  & -0.013  \\ 
  Al $\times$ Cu  & 0.107  & 0.251      \\ 
  Al $\times$ Fe & 0  & 0.026    \\ 
  Al $\times$ P & -0.057  & -0.218  \\ 
  Al $\times$ Zn  & -0.179 & -0.458  \\ 
  Al $\times$ N.min  & -0.041  & -0.157 \\ 
  Cu $\times$ Mg & 0  & -0.128   \\ 
  Fe $\times$ Mn  & 0 & 0.144   \\ 
  P $\times$ N  & -0.021  & -0.345  \\ 
  P $\times$ pH & -0.267 & -0.123   \\ 
  N $\times$ N.min  & 0  & 0.055    \\ 
\hline
Interaction  & 0.383 &  0.295\\
\hline
\end{tabular}
\end{center}
\end{table}

\begin{figure}[!ht]
\begin{center}
\renewcommand{\arraystretch}{0}
\begin{tabular}{l @{\hskip 0.08in} l @{\hskip 0.08in} l @{\hskip 0.08in} l @{\hskip 0.08in} l}
\includegraphics[width=0.185\textwidth]{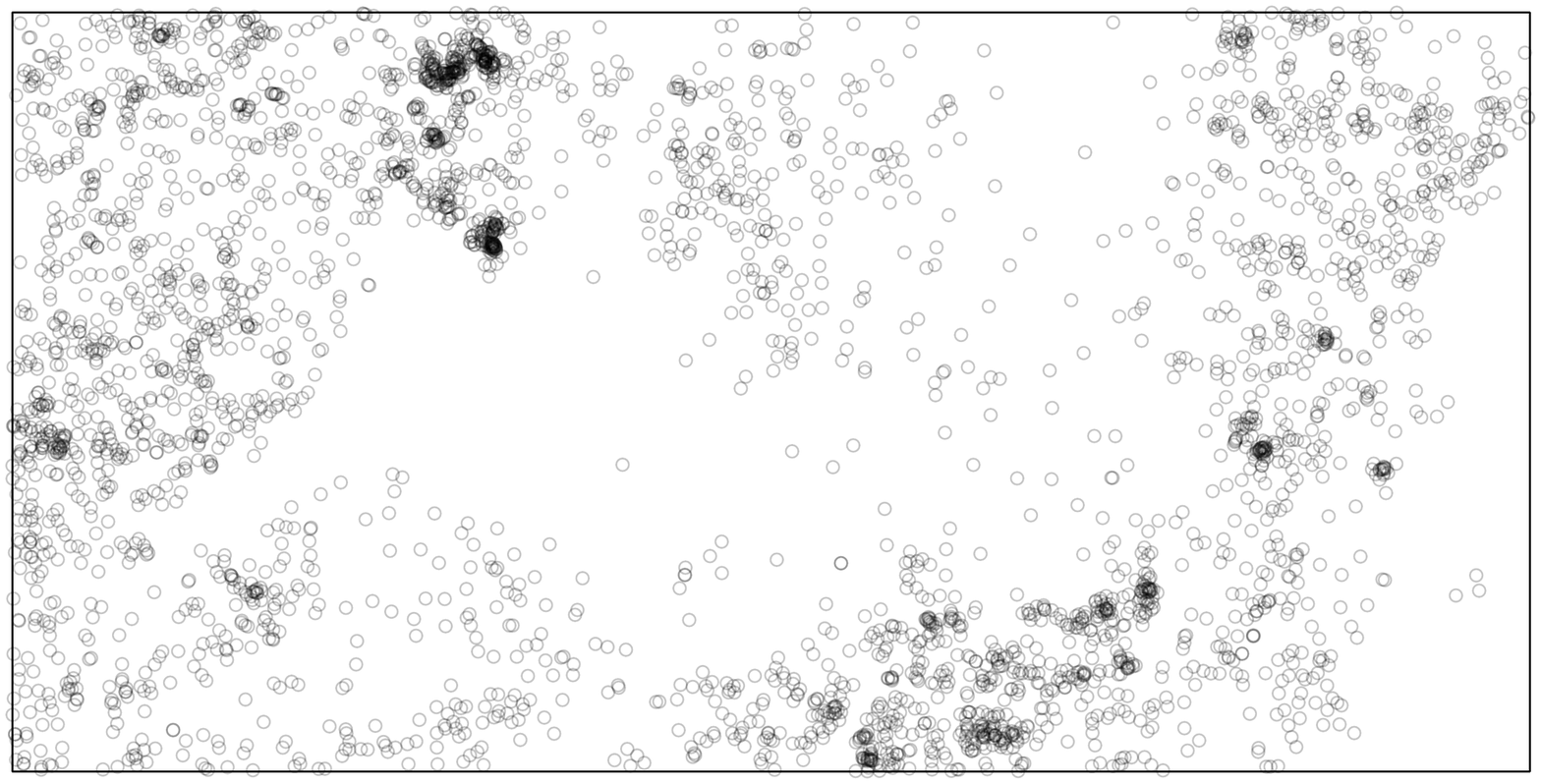} & \includegraphics[width=0.185\textwidth]{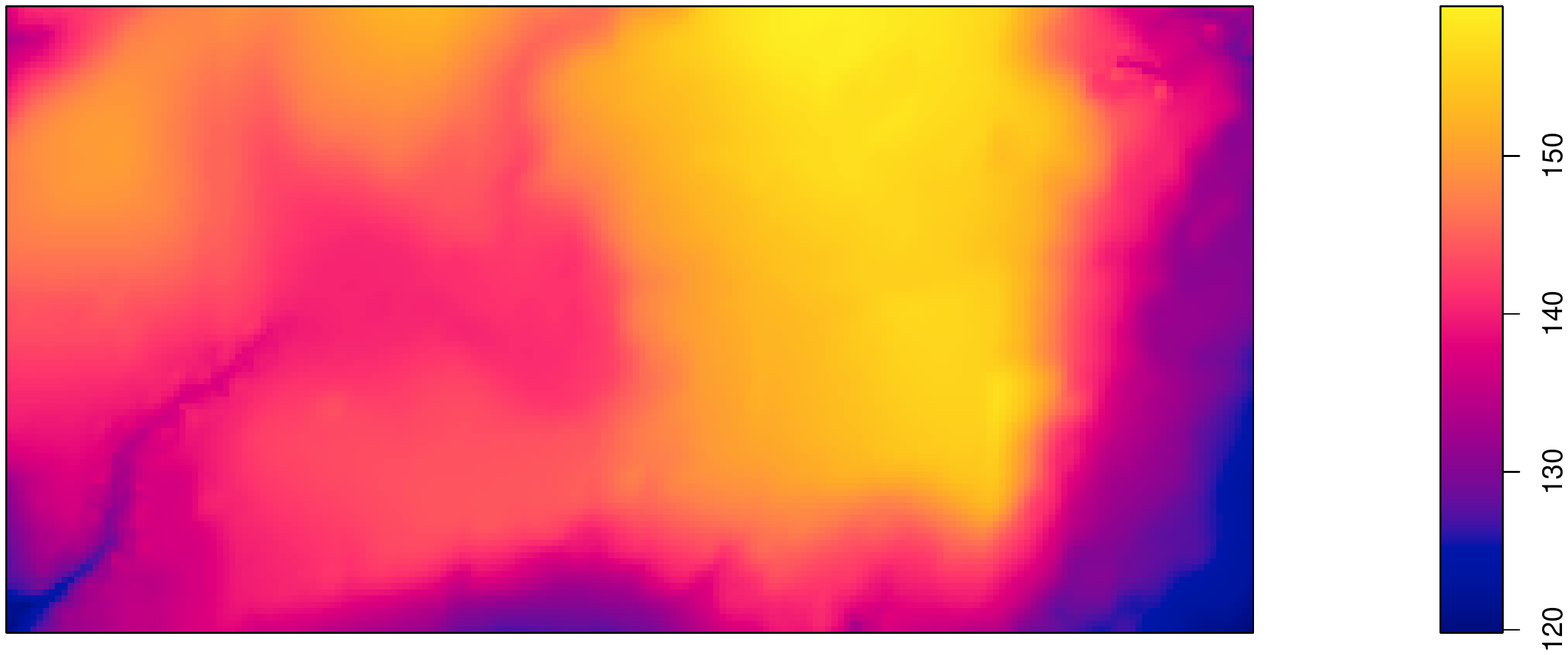} & \includegraphics[width=0.185\textwidth]{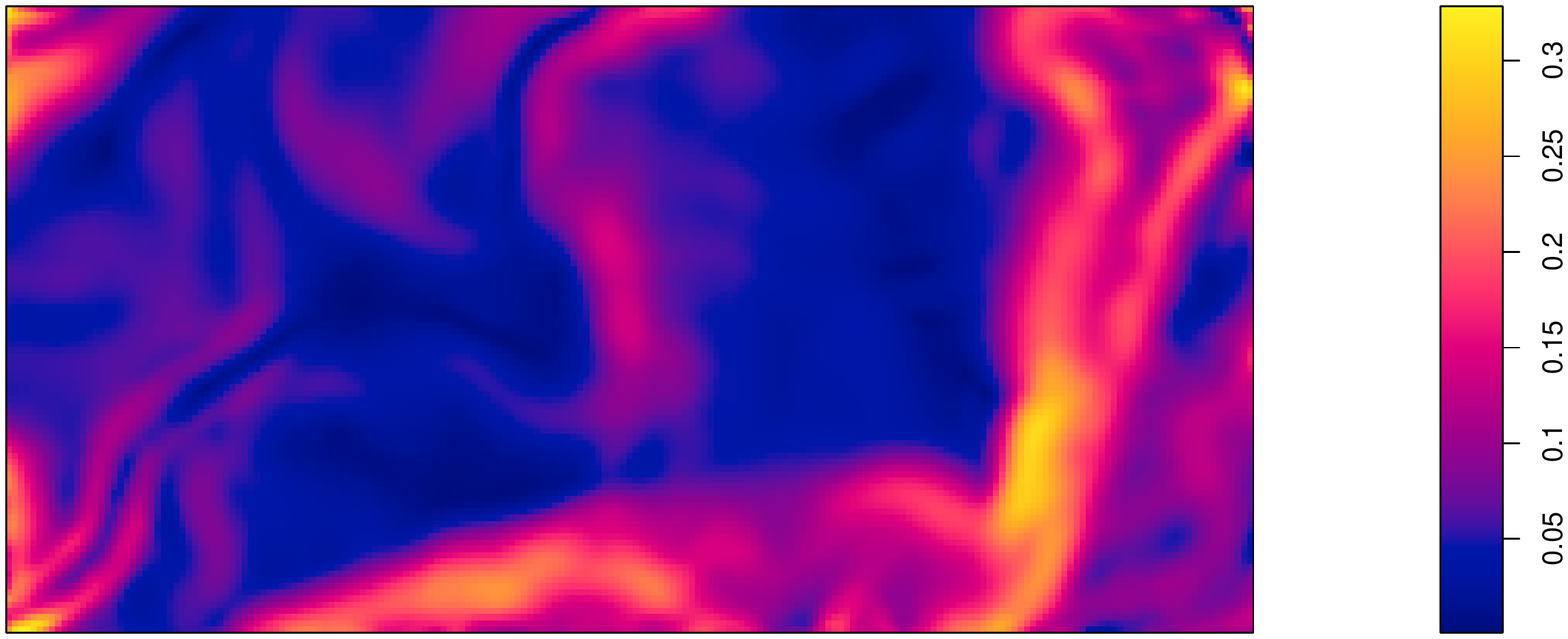} & \includegraphics[width=0.185\textwidth]{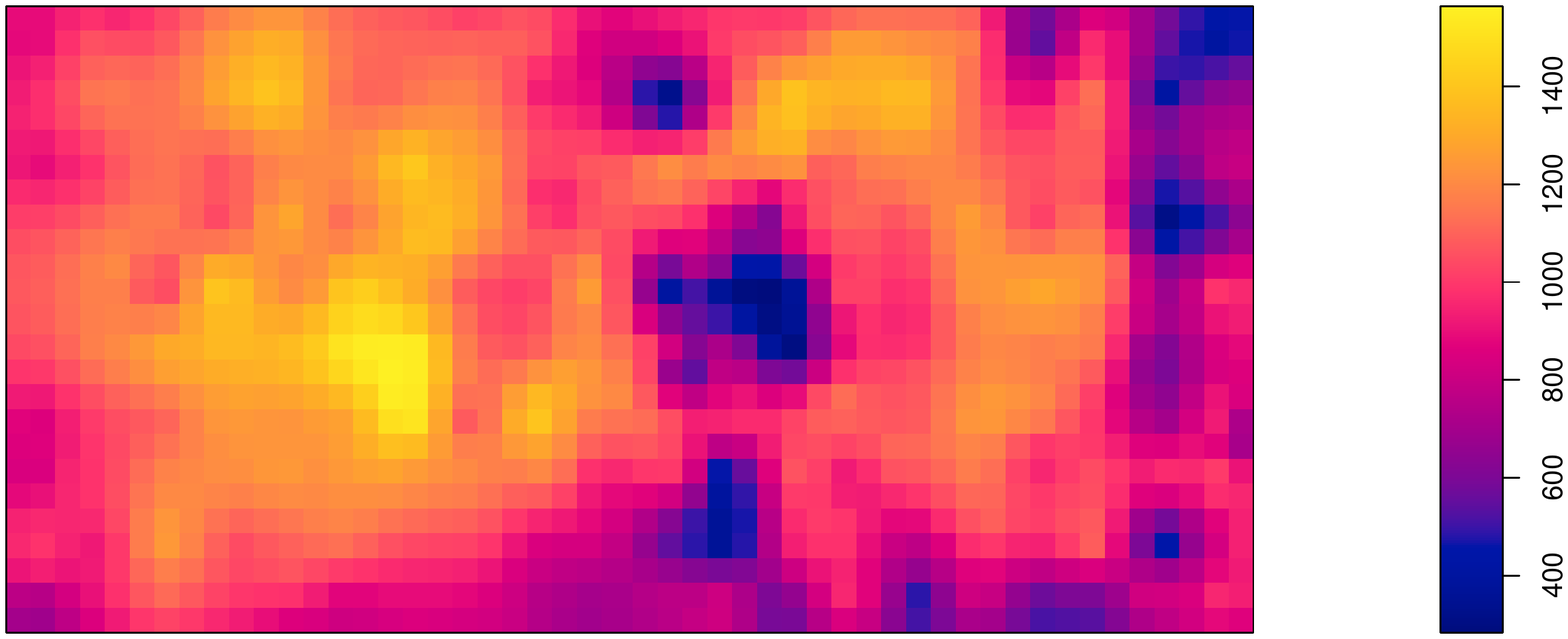} & \includegraphics[width=0.185\textwidth]{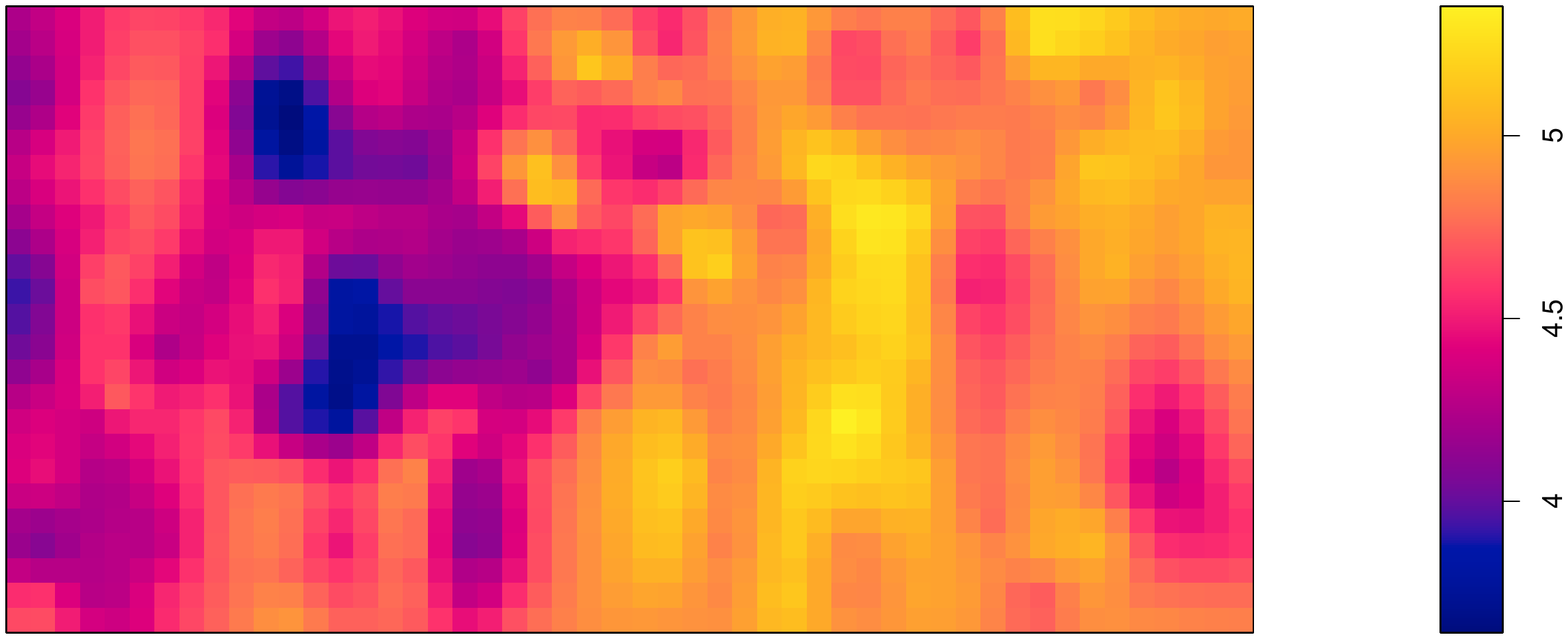} \\[-.3cm]
\includegraphics[width=0.185\textwidth]{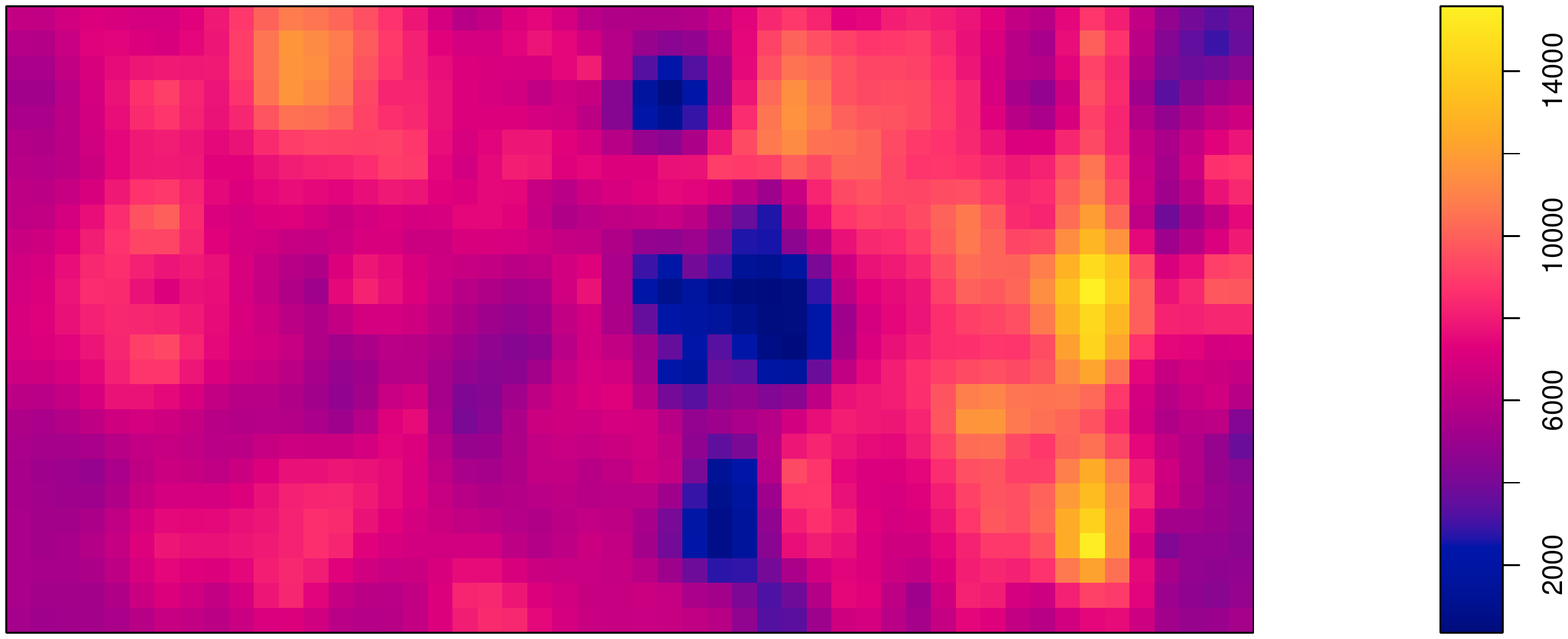} & \includegraphics[width=0.185\textwidth]{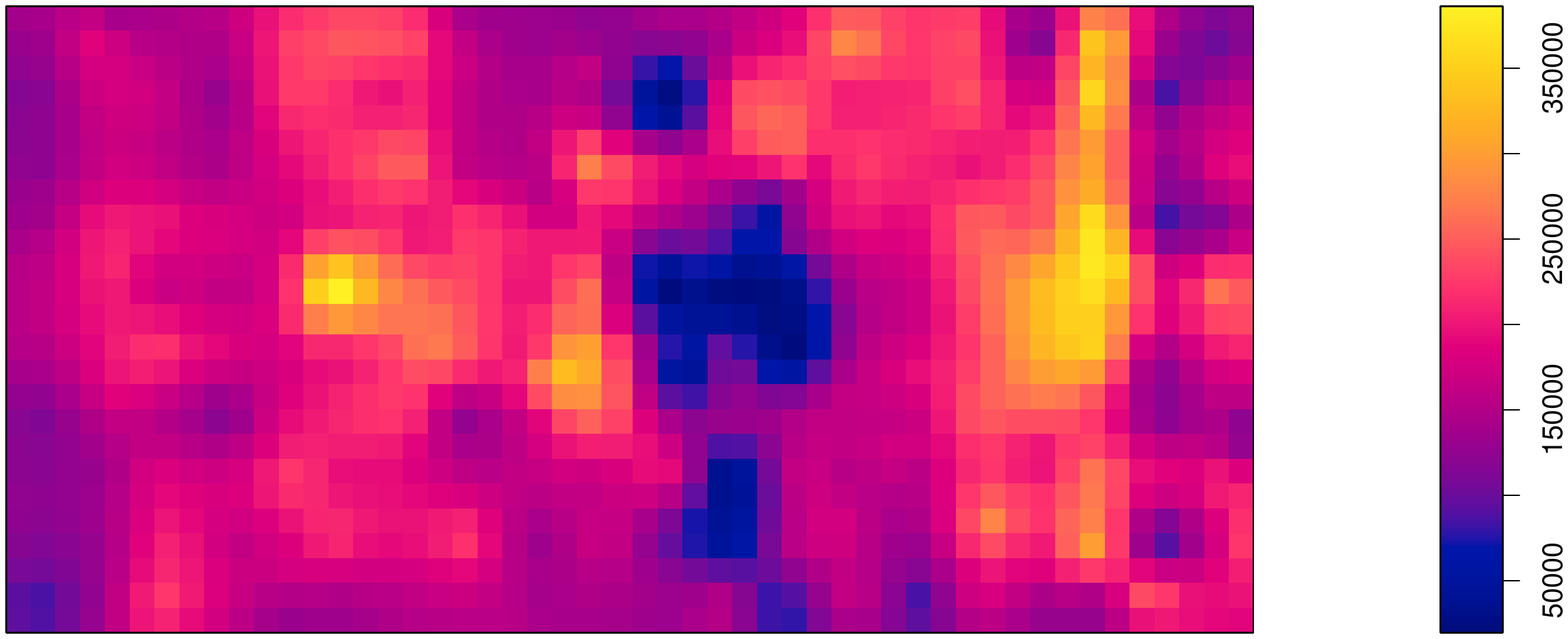} & \includegraphics[width=0.185\textwidth]{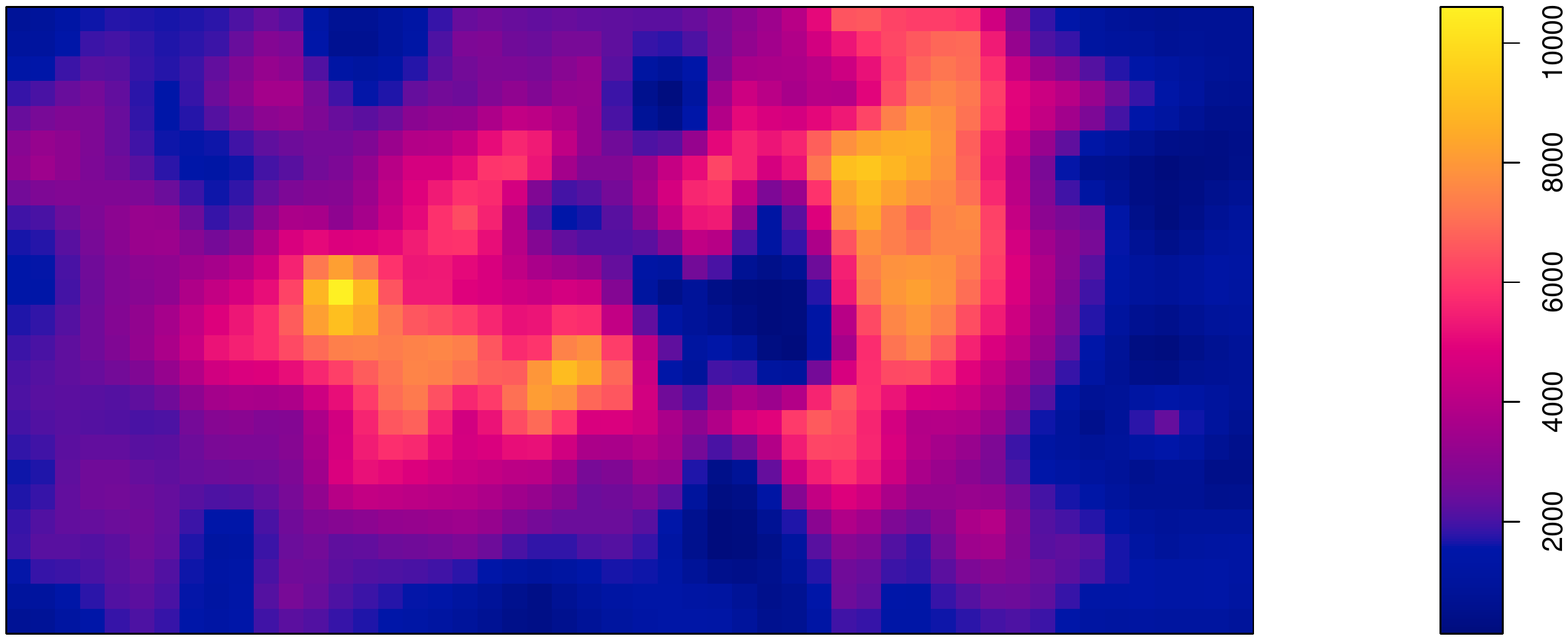} & \includegraphics[width=0.185\textwidth]{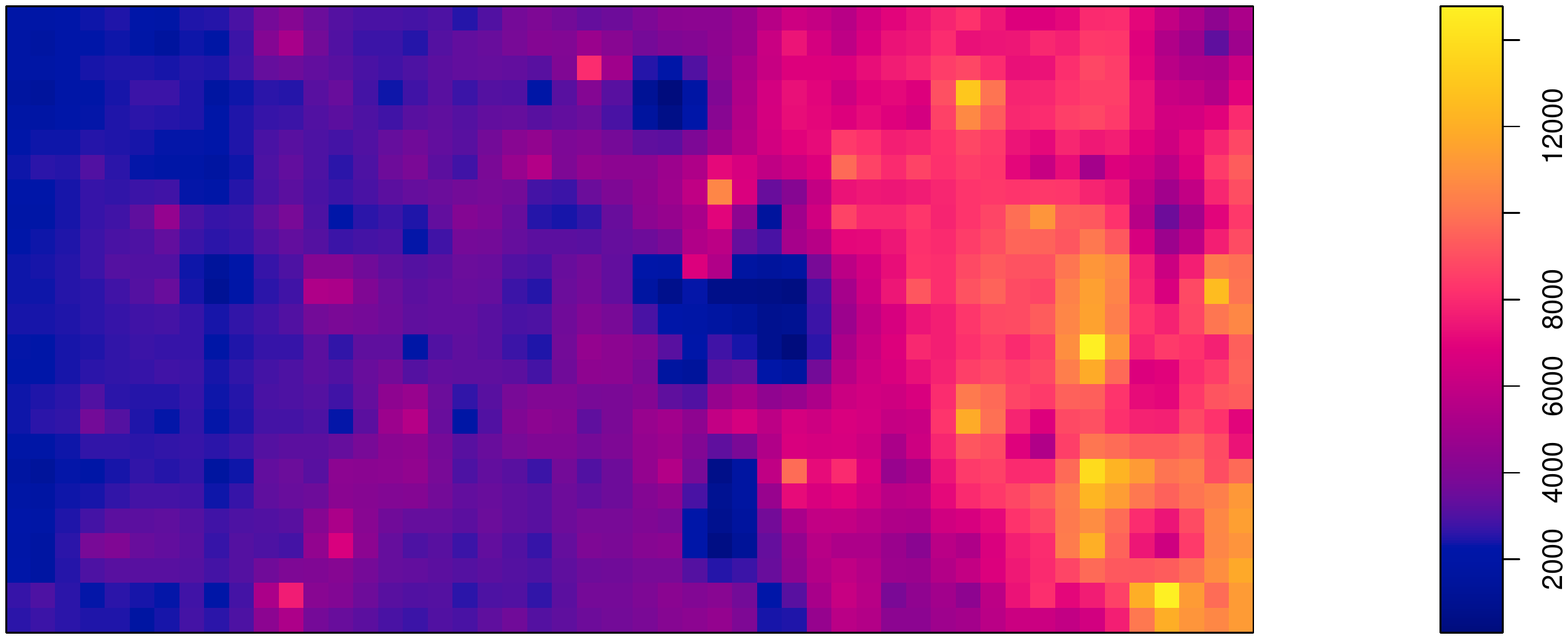} & \includegraphics[width=0.185\textwidth]{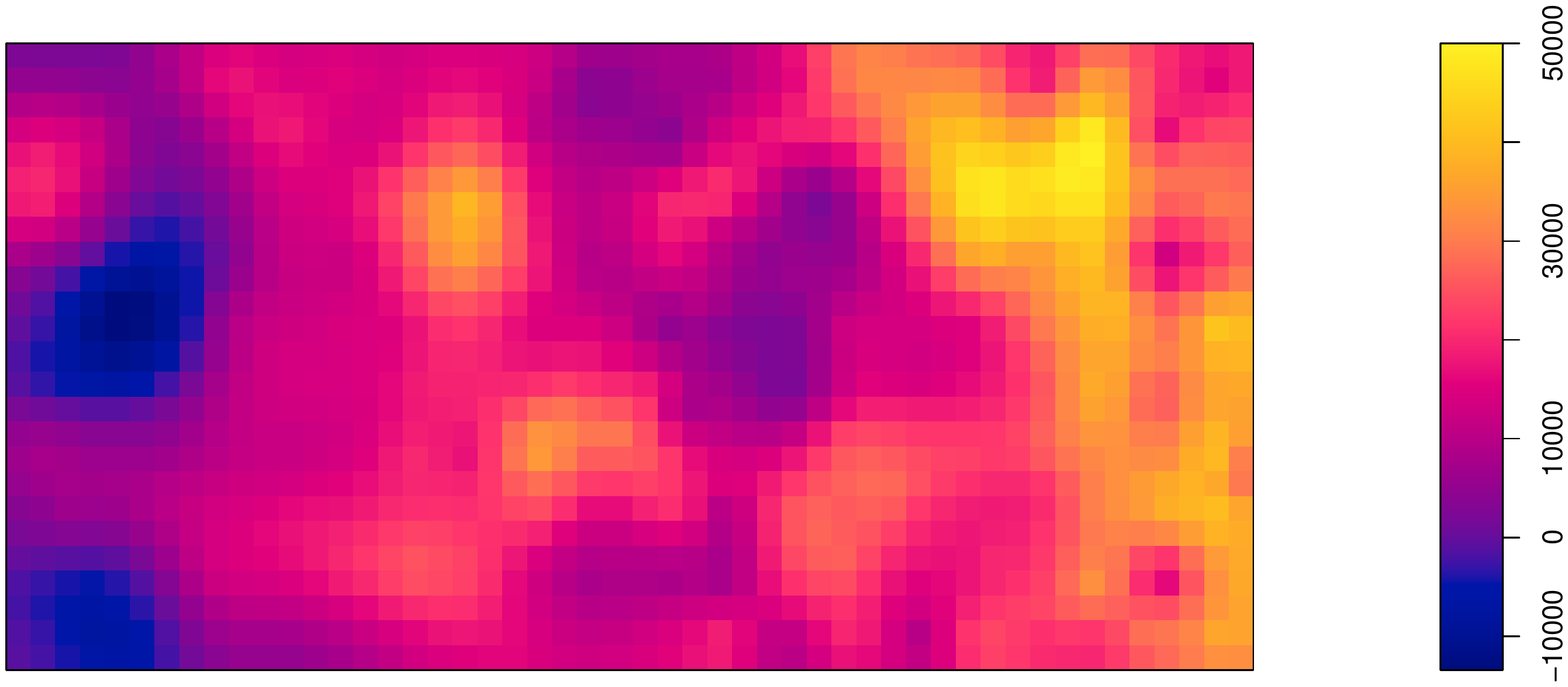} \\[-.5cm]
\includegraphics[width=0.185\textwidth]{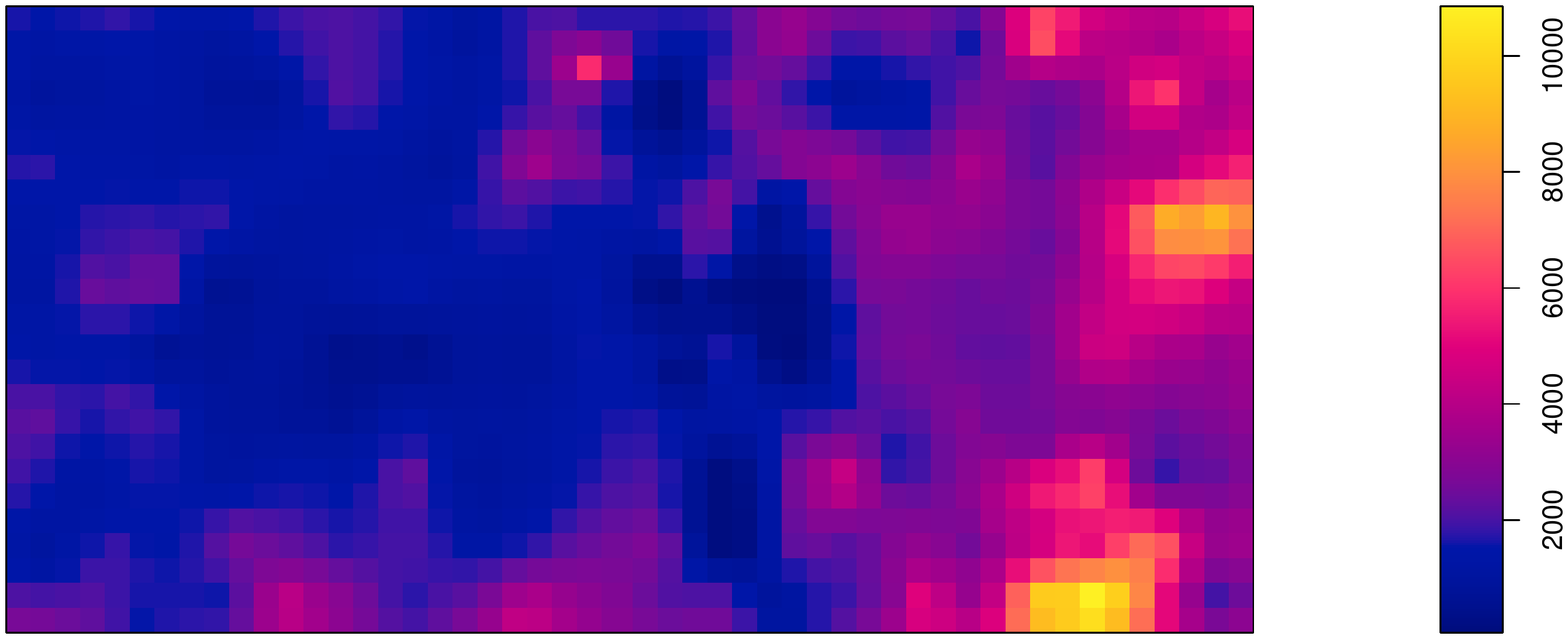} & \includegraphics[width=0.185\textwidth]{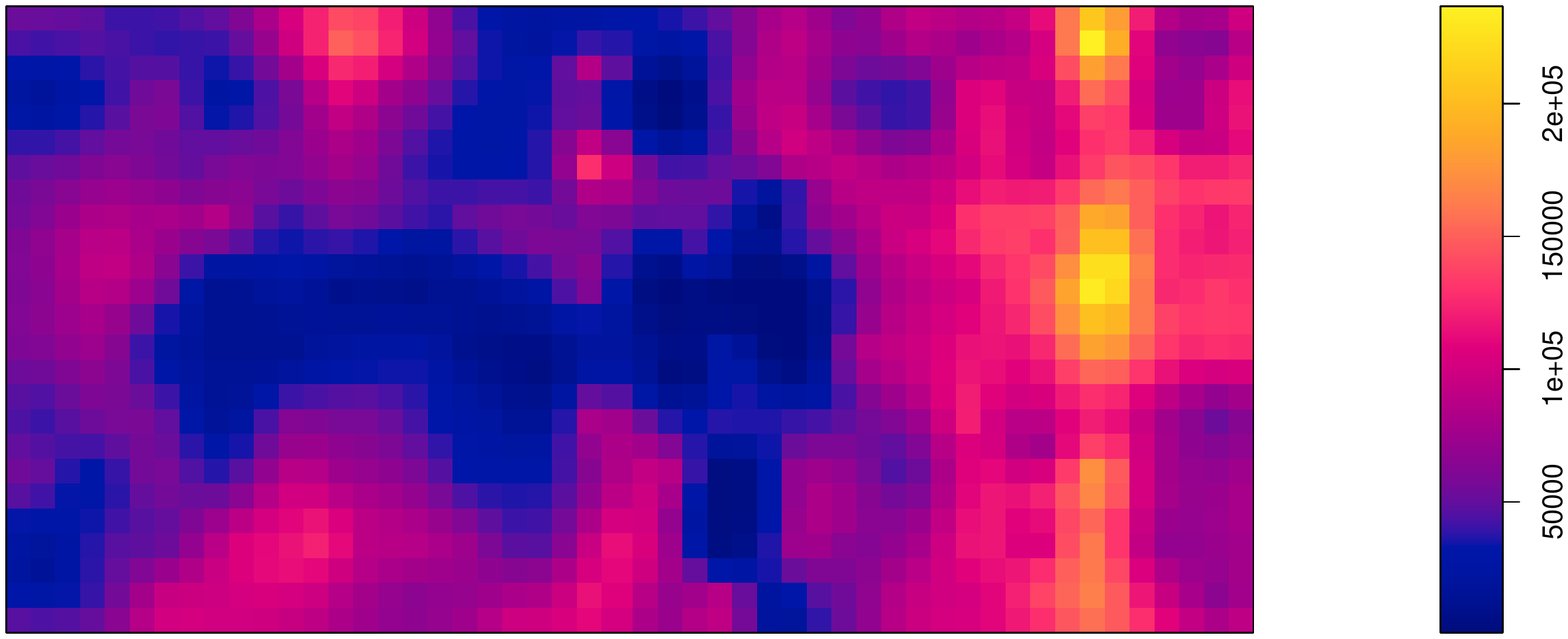} & \includegraphics[width=0.185\textwidth]{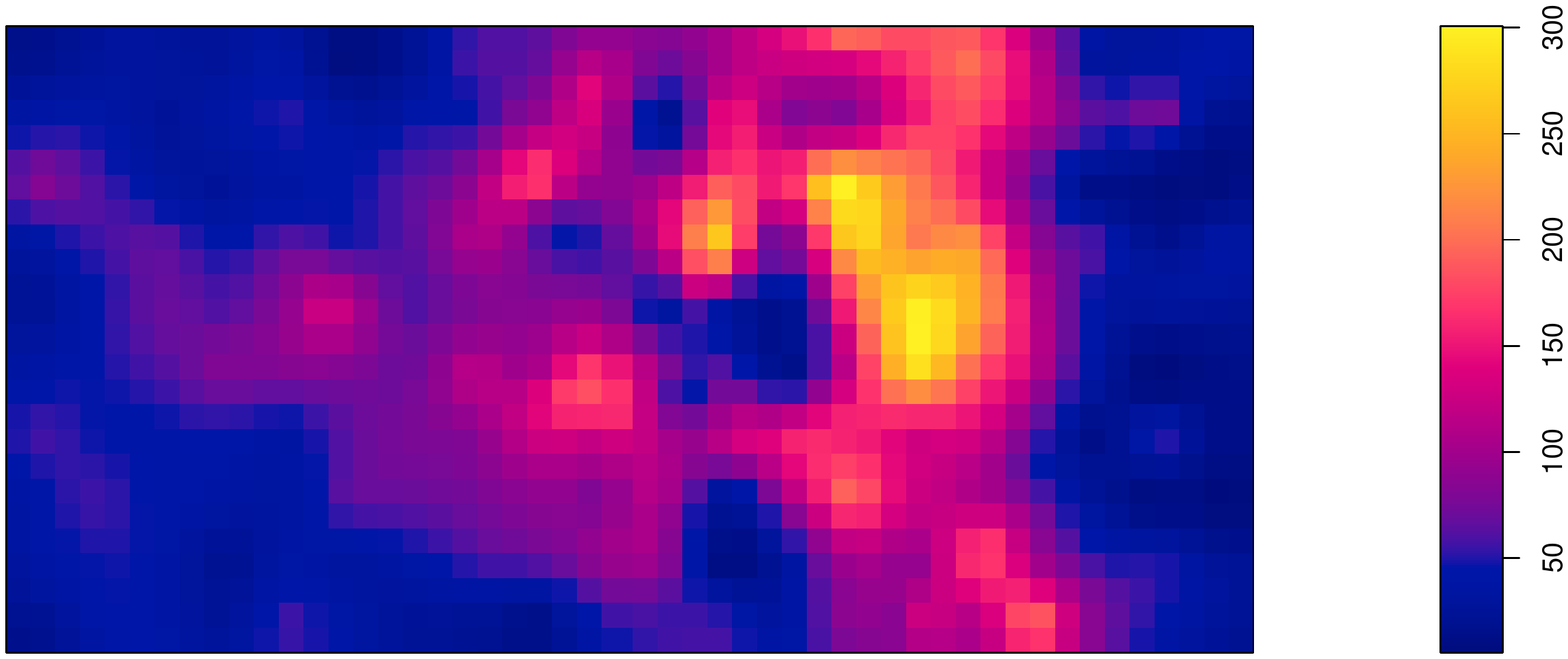} & \includegraphics[width=0.185\textwidth]{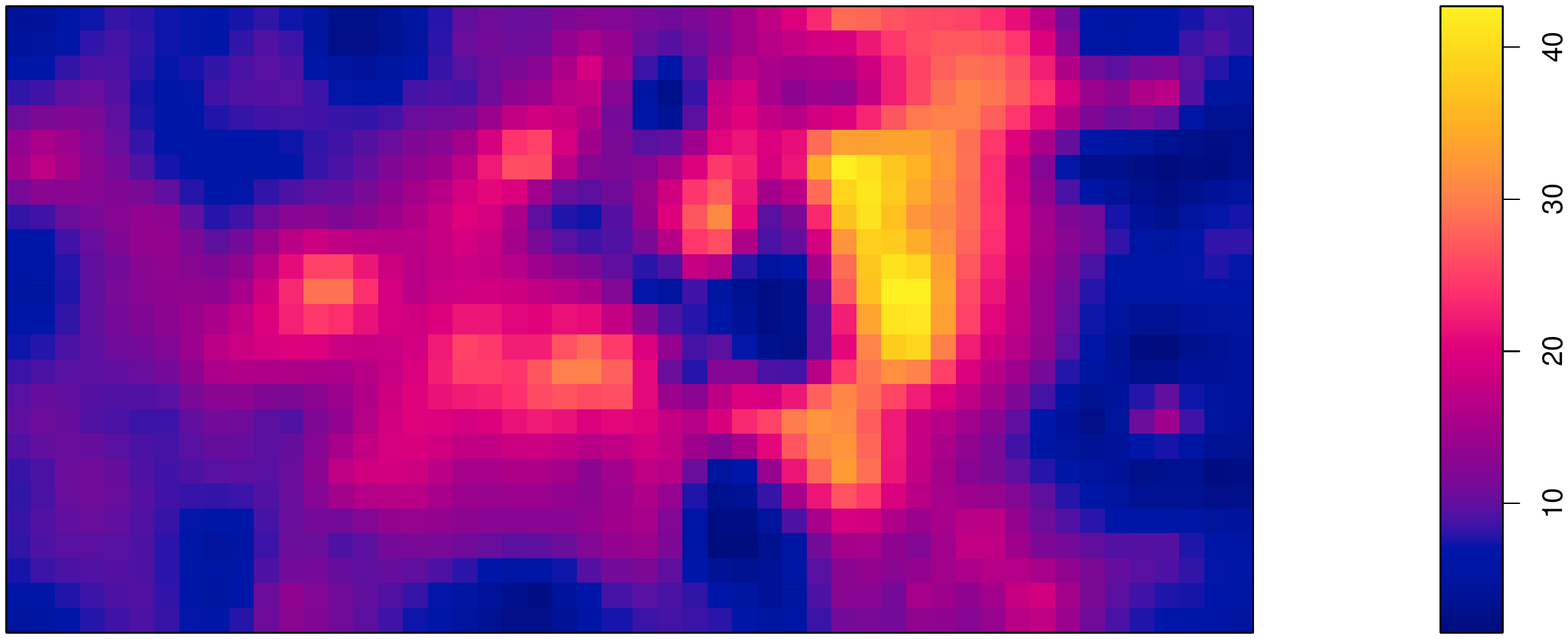} & \includegraphics[width=0.185\textwidth]{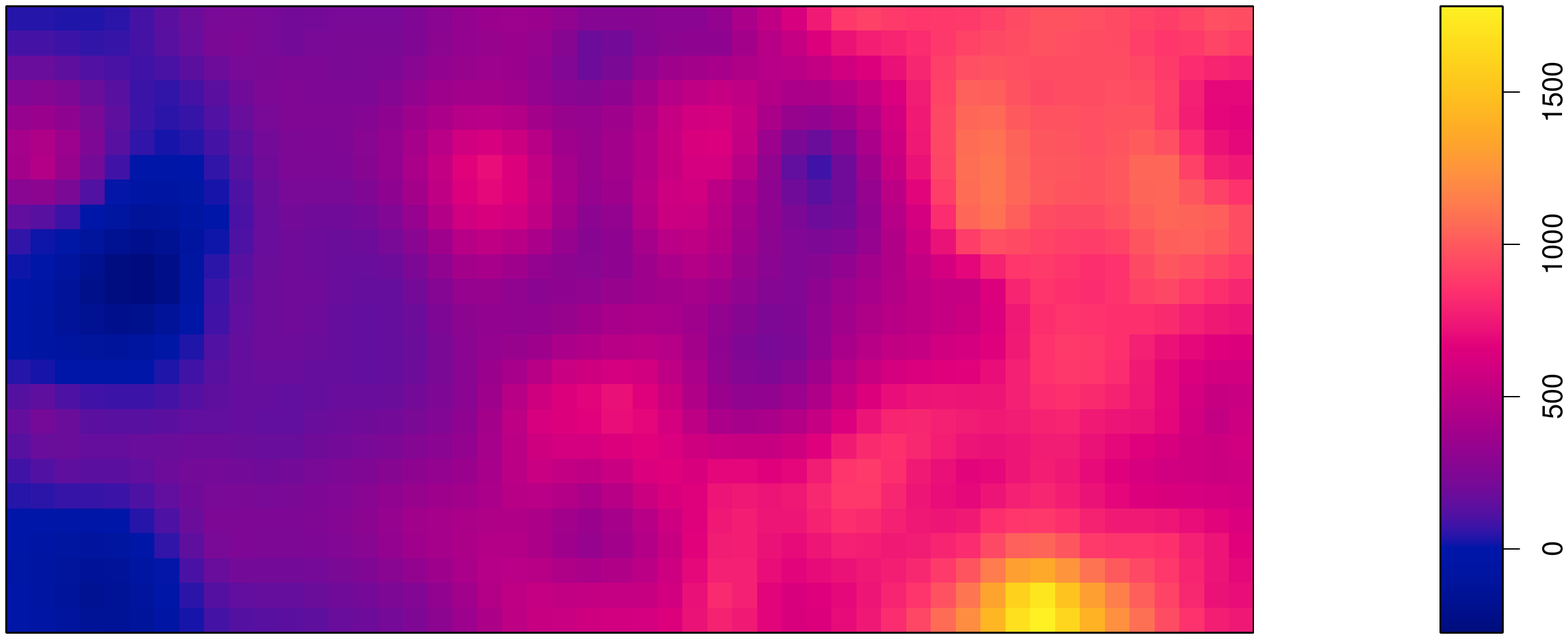}
\end{tabular}
\caption{Maps of locations of BPL trees and the fourteen (14) selected covariates using the PLCL-AENET method. From left to right: BPL trees, elevation, slope, Aluminum and pH (1st row), interaction between Aluminum and Copper, between Aluminum and Iron, between Aluminum and Phosphorus, between Aluminum and Zinc, and between Aluminum and Mineralized Nitrogen (2nd row), interaction between Copper and Magnesium, between Iron and Manganese, between Phosphorus and Nitrogen, Phosphorus and pH, and between Nitrogen and Mineralized Nitrogen (3rd row).}
\label{fig:app}
\end{center}
\end{figure}

\section{Discussion and conclusion}\label{sec:discussion}

In this paper, we develop regularization methods for inhomogeneous GPPs via PPL and PLCL approaches. We impose convex as well as non-convex penalties on the composite likelihoods in order to perform variable selection and parameter estimation simultaneously in a setting where the number of covariates diverges as the volume of the observation window increases. The implementation in \textsf{R} of our procedure is quite simple as it consists to implement penalized generalized linear models. This is done in \textsf{R} by combining \textsf{spatstat} package with \textsf{glmnet} and \textsf{ncvreg} packages. We propose a new composite information criterion, cERIC, which extends classical composite information criteria to account explicitly the effect of the tuning parameter. Simulations show that cERIC outperforms a competing criterion, cBIC for choosing the tuning parameter; and that PLCL outperforms PPL in terms of selection properties and predictive performance. These simulation results are obtained by using adaptive lasso and adaptive elastic net as regularization methods for both PPL and PLCL. Simulations are also carried out in order to compare the selection and prediction performances of the pseudo-likelihood function regularized with convex and non-convex penalties. It follows from these simulations that adaptive lasso, adaptive elastic net, SCAD and MC+ methods be used to penalize the pseudo-likelihood function of a spatial GPP  when the observed point pattern is quite clustered and the covariance matrix of the covariates has a complex structure. Otherwise, we recommend the convex penalties. From a theoretical point of view, we provide general conditions on the spatial GPP to ensure its existence and some more conditions on the penalty function to establish asymptotic properties of the regularized pseudo-likelihood estimator in terms of sparsity, consistency and asymptotic normality. These results are the object of two theorems in the present paper and also hold in an unregularized setting, that is the pseudo-likelihood estimator for a non-stationary exponential Gibbs model is consistent and asymptotically normal. Adaptive lasso, adaptive elastic net, SCAD and MC+ are among the only regularization methods we consider which satisfy the two theorems mentioned above. We apply our procedure by fitting a Geyer saturation model to the B. pendula data obtained from census programs on the Barro Corrolado Island via PLCL combined with adaptive lasso and adaptive elastic net penalties. Although both fitted models are parsimonious, PLCL under adaptive elastic net penalty exhibits better model fit.   


\appendix




\section{Auxiliary Lemma}  \label{sec:auxLemma}

The following result is used in the proof of Theorems~\ref{THM:ROOT}-\ref{THM:SPARSITYCLT}. Throughout the proofs, the notation $\mathbf X_{{n}} = O_{\mathrm P} (x_n)$ or $\mathbf X_{{n}} = o_{\mathrm P} (x_n)$ for a random vector $\mathbf X_n$ and a sequence of real numbers $x_n$ means that $\|\mathbf X_n\|=O_{\mathrm P}(x_n)$ and $\|\mathbf X_n\|=o_{\mathrm P}(x_n)$. In the same way for a vector $\mathbf V_n$ or a squared matrix $\mathbf M_n$, the notation $\mathbf V_n=O(x_n)$ and  $\mathbf M_n=O(x_n)$ mean that $\|\mathbf V_n\|=O(x_n)$ and  $\|\mathbf M_n\|=O(x_n)$, where the matrix norm corresponds to the spectral norm.

\begin{lemma} \label{lemma1}
Under the conditions ($\mathcal C$.\ref{C:Dn})-($\mathcal C$.\ref{C:locsta}), the following result holds  as $n\to~\infty$
\begin{align}
\mathbf{LPL}^{(1)}_{n}( \bX;\btheta_0) =  O_{\mathrm P}(\sqrt{p_n |D_n|})\label{eq:normal}.
\end{align}
\end{lemma}

\begin{proof}
By~\citet[Lemma~3.1]{coeurjolly2013fast}, we have
\begin{align*}
\Var[ \mathbf{LPL}^{(1)}_{n}( \bX;\btheta_0)]= \mathbf{A}_n(\btheta_{0})+\mathbf{B}_n(\btheta_{0})
\end{align*}
where $\mathbf{A}_n(\btheta_{0})$ and $\mathbf{B}_n(\boldsymbol \theta_{0})$ are defined in (\ref{sec:not}). Now,
\begin{align*}
\| \mathbf{A}_n(\boldsymbol \theta_{0})\| \leq {\int_{D_n} \|\EE [\bS(u,\bX)\bS(u,\bX)^\top \lambda_{\btheta_{0}}(u,\bX)] \|\mathrm{d}u} \\
+ {\int_{D_n} \|\EE [\bZ(u)\bZ(u)^\top \lambda_{\btheta_{0}}(u,\bX)] \| \mathrm{d}u}.
\end{align*}
Conditions ($\mathcal C$.\ref{C:Theta})-($\mathcal C$.\ref{C:locsta}) imply that 
\[
 \mathbf{A}_n(\btheta_{0})=O(|D_n|) + O(p_n |D_n|) = O(p_n |D_n|).
\]
We rewrite $ \mathbf{B}_n(\btheta_{0}) =  \mathbf{I}_n(\boldsymbol \theta_{0}) +  \mathbf{J}_n(\boldsymbol \theta_{0})$ where 

\[
 \mathbf{I}_n(\boldsymbol \theta_{0}) \hspace{-.11cm} = \hspace{-.11cm} \EE \left[ {\int_{D_n} \hspace{-.1cm} \int_{D_n}  \hspace{-.3cm} \mathbf{t}(u,\bX)\mathbf{t}(v,\bX)^\top  \Big(  \lambda_{\boldsymbol{\theta}_{0}}(u,\bX) \lambda_{\boldsymbol{\theta}_{0}}(v,\bX) - \lambda_{\boldsymbol{\theta}_{0}}(\{u,v\},\bX)  \Big) \mathrm{d}v \mathrm{d}u} \right]
\] and 
\[
 \mathbf{J}_n(\boldsymbol \theta_{0})= \EE \left[ {\int_{D_n} \int_{D_n} \Delta_v \mathbf{t}(u,\bX) \Delta_u \mathbf{t}(v,\bX)^\top \lambda_{\boldsymbol{\theta}_{0}}(\{u,v\},\bX) \mathrm{d}v \mathrm{d}u} \right].
\]
The finite range property ensured by ($\mathcal C$.\ref{C:Theta}) gives us:
\begin{align*}
\lambda_{\boldsymbol{\theta}_0}(v,\bX)&=\lambda_{\boldsymbol{\theta}_0}(v,\bX \cap B(v,R)), \\
\lambda_{\boldsymbol{\theta}_0}(v,\bX \cup u)&=\lambda_{\boldsymbol{\theta}_0}(v,[\bX \cap B(v,R)] \cup [\{u\} \cap B(v,R)]) \quad  \\
\mathbf{t}(v,\bX \cup u)&=\mathbf{t}(v,[\bX \cap B(v,R)] \cup [\{u\} \cap B(v,R)]).
\end{align*}
For any $v \in D_n \setminus B(u,R)$
\[
\lambda_{\boldsymbol{\theta}_0}(v,\bX \cup u)=\lambda_{\boldsymbol{\theta}_0}(v,\bX)
\quad \mbox{and} \quad
\mathbf{t}(v,\bX \cup u) = \mathbf{t}(v,\bX),\mbox{ that is}
\]

\[
\lambda_{\btheta_0}(u,\bX) \lambda_{\btheta_0}(v,\bX) = \lambda_{\btheta_0}(\{u,v\},\bX) \quad \mbox{and}  \quad \Delta_u \mathbf{t}(v,\bX) = 0.
\]
Hence 
\[
\hspace{-.15cm}   \mathbf{I}_n(\boldsymbol \theta_{0}) \hspace{-.09cm}=\hspace{-.09cm} \EE  \left[ \hspace{-.11cm} {\int_{D_n} \hspace{-.11cm} \int_{D_n \cap B(u,R)} \hspace{-.98cm}  \mathbf{t}(u,\bX)\mathbf{t}(v,\bX)^\top  \hspace{-.11cm}  \Big(  \lambda_{\boldsymbol{\theta}_{0}}(u,\bX) \lambda_{\boldsymbol{\theta}_{0}}(v,\bX)  \hspace{-.11cm} - \hspace{-.11cm} \lambda_{\boldsymbol{\theta}_{0}}(\{u,v\},\bX) \Big)  \mathrm{d}v \mathrm{d}u} \right] 
\]  and 

\[
 \mathbf{J}_n(\boldsymbol \theta_{0})= \EE \left[ {\int_{D_n} \int_{D_n \cap B(u,R)} \Delta_v \mathbf{t}(u,\bX) \Delta_u \mathbf{t}(v,\bX)^\top \lambda_{\boldsymbol{\theta}_{0}}(\{u,v\},\bX) \mathrm{d}v \mathrm{d}u} \right].
\]
This implies from conditions ($\mathcal C$.\ref{C:Theta})-($\mathcal C$.\ref{C:locsta}) that 
\[
 \mathbf{I}_n(\boldsymbol \theta_{0})=O(p_n |D_n|) \quad \mbox{and} \quad
 \mathbf{J}_n(\boldsymbol \theta_{0})=O(p_n |D_n|).
\]
Thus $ \mathbf{B}_n(\boldsymbol \theta_{0})=O(p_n |D_n|)$ and the result is proved since for any centered real-valued stochastic process $Y_n$ with finite variance, $Y_n=O_{\mathrm P}(\sqrt{\Var[Y_n]})$.
\end{proof}

\section{Proof of Theorem~{\ref{THM:ROOT}}} \label{proof1}

In the proof of this result and the following ones, the notation $\kappa$ stands for a generic constant which may vary from line to line. In particular this constant is independent of $n$, $\boldsymbol \theta_0$ and $\mathbf k$.

\begin{proof}
Let $ d_n = \sqrt{p_n} (|D_n|^{-1/2}+a_n)$, and $\mathbf{k}=\{k_1, k_2, \ldots, k_{p_n}\}^\top  $. We remind the reader that the estimate of $\boldsymbol\theta_0$ is defined as the maximum of the function $Q_n$ (given by~(\ref{qn})) over $\Theta$, an open convex bounded set of $\mathbb R^{p_n}$. For any $\mathbf k$ such that $\|\mathbf k\|\leq K<\infty$, $\boldsymbol \theta_0 + d_n \mathbf k \in \Theta$ for $n$  sufficiently large. Assume this is valid in the following.
To prove Theorem~\ref{THM:ROOT}, we follow the main argument by~\cite{fan2001variable} and aim at proving that for any given $\epsilon>0$, there exists $K>0$ such that for $n$ sufficiently large
\begin{equation}
\label{eq:15}
\mathrm{P}\bigg(\sup_{\|\mathbf{k}\| = K} \Delta_n(\mathbf k)>0\bigg)\leq \epsilon,
\quad \mbox{ where } \Delta_n(\mathbf k) = Q_n(\bX ;\boldsymbol \theta_0+d_n\mathbf{k})-Q_n(\bX;\boldsymbol \theta_0).
\end{equation}
Equation~\eqref{eq:15} will imply that with probability at least $1-\epsilon$, there exists a local maximum in the ball $\{\boldsymbol \theta_0+d_n\mathbf{k}:\|\mathbf{k}\| \leq K\}$, and therefore  a local maximizer $\boldsymbol{\hat{\theta}}$ such that $\|{ \boldsymbol {\hat \theta}-\boldsymbol \theta_0}\|=O_\mathrm{P}(d_n)$.  We decompose $\Delta_n(\mathbf k)$ as $\Delta_n(\mathbf k)= T_1+T_2$ where
\begin{align*}
	T_1 = & \mbox{LPL}_n(\bX;\boldsymbol \theta_0+d_n\mathbf{k})-\mbox{LPL}_n(\bX; \boldsymbol \theta_0) \\
	T_2 = & |D_n|{\sum_{j=l+1}^{p_n} \big( p_{\lambda_{n,j}}(|\theta_{0j}|)}- p_{\lambda_{n,j}}(|\theta_{0j}+d_nk_j|)\big).
\end{align*}
Since $\boldsymbol{\theta}  \mapsto \lambda_{\boldsymbol{\theta}}(u, \bx)$ is infinitely continuously differentiable, then using a second-order Taylor expansion there exists $t\in (0,1)$ such that

\begin{align*}
	T_1 =& \, d_n \mathbf k^\top \mathbf{LPL}_n^{(1)}(\bX;\boldsymbol \theta_0) + \frac12d_n^2\mathbf k^\top \mathbf{LPL}_n^{(2)} (\bX;\boldsymbol \theta_0 + td_n \mathbf k) \mathbf k .	
\end{align*}
Since $\mathbf{LPL}_n^{(2)} (\bX;\boldsymbol \theta_0+ td_n \mathbf k)=-\mathbf{A}_n(\bX;\boldsymbol{\theta}_{0} + td_n \mathbf k)$, we can rewrite $T_1$ as 

\begin{align*}
	T_1 =& \, d_n \mathbf k^\top \mathbf{LPL}_n^{(1)}(\bX;\boldsymbol \theta_0) -\frac12d_n^2\mathbf k^\top \mathbf{A}_n(\bX;\boldsymbol{\theta}_{0} + td_n \mathbf k)\mathbf k. 	
\end{align*}
Now, denote $\check \nu := \liminf_{n\to \infty} \nu_{\min}(|D_n|^{-1}\mathbf{A}_n(\bX;\boldsymbol \theta_0+td_n \mathbf k))$. Let $\mathscr{N}(\btheta_0)$ be a neighborhood of $\btheta_0$. For n sufficiently large, $\btheta_0 + td_n \mathbf k \in \mathscr{N}(\btheta_0)$. Thus, by condition ($\mathcal C$.\ref{C:An}) we have that for any $\mathbf k$ almost surely
\[
0<\check \nu \leq  \frac{\mathbf k^\top \left( |D_n|^{-1}\mathbf A_n(\bX;\boldsymbol\theta_0+td_n \mathbf k)\right) \mathbf k}{\|\mathbf k\|^2}. 
\]
Hence, for $n$ sufficiently large
\[
	T_1 \leq d_n \| \mathbf{LPL}_n^{(1)}(\bX;\boldsymbol \theta_0)\| \, \| \mathbf k \|  - \frac{\check \nu}2 d_n^2 |D_n| \|\mathbf k\|^2. 
\]
We can rewrite $T_2$ as 
\[
T_2=|D_n|{\sum_{j=1}^{q_n} \big( p_{\lambda_{n,j+l}}(|\beta_{0j}|)}- p_{\lambda_{n,j+l}}(|\beta_{0j}+d_nk_{j+l}|)\big)
\] where $q_n \approx p_n$ for $n$ sufficiently large. We have

\[
T_2\leq T_2^\prime := |D_n|{\sum_{j=1}^s \big( p_{\lambda_{n,j+l}}(|\beta_{0j}|)}- p_{\lambda_{n,j+l}}(|\beta_{0j}+d_nk_{j+l}|)\big) 	
\]
since for any $j$ the penalty function $p_{\lambda_{n,j}}$ is nonnegative and $p_{\lambda_{n,j}}(|\beta_{0j}|)=0$ for $j=s+1,\dots,q_n$. \\
From ($\mathcal C$.\ref{C:plambda}), for $n$ sufficiently large, $p_{\lambda_{n,j}}$ is twice continuously differentiable for every $\beta_j = \beta_{0j}+t d_n k_{j+l}$ with $t\in (0,1)$. Therefore using a third-order Taylor expansion, there exist $t_j \in (0,1)$, $j=1,\dots,s$ such that $-T_2^\prime=T_{2,1}^\prime+T_{2,2}^\prime+T_{2,3}^\prime$, where
\begin{align*}
T_{2,1}^\prime &=d_n|D_n|\sum_{j=1}^s k_{j+l} p_{\lambda_{n,j+l}}^\prime(|\beta_{0j}|) \sign(\beta_{0,j}) \leq \sqrt s a_n d_n |D_n| \|\mathbf k \| \leq d_n^2 |D_n| \|\mathbf k \|, \\
T_{2,2}^\prime&=\frac12 d_n^2 |D_n|\sum_{j=1}^s k_{j+l}^2 p^{\prime\prime}_{\lambda_{n,j+l}} (|\beta_{0j}|) \leq c_n d_n^2|D_n| \|\mathbf k\|^2, \\
T_{2,3}^\prime&= \frac16 d_n^3|D_n| \sum_{j=1}^s k_{j+l}^3 p^{\prime\prime\prime}_{\lambda_{n,j+l}}  (|\beta_{0j}+t_j d_n k_{j+l}|) \leq \kappa d_n^3 |D_n|.
\end{align*}
The three inequalities above are obtained using the definitions of $a_n$ and $c_n$, condition~($\mathcal C$.\ref{C:plambda}) and Cauchy-Schwarz inequality. We deduce that for $n$ sufficiently large
\[
T_2\leq |T_2^\prime| \leq 2 d_n^2|D_n| \|\mathbf k\|
\] and then 
\[
\Delta_n(\mathbf k) \leq d_n \| \mathbf{LPL}_n^{(1)}(\bX;\boldsymbol \theta_0)\| \, \| \mathbf k \| -\frac{\check \nu}4 d_n^2 |D_n| \|\mathbf k\|^2 + 2 d_n^2 |D_n| \|\mathbf k\|.
\]
We now return to (\ref{eq:15}): for $n$ sufficiently large
\[
\mathrm{P}\bigg({\sup_{\|\mathbf{k}\|= K}  \Delta_n(\mathbf{k})>0}\bigg) \leq 	\mathrm P \bigg(
\| \mathbf{LPL}_n^{(1)}(\bX;\boldsymbol\theta_0)\| > \frac{\check \nu}2 d_n|D_n| K - 2 d_n |D_n|
 \bigg).
\]
Since $d_n |D_n|=O(\sqrt{p_n|D_n|})$, by choosing $K$ large enough, there exists $\kappa$ such that for $n$  sufficiently large
\[
	\mathrm P \bigg( \sup_{\|\mathbf k\|=K}  \Delta_n(\mathbf k) >0\bigg) \leq \mathrm P \bigg( \|\mathbf{LPL}_n^{(1)}(\bX;\boldsymbol \theta_0)\| >\kappa \sqrt{p_n|D_n|}\bigg) \leq \epsilon
\]
for any given $\epsilon>0$ from \eqref{eq:normal}. 
\end{proof}


\section{Auxiliary lemmas for the proof of Theorem~\ref{THM:SPARSITYCLT}} \label{proof2}
\begin{lemma}
\label{lemma2}
Assume the conditions ($\mathcal C$.\ref{C:Dn})-($\mathcal C$.\ref{C:BnCn}) and condition ($\mathcal C$.\ref{C:plambda}) hold. If $a_n=O(|D_n|^{-1/2})$ and $b_n \sqrt{|D_n|/p_n^2}\to \infty$ as $n\to\infty$, then with probability tending to $1$, for any {$\boldsymbol \theta_1$} satisfying $\|{\boldsymbol{\theta_1} - \boldsymbol{\theta_{01}}}\|=O_\mathrm{P}(\sqrt{p_n/|D_n|})$, and for any constant $K_1 > 0$,
\begin{align*}
Q_n\Big(\bX;({\boldsymbol \theta_1}^\top,\mathbf{0}^\top)^\top \Big)
= \max_{\| \boldsymbol \theta_2\| \leq K_1 \sqrt{p_n/|D_n|}}
Q_n\Big(\bX;({\boldsymbol \theta_1}^\top,{\boldsymbol \theta_2}^\top)^\top \Big).
\end{align*}
\end{lemma}
\begin{proof}
It is sufficient to show that with probability tending to $1$ as ${n\to \infty}$, for any ${\boldsymbol \theta_1}$ satisfying $\|{\boldsymbol \theta_1 -\boldsymbol \theta_{01}}\|=O_\mathrm{P}(\sqrt{p_n/|D_n|})$, for some $\varepsilon_n=K_1\sqrt{p_n/|D_n|}$, and for $j=l+s+1, \ldots, p_n$,

\begin{equation}
\label{eq:lem1}
\frac {\partial Q_n(\bX;\bf \boldsymbol \theta)}{\partial\theta_j}<0 \quad
\mbox { for } 0<\theta_j<\varepsilon_n, \mbox{ and}
\end{equation}

\begin{equation}
\label{eq:lem1b}
\frac {\partial Q_n(\bX;\bf \boldsymbol \theta)}{\partial\theta_j}>0 \quad
\mbox { for } -\varepsilon_n<\theta_j<0.
\end{equation}
Let $j \in \{l+s+1,\cdots,p_n\}$. From (\ref{ch2:Pois}) we have
\[
\frac {\partial \mbox{LPL}_n(\bX;\boldsymbol \theta)}{\partial\theta_j}=\frac {\partial \mbox{LPL}_n{(\bX;\boldsymbol \theta_0)}}{\partial\theta_j}+ R_n
\] 
where $R_n=\int_{D_n} \mathbf{t}_j(u,\bX) \left(\lambda_{\boldsymbol{\theta}_{0}}(u,\bX)- \lambda_{\boldsymbol\theta}(u,\bX)\right)\mathrm{d}u$.
Using similar arguments used in the proof of Lemma~\ref{lemma1}, we can prove that
\[
\frac {\partial \mbox{LPL}_n{(\bX;\boldsymbol \theta_0)}}{\partial\theta_j}=O_\mathrm{P}(\sqrt{|D_n|}).
\]
Let $u \in \mathbb{R}^d$. Using first-order Taylor expansion, there exists $t \in (0,1)$ such that
\[
\lambda_{\boldsymbol{\theta}}(u,\bX)=\lambda_{\boldsymbol{\theta}_{0}}(u,\bX)+\left(\boldsymbol{\theta}-\boldsymbol{\theta}_{0}\right)^\top \mathbf{t}(u,\bX) \lambda_{\tilde{\boldsymbol{\theta}}}(u,\bX)
\]
where $\tilde{\boldsymbol{\theta}}=\boldsymbol{\theta}_{0}+t(\boldsymbol{\theta}-\boldsymbol{\theta}_{0})$. \\
For $n$ sufficiently large, we have by Cauchy-Schwarz inequality and conditions  ($\mathcal C$.\ref{C:Theta})-($\mathcal C$.\ref{C:locsta})
\[
|R_n| \leq \kappa \hspace{-.15cm} \int_{D_n} \hspace{-.3cm} \| \boldsymbol{\theta}_0-\boldsymbol{\theta} \| \| \mathbf{t}(u,\bX) \| \mathrm{d}u= \hspace{-.15cm} \int_{D_n} \hspace{-.3cm} O_\mathrm{P}(\sqrt{p_n/|D_n|}) O_\mathrm{P}({\sqrt{p_n}}) \mathrm{d}u=O_\mathrm{P}(\sqrt{|D_n| p_n^2}).
\]
We then deduce that for any $j=l+s+1, \ldots, p_n$
\[
\frac {\partial \mbox{LPL}_n(\bX;\boldsymbol \theta)}{\partial\theta_j}=O_\mathrm{P}(\sqrt{|D_n| p_n^2}).
\]
Now, we want to prove (\ref{eq:lem1}). Let $0<\theta_j<\varepsilon_n$ and $b_n$ the sequence given by~(\ref{eq:bn}). By condition ($\mathcal C$.\ref{C:plambda}), $b_n$ is well-defined and since by assumption $b_n\sqrt{|D_n|/p_n^2}\to~\infty$, in particular,  $b_n>0$ for $n$ sufficiently large. Therefore, for $n$ sufficiently large,
\begin{align*}
\mathrm{P} \left ( \frac {\partial Q_n(\bX;\boldsymbol \theta)}{\partial\theta_j}<0 \right)&=\mathrm{P} \left ( \frac {\partial \mathrm{LPL}_n(\bX;\boldsymbol \theta)}{\partial\theta_j} - |D_n|p'_{\lambda_{n,j}}(|\theta_j|)\sign(\theta_j)<0 \right)\\
&=\mathrm{P} \left ( \frac {\partial \mathrm{LPL}_n(\bX;\boldsymbol \theta)}{\partial\theta_j}< |D_n|p'_{\lambda_{n,j}}(|\theta_j|) \right)\\
& \geq \mathrm{P} \left ( \frac {\partial \mathrm{LPL}_n(\bX;\boldsymbol \theta)}{\partial\theta_j}< |D_n|b_n \right)\\
&= \mathrm{P} \left ( \frac {\partial \mathrm{LPL}_n(\bX;\boldsymbol \theta)}{\partial\theta_j}< \sqrt{|D_n| p_n^2}\sqrt{\frac{|D_n|}{p_n^2}}b_n \right).
\end{align*}
$\mathrm{P} \left ( {\partial Q_n(\bX;\boldsymbol \theta)}/{\partial\theta_j}<0 \right) \xrightarrow{}1 \mbox{ as } n \to \infty$
since $ {\partial \mathrm{LPL}_n(\bX;\boldsymbol \theta)}/{\partial\theta_j}=O_\mathrm{P}(\sqrt{|D_n|p_n^2})$ and $b_n\sqrt{|D_n|/p_n^2} \xrightarrow{} \infty$. This proves (\ref{eq:lem1}). We proceed similarly to prove (\ref{eq:lem1b}).
\end{proof}

The next lemma provides a central limit theorem (CLT) for $\mathbf{LPL}^{(1)}_{n,1}( \bX;\boldsymbol{\theta}_{0})$, the vector of the first $m=l+s$ components of $\mathbf{LPL}^{(1)}_{n}( \bX;\boldsymbol{\theta}_{0})$. Its proof is based on a general CLT for nonstationary conditionally centered random fields proved by \cite{coeurjolly2017parametric}.
\begin{lemma}
\label{lemma3}
Under the conditions ($\mathcal C$.\ref{C:Dn})-($\mathcal C$.\ref{C:BnCn}) and $m=l+s$, the following convergence holds in distribution as $n\to\infty$
\begin{align}
\ \{\mathbf{A}_{n,11}(\bX;\boldsymbol{\theta}_{0})+\mathbf{B}_{n,11}(\bX;\boldsymbol{\theta}_{0})\}^{-1/2}\mathbf{LPL}^{(1)}_{n,1}( \bX;\boldsymbol{\theta}_{0}) \xrightarrow{d} \mathcal{N}(\mathbf{0},\mathbf{I}_{m}) \label{eq:clt}
\end{align}
\end{lemma}

\begin{proof}
Denoting by $\Delta_j$ the unit cube centered at $j \in \mathbb{Z}^d$, we let \\ $\Delta_{n,j}=\Delta_j \cap D_n $ and $\boldsymbol{\mathcal{I}}_n \subset \mathbb{Z}^d$ the set such that $D_n=\cup_{j \in \boldsymbol{\mathcal{I}}_n}\Delta_{n,j}$. We define
\[
\mathbf{LPL}^{(1)}_{n,1}( \bX;\boldsymbol{\theta}_0)=\sum_{j \in \mathcal{I}_n} \bY_{n,j}
\]
where 
\[
\bY_{n,j}=\sum_{u \in \bX_{\Delta_{n,j}}} \mathbf{t}(u,\bX \setminus u)-\int_{\Delta_{n,j}} \mathbf{t}(u,\bX) \lambda_{\boldsymbol{\theta}_0}(u,\bX) \mathrm{d}u.
\]
We also define 
\[
\hat{\boldsymbol{\Sigma}}_n= \sum_{j \in \boldsymbol{\mathcal{I}}_n} \sum_{\substack{
   k \in \boldsymbol{\mathcal{I}}_n \\
   | k-j | \leq R
  }} \bY_{n,j} \bY_{n,k}^\top \quad \mbox{and} \quad \boldsymbol{\Sigma}_n=\EE \,\hat{\boldsymbol{\Sigma}}_n.
\] 
To prove Lemma~\ref{lemma3}, we apply Theorem A.1 of \cite{coeurjolly2017parametric}. We are led to verify the following assumptions:
\begin{enumerate}
\item[(a)]  $\EE \bY_{n,j} = 0$ and there exists $q\geq 1$ such that $\sup_{n \geq 1} \sup_{j \in \boldsymbol{\mathcal{I}_n}} \EE ||\bY_{n,j}||^{4q}<~\infty$, \\
\item[(b)] for any sequence $\boldsymbol{\mathcal{J}}_n \subset \boldsymbol{\mathcal{I}}_n$   such that $|\boldsymbol{\mathcal{J}}_n |$ $\to \infty$ as $n \to \infty,$ \[
|\boldsymbol{\mathcal{J}}_n|^{-1} \sum_{j,k \in \boldsymbol{\mathcal{J}}_n}  \| \EE (\bY_{n,j}\bY_{n,k}^\top) \|=O(1),
\] \\
\item[(c)] there exists a positive definite matrix $\boldsymbol{Q}$ such that  $|\boldsymbol{\mathcal{I}}_n|^{-1} \boldsymbol{\Sigma}_n \geq \boldsymbol{Q}$ for all sufficiently large $n$, \\
\item[(d)] as $n \to \infty$
\[
|\boldsymbol{\mathcal{I}}_n|^{-1/2} \sum_{j \in \boldsymbol{\mathcal{I}_n}} \EE || \EE (\bY_{n,j}| \bX_{n,k}, k\neq j)|| \to 0. 
\]
\end{enumerate}

\begin{itemize}
\item Condition (a): By GNZ formula~(\ref{gnz}), $\bY_{n,j}$ has zero mean for any $n\geq 1$ and $j \in \boldsymbol{\mathcal{I}_n}$. Now,
\[
\EE ||\bY_{n,j}||^{4q}=\EE \left \{ \left \|  \sum_{u \in \bX_{\Delta_{n,j}}} \mathbf{t}(u,\bX \setminus u)-\int_{\Delta_{n,j}} \mathbf{t}(u,\bX) \lambda_{\boldsymbol{\theta}_0}(u,\bX) \mathrm{d}u  \right \|^{4q} \right \}.
\]
To show the remaining statement in $(a)$, it is sufficient to show that 
\[
\EE \left \{ \left |  \sum_{u \in \bX_{\Delta_{n,j}}} \mathbf{t}_i(u,\bX \setminus u)-\int_{\Delta_{n,j}} \mathbf{t}_i(u,\bX) \lambda_{\boldsymbol{\theta}_0}(u,\bX) \mathrm{d}u  \right |^{4q} \right \} <\infty 
\]
for any $i=1,\cdots, m$. Taking $q=1$, we have
\begin{multline}
\label{binom}
\EE \left \{ \left |  \sum_{u \in \bX_{\Delta_{n,j}}} \mathbf{t}_i(u,\bX \setminus u)-\int_{\Delta_{n,j}} \mathbf{t}_i(u,\bX) \lambda_{\boldsymbol{\theta}_0}(u,\bX) \mathrm{d}u  \right |^{4} \right \} \\
 = \EE \Bigg \{ \left ( \sum_{u \in \bX_{\Delta_{n,j}}} \mathbf{t}_i(u,\bX \setminus u)\right )^{4} + \left ( \int_{\Delta_{n,j}} \mathbf{t}_i(u,\bX) \lambda_{\boldsymbol{\theta}_0}(u,\bX)  \mathrm{d}u  \right )^{4} \\
 - 4 \left ( \sum_{u \in \bX_{\Delta_{n,j}}} \mathbf{t}_i(u,\bX \setminus u)\right )^{3}  \int_{\Delta_{n,j}} \mathbf{t}_i(u,\bX) \lambda_{\boldsymbol{\theta}_0}(u,\bX)  \mathrm{d}u \\
 + 6 \left ( \sum_{u \in \bX_{\Delta_{n,j}}} \mathbf{t}_i(u,\bX \setminus u)\right )^{2} \left(  \int_{\Delta_{n,j}} \mathbf{t}_i(u,\bX) \lambda_{\boldsymbol{\theta}_0}(u,\bX)  \mathrm{d}u \right)^2   \\
 - 4 \sum_{u \in \bX_{\Delta_{n,j}}} \mathbf{t}_i(u,\bX \setminus u)  \left(  \int_{\Delta_{n,j}} \mathbf{t}_i(u,\bX) \lambda_{\boldsymbol{\theta}_0}(u,\bX)  \mathrm{d}u \right)^3 \Bigg \}.
\end{multline}

Note that
\begin{align*}
 \left ( \sum_{u} \mathbf{t}_i(u,\bX \setminus u)\right )^{4} =& \left( \left ( \sum_{u} \mathbf{t}_i(u,\bX \setminus u)\right )^{2} \right)^{2}  \\
 =& \left( \sum_{u } \mathbf{t}_i^2(u,\bX \setminus u) + \sum_{u,v}^{\neq}  \mathbf{t}_i(u,\bX \setminus u) \mathbf{t}_i(v,\bX \setminus v) \right)^{2} \\
 =& \sum_{u } \mathbf{t}_i^4(u,\bX \setminus u) + \sum_{u,v}^{\neq} \mathbf{t}_i^2(u,\bX \setminus u) \mathbf{t}_i^2(v,\bX \setminus v) \; + \\
&  \hspace{-1cm} \sum_{u,v,w,y}^{u\neq w, v \neq y} \mathbf{t}_i(u,\bX \setminus u) \mathbf{t}_i(v,\bX \setminus v) \mathbf{t}_i(w,\bX \setminus w) \mathbf{t}_i(y,\bX \setminus y) \\
+ & 2 \sum_{u,v,w}^{v \neq w} \mathbf{t}_i^2(u,\bX \setminus u) \mathbf{t}_i(v,\bX \setminus v) \mathbf{t}_i(w,\bX \setminus w). 
\end{align*}
Now, using GNZ formula we have
\begin{align*}
\EE \left [ \left ( \sum_{u \in \bX_{\Delta_{n,j}} } \mathbf{t}_i(u,\bX \setminus u)\right )^{4} \right ] =& \int_{\Delta_{n,j}}  \EE \left [ \mathbf{t}_i^4(u,\bX) \lambda_{\boldsymbol{\theta}_0}(u,\bX) \right ]  
\mathrm{d}u \; + \\
& \hspace{-1.5cm} \int_{\Delta_{n,j}^2}  \EE \left [  \mathbf{t}_i^2(u,\bX) \mathbf{t}_i^2(v,\bX) \lambda_{\boldsymbol{\theta}_0}(\{u,v\}, \bX) \right ] \mathrm{d}u \mathrm{d}v \; +\\
&  \hspace{-4.8cm} \int_{\Delta_{n,j}^4} \hspace{-.4cm} \EE \left [  \mathbf{t}_i(u,\bX) \mathbf{t}_i(v,\bX) \mathbf{t}_i(w,\bX)  \mathbf{t}_i(y,\bX) \lambda_{\boldsymbol{\theta}_0}(\{u,w\}, \bX)  \lambda_{\boldsymbol{\theta}_0}(\{v,y\},\bX)  \right ] \hspace{-.15cm}  \mathrm{d}u \mathrm{d}v \mathrm{d}w \mathrm{d}y\\
& \hspace{-3.8cm} + 2 \int_{\Delta_{n,j}^3} \hspace{-.15cm} \EE \left [  \mathbf{t}_i^2(u,\bX) \mathbf{t}_i(v,\bX) \mathbf{t}_i(w,\bX) \lambda_{\boldsymbol{\theta}_0}(u,\bX) \lambda_{\boldsymbol{\theta}_0}(\{v,w\}, \bX)\right ] \hspace{-.15cm} \mathrm{d}u \mathrm{d}v \mathrm{d}w.
\end{align*}
Since $\lambda_{\boldsymbol{\theta}_0}(\{u,v\}, \bX) = \lambda_{\boldsymbol{\theta}_0}(u, \bX \cup v) \lambda_{\boldsymbol{\theta}_0}(v, \bX)$ for any $u,v \in D_n$, it follows from Cauchy-Schwarz inequality  and conditions ($\mathcal C$.\ref{C:cov})-($\mathcal C$.\ref{C:locsta}) that 
\[
\EE \left [ \left ( \sum_{u \in \bX_{\Delta_{n,j}} } \mathbf{t}_i(u,\bX \setminus u)\right )^{4} \right ] < \infty.
\]
We also have
\begin{align*}
\EE \left [ \left ( \int_{\Delta_{n,j}} \mathbf{t}_i(u,\bX) \lambda_{\boldsymbol{\theta}_0}(u,\bX)  \mathrm{d}u  \right )^{4} \right ] = & \int_{\Delta_{n,j}^4} \EE \left [   \prod_{k=1}^4 \mathbf{t}_i(u_k,\bX) \lambda_{\boldsymbol{\theta}_0}(u_k, \bX) \mathrm{d}u_k \right ].
\end{align*}
Again combining conditions ($\mathcal C$.\ref{C:cov})-($\mathcal C$.\ref{C:locsta}) and Cauchy-Schwarz inequality, it follows that 
\[
\EE \left [ \left ( \int_{\Delta_{n,j}} \mathbf{t}_i(u,\bX) \lambda_{\boldsymbol{\theta}_0}(u,\bX)  \mathrm{d}u  \right )^{4} \right ] < \infty.
\]
We proceed similarly to prove that the other terms in (\ref{binom}) are also finite. This finally implies that
\begin{align}
\label{spuissance}
\EE \left \{ \left |  \sum_{u \in \bX_{\Delta_{n,j}}} \mathbf{t}_i(u,\bX \setminus u)-\int_{\Delta_{n,j}} \mathbf{t}_i(u,\bX) \lambda_{\boldsymbol{\theta}_0}(u,\bX) \mathrm{d}u  \right |^{4} \right \} <\infty.
\end{align}
Assumption $(a)$ is then verified by choosing $q=1$. \\
\item Condition (d): From Lemma 2.2 in~\cite{coeurjolly2017parametric}, we have 
\[
\EE(\bY_{n,j}|\bX_{n,k}, k \neq j)=0,
\]
which implies assumption $(d)$. \\
\item Condition (c): By the definition of $\bY_{n,j}$, we have
\begin{align}
\label{sigma:n}
\hat{\boldsymbol{\Sigma}}_n=\mathbf{A}_{n,11}(\bX;\boldsymbol{\theta}_{0})+\mathbf{B}_{n,11}(\bX;\boldsymbol{\theta}_{0}) \quad \mbox{and} \quad 
\boldsymbol{\Sigma}_n=\mathbf{A}_{n,11}(\boldsymbol{\theta}_{0})+\mathbf{B}_{n,11}(\boldsymbol{\theta}_{0}).  
\end{align}
Condition ($\mathcal C$.\ref{C:BnCn}) implies that there exists  a positive definite matrix $\boldsymbol{Q}$ such that for all sufficiently large $n$, we have $|D_n|^{-1} \boldsymbol{\Sigma}_n \geq \boldsymbol{Q}$. By condition ($\mathcal C$.\ref{C:Dn}) and the definition of $\boldsymbol{\mathcal{I}_n}$, we have $|\boldsymbol{\mathcal{I}_n}|=O(|D_n|)$. We therefore have $|\boldsymbol{\mathcal{I}_n}|^{-1} \boldsymbol{\Sigma}_n \geq \boldsymbol{Q}$. \\
\item Condition (b): Finally, we have
\[
\EE \bY_{n,j}\bY_{n,k}^\top=\mathrm{cov} \left(I_{\Delta_{n,j}}(\bX,\mathbf{t}),I_{\Delta_{n,k}}(\bX,\mathbf{t}) \right)
\]
where $I_{\Delta_{n,j}}(\bX,\mathbf{t})$ is the $\mathbf{t}$-innovation function defined, see e.g. \cite{baddeley2005residual} or~\cite{coeurjolly2013fast}, by
\[
I_{\Delta_{n,j}}(\bX,\mathbf{t})=\sum_{u \in \bX_{\Delta_{n,j}}} \mathbf{t}(u,\bX \setminus u)-\int_{\Delta_{n,j}} \mathbf{t}(u,\bX) \lambda_{\boldsymbol{\theta}_0}(u,\bX) \mathrm{d}u.
\]
From~\cite{coeurjolly2013fast}, we can show that
\begin{align}
\label{covinnov}
\EE \bY_{n,j}\bY_{n,k}^\top=O(1).
\end{align}
We also have
\begin{align*}
\sum_{ j,k \in \boldsymbol{\mathcal{J}_n}} \EE(\bY_{n,j}\bY_{n,k}^\top)=\sum_{\substack{
   j,k \in \boldsymbol{\mathcal{J}_n} \\
   | k-j | \leq R
  }} \EE(\bY_{n,j} \bY_{n,k}^\top) + \sum_{\substack{
   j,k \in \boldsymbol{\mathcal{J}_n} \\
   | k-j | > R
  }} \EE(\bY_{n,j} \bY_{n,k}^\top).
\end{align*}
From the finite range property~(\ref{fr}), $\bY_{n,j}$ is a function of $\bX_{n,k}$ for any $k \in \boldsymbol{\mathcal{J}_n}$ such that $| k-j | \leq R$. Thus, if $| k-j | > R$, we have 
\begin{align*}
\EE \bY_{n,j}\bY_{n,k}^\top= &\EE \left [ \EE \left [\bY_{n,j}\bY_{n,k}^\top| \bX_{n,k}, k\neq j \right ]  \right] \\
=& \EE \left [ \EE \left [\bY_{n,j}| \bX_{n,k}, k\neq j \right ]  \bY_{n,k}^\top\right] \\
=&0
\end{align*}
whereby we deduce that 
\begin{align}
\label{assb}
\sum_{ j,k \in \boldsymbol{\mathcal{J}_n}} \EE(\bY_{n,j}\bY_{n,k}^\top)=\sum_{\substack{
   j,k \in \boldsymbol{\mathcal{J}_n} \\
   | k-j | \leq R
  }} \EE(\bY_{n,j} \bY_{n,k}^\top).
\end{align}
Equations (\ref{covinnov}) and (\ref{assb}) imply assumption $(b)$. 
\end{itemize}
Conditions (a)-(d) being valid, we can now apply Theorem A.1 in~\cite{coeurjolly2017parametric} and get 
\begin{align}
\label{asympt}
\boldsymbol{\Sigma}_n^{-1/2} \boldsymbol{S}_n \xrightarrow{d} \mathcal{N}(\mathbf{0},\mathbf{I}_{m}), 
\end{align}
where $\boldsymbol{S}_n=\mathbf{LPL}^{(1)}_{n,1}( \bX;\boldsymbol{\theta}_0)$.
We have
\[
\hat{\boldsymbol{\Sigma}}_n^{-1/2} \boldsymbol{S}_n= (\boldsymbol{\Sigma}_n^{-1} \hat{\boldsymbol{\Sigma}}_n)^{-1/2} \boldsymbol{\Sigma}_n^{-1/2} \boldsymbol{S}_n \quad \mbox{and} \quad \boldsymbol{\Sigma}_n^{-1} \hat{\boldsymbol{\Sigma}}_n - \mathbf{I}_{m} = \boldsymbol{\Sigma}_n^{-1} (\hat{\boldsymbol{\Sigma}}_n -  \boldsymbol{\Sigma}_n).
\]
As mentioned in Theorem A.1, assumptions $(a)$ and $(b)$ imply that 
\begin{align}
\label{sigmahat}
\mbox{as} \quad n \to \infty, \quad  |\boldsymbol{\mathcal{I}}_n|^{-1} (\hat{\boldsymbol{\Sigma}}_n-\boldsymbol{\Sigma}_n) \to 0 \quad \mbox{in} \quad L^2.
\end{align}
Since $\hat{\boldsymbol{\Sigma}}_n-\boldsymbol{\Sigma}_n=o_{\mathrm{P}}(|D_n|)$ by (\ref{sigmahat}) and $\boldsymbol{\Sigma}_n^{-1}=O(|D_n|^{-1})$ by condition ($\mathcal C$.\ref{C:BnCn}), we have 
\[
\boldsymbol{\Sigma}_n^{-1} \hat{\boldsymbol{\Sigma}}_n - \mathbf{I}_{m}=o_{\mathrm{P}}(1), \; \mbox{which gives}  \quad \boldsymbol{\Sigma}_n^{-1} \hat{\boldsymbol{\Sigma}}_n=1+o_{\mathrm{P}}(1).
\]
The latter equality means that $\boldsymbol{\Sigma}_n^{-1} \hat{\boldsymbol{\Sigma}}_n$ converges in probability to $1$, which finally implies that 
\begin{align}
\label{convergence:proba}
(\boldsymbol{\Sigma}_n^{-1} \hat{\boldsymbol{\Sigma}}_n)^{-1/2}=1+o_{\mathrm{P}}(1).
\end{align}
We deduce from (\ref{asympt}) and Slutsky's Theorem that
\[
\hat{\boldsymbol{\Sigma}}_n^{-1/2} \boldsymbol{S}_n \xrightarrow{d} \mathcal{N}(\mathbf{0},\mathbf{I}_{m}).
\]
\end{proof}
\section{ Proof of Theorem~\ref{THM:SPARSITYCLT}}  \label{proof3}

\begin{proof} We now focus on the proof of Theorem~\ref{THM:SPARSITYCLT}. Since Theorem~\ref{THM:SPARSITYCLT}(i) is proved by Lemma~\ref{lemma2}, we only need to prove Theorem~\ref{THM:SPARSITYCLT}(ii), which is the asymptotic normality of $\boldsymbol {\hat{\theta}}_1$. As shown in Theorem~\ref{THM:ROOT}, there is a root-($|D_n|/p_n$) consistent local maximizer $\boldsymbol{\hat{\theta}}$ of $Q_n(\bX;\boldsymbol \theta)$, and it can be shown that there exists an estimator $\boldsymbol {\hat{\theta}}_1$ in Theorem~\ref{THM:ROOT} that is a root-$(|D_n|/p_n)$ consistent local maximizer of $ Q_n\Big(\bX;({\boldsymbol \theta_1}^\top,\mathbf{0}^\top)^\top \Big)$, which is regarded as a function of  $\boldsymbol {\theta}_1$, and that satisfies
\begin{align*}
\frac {\partial Q_n(\bX;\boldsymbol {\hat \theta})}{\partial\theta_j}=0 \quad
\mbox { for } j=1,\ldots,m=l+s
\mbox {, and } \boldsymbol{\hat \theta}=( \boldsymbol {\hat{\theta}}_1^\top,\mathbf{0}^ \top)^\top.
\end{align*}
There exists $t\in (0,1)$ and $\boldsymbol{\breve{\theta}}= \boldsymbol\theta_0 + t(\boldsymbol{\hat \theta}-\boldsymbol\theta_0)$  such that
\begin{align}
0
=&\frac {\partial \mathrm{LPL}_n{(\bX;\boldsymbol{\hat \theta})}}{\partial\theta_j}-|D_n|p'_{\lambda_{n,j}}(|\hat \theta_{j}|)\sign(\hat \theta_j) \nonumber\\
=&\frac {\partial \mathrm{LPL}_n{(\bX;\boldsymbol \theta_0)}}{\partial\theta_j}+{\sum_{r=1}^{m} \frac {\partial^2 \mathrm{LPL}_n{(\bX; \boldsymbol{\breve{\theta}})}}{\partial\theta_j \partial\theta_r}}({\hat \theta_r}-\theta_{0r})-|D_n|p'_{\lambda_{n,j}}(|\hat \theta_{j}|)\sign(\hat \theta_j) \nonumber\\
=&\frac {\partial \mathrm{LPL}_n{(\bX;\boldsymbol \theta_0)}}{\partial\theta_j}+{\sum_{r=1}^m \frac {\partial^2 \mathrm{LPL}_n{(\bX; \boldsymbol \theta_0)}}{\partial\theta_j \partial\theta_r}}({\hat \theta_r}-\theta_{0r})+{\sum_{r=1}^m \Psi_{n,jr}({\hat \theta_r}-\theta_{0r})} \nonumber \\
&-|D_n|p'_{\lambda_{n,j}}(|\theta_{0j}|)\sign(\theta_{0j})-|D_n|\phi_{n,j}, \label{eq:0equal}
\end{align}
where 
\begin{align*}
\Psi_{n,jr}=\frac {\partial^2 \mathrm{LPL}_n{(\bX;\boldsymbol{\breve{\theta}})}}{\partial\theta_j \partial\theta_r}-\frac {\partial^2 \mathrm{LPL}_n{(\bX;\boldsymbol \theta_0)}}{\partial\theta_j \partial\theta_r}
\end{align*}
and $\phi_{n,j}=p'_{\lambda_{n,j}}(|\hat \theta_{j}|)\sign(\hat \theta_j)-p'_{\lambda_{n,j}}(|\theta_{0j}|)\sign(\theta_{0j})$. The rest of the proof for $\phi_{n,j}$ follows the same lines and arguments in \citet{choiruddin2018convex}, that is
\begin{align}
 \label{eq:phinj}
\phi_{n,j}=p''_{\lambda_{n,j}}(|\theta_{0j}|)(\hat \theta_j- \theta_{0j})(1+o_{\mathrm{P}}(1))+O_{\mathrm{P}}(p_n/|D_n|)+o_{\mathrm{P}}(|D_n|^{-1/2}).
\end{align}
Let $\mathbf{LPL}^{(1)}_{n,1}(\bX;\boldsymbol \theta_{0})$ (resp. $\mathbf{LPL}^{(2)}_{n,1}(\bX;\boldsymbol \theta_{0})$) be the first $m=l+s$ components (resp. $m \times m$ top-left corner) of $\mathbf{LPL}^{(1)}_{n}(\bX;\boldsymbol \theta_{0})$ (resp. $\mathbf{LPL}^{(2)}_{n}(\bX;\boldsymbol \theta_{0})$). Let also $\boldsymbol \Psi_n$ be the $m \times m$ matrix containing $\Psi_{n,jr}, j,r=1,\ldots,m$. Finally, let the vector $\mathbf{p}'_n$, the vector $\boldsymbol \phi_n$ and the $m \times m$ matrix $\mathbf{M}_n$ be defined by
\begin{align*}
\mathbf{p}'_n&=\{p'_{\lambda_{n,1}}(|\theta_{01}|)\sign(\theta_{01}),\ldots,p'_{\lambda_{n,m}}(|\theta_{0m}|)\sign(\theta_{0m})\}^\top, \\
\boldsymbol \phi_n&=\{\phi_{n,1},\ldots,\phi_{n,m}\}^\top, \mbox{ and}\\
\mathbf{M}_n&=\{ \mathbf{A}_{n,11}(\bX;\boldsymbol \theta_{0})+\mathbf{B}_{n,11}(\bX;\boldsymbol \theta_{0})\}^{-1/2}.
\end{align*}
We rewrite both sides of~\eqref{eq:0equal} as
\begin{equation}
\mathbf{LPL}^{(1)}_{n,1}(\bX;\boldsymbol \theta_{0})+\mathbf{LPL}^{(2)}_{n,1}(\bX;\boldsymbol \theta_{0})(\boldsymbol{\hat \theta}_1-\boldsymbol \theta_{01})+ \boldsymbol \Psi_n (\boldsymbol{\hat \theta}_1-\boldsymbol \theta_{01}) -|D_n| \mathbf{p}'_n-|D_n| \boldsymbol \phi_n   =0.\label{eq:0vec}
\end{equation}
By definition of $\boldsymbol \Pi_n$ given by (\ref{eq:pi}) and from~\eqref{eq:phinj}, we obtain $\boldsymbol \phi_n=\boldsymbol \Pi_n (\boldsymbol{\hat \theta}_1-\boldsymbol \theta_{01})\big(1+o_{\mathrm{P}}(1)\big)+o_{\mathrm P}(p_n/|D_n|)+o_{\mathrm P}(|D_n|^{-1/2})$. Using this, we deduce, by premultiplying both sides of~\eqref{eq:0vec} by $\mathbf M_n$, that
\begin{align*}
\mathbf{M}_n \mathbf{LPL}^{(1)}_{n,1}(\bX;\boldsymbol \theta_{0})-&\mathbf{M}_n \big(\mathbf{A}_{n,11}(\bX;\boldsymbol \theta_{0})+ |D_n| \boldsymbol \Pi_n\big)(\boldsymbol{\hat \theta}_1-\boldsymbol \theta_{01}) \\
& =O(|D_n| \, \|\mathbf{M}_n \mathbf{p}'_n \|) + o_{\mathrm{P}}(|D_n| \, \|\mathbf{M}_n\boldsymbol \Pi_n (\boldsymbol{\hat \theta}_1-\boldsymbol \theta_{01}) \|)  \\
& \quad+ O_\mathrm{P} (\|\mathbf M_n \| \; p_n)+ o_\mathrm{P} (\|\mathbf M_n \| \; |D_n|^{1/2}) \\
& \quad+ O_{\mathrm{P}} (\|\mathbf M_n \boldsymbol \Psi_n (\boldsymbol{\hat \theta}_1-\boldsymbol \theta_{01}) \|) .
\end{align*}
Now, using $\hat{\boldsymbol{\Sigma}}_n$ and $\boldsymbol{\Sigma}_n$ defined in (\ref{sigma:n}) we get
\begin{align*}
\mathbf{M}_n = &\left(\boldsymbol{\Sigma}_n^{-1} \hat{\boldsymbol{\Sigma}}_n \right)^{-\frac12} \boldsymbol{\Sigma}_n^{-\frac12} 
=  \Big(1+o_{\mathrm{P}}(1) \Big) \, \times \, O\left(|D_n|^{-1/2}\right).
\end{align*}
where the latter equation ensues from (\ref{convergence:proba}) and condition ($\mathcal C$.\ref{C:BnCn}). This implies that 
 $\|\mathbf M_n\|=O_{\mathrm{P}}(|D_n|^{-1/2})$. We have $\| \boldsymbol \Psi_n\|=O_{\mathrm{P}}(\sqrt{p_n|D_n|})$  by conditions ($\mathcal C$.\ref{C:Theta})-($\mathcal C$.\ref{C:locsta}) and by Theorem~\ref{THM:ROOT}, and $\|\boldsymbol{\hat \theta}_1-\boldsymbol \theta_{01}\|=O_{\mathrm{P}}(\sqrt{p_n/|D_n|})$  by Theorem~\ref{THM:ROOT} and by Theorem~\ref{THM:SPARSITYCLT}(i).
Finally, since by assumptions $a_n\sqrt{|D_n|}\to 0$, $c_n\sqrt{p_n}\to \infty$ and $p_n^2/|D_n|\to 0$ as $n\to \infty$, we deduce that
\begin{align*}
|D_n| \,\|\mathbf{M}_n \mathbf{p}'_n\|&=O_{\mathrm{P}}(a_n|D_n|^{1/2})=o_{\mathrm{P}}(1), \\
|D_n| \, \|\mathbf{M}_n\boldsymbol \Pi_n (\boldsymbol{\hat \theta}_1-\boldsymbol \theta_{01}) \| &= O_{\mathrm{P}}\left(c_n\sqrt{p_n}\right)=o_{\mathrm{P}}(1),\\
\|\mathbf M_n \| \; |D_n|^{1/2} &= O_{\mathrm{P}}(1),\\
\|\mathbf M_n \| \; p_n &= O_{\mathrm{P}}\left(\sqrt{\frac{p_n^2}{|D_n|}}\right)=o_{\mathrm{P}}(1),\\
\|\mathbf{M}_n \boldsymbol \Psi_n (\boldsymbol{\hat \theta}_1-\boldsymbol \theta_{01})\|&=O_{\mathrm{P}}\left(\sqrt{\frac{p_n^2}{|D_n|}}\right)=o_{\mathrm{P}}(1). \\
\end{align*}
Therefore, we have that
\begin{align} \label{eq:MnLn}
\mathbf{M}_n \mathbf{LPL}^{(1)}_{n,1}(\bX;\boldsymbol \theta_{0})-\mathbf{M}_n \big(\mathbf{A}_{n,11}(\bX;\boldsymbol \theta_{0})+ |D_n| \Pi_n\big)(\boldsymbol{\hat \theta}_1-\boldsymbol \theta_{01}) =o_{\mathrm{P}}(1).
\end{align}

By~\eqref{eq:clt} in Lemma~\ref{lemma3} and by Slutsky's Theorem, we deduce that
\begin{align*}
\{ \mathbf{A}_{n,11}(\bX;\boldsymbol \theta_{0})+\mathbf{B}_{n,11}(\bX;\boldsymbol \theta_{0})\}^{-1/2}
\{\mathbf{A}_{n,11}(\bX;\boldsymbol \theta_{0})+|D_n| \boldsymbol \Pi_n\}(\boldsymbol{\hat \theta}_1-\boldsymbol \theta_{01})&\xrightarrow{d} \mathcal{N}(0,\mathbf{I}_{m})
\end{align*}
as $n \to \infty$, which can be rewritten, in particular under ($\mathcal C$.\ref{C:An}), as 
\[
|D_n|^{1/2}\boldsymbol \Sigma_n(\bX;\boldsymbol \theta_{0})^{-1/2}(\boldsymbol{\hat \theta}_1-\boldsymbol \theta_{01})\xrightarrow{d}\mathcal{N}(0,\mathbf{I}_{m})	
\]
where $\mathbf \Sigma_n(\bX;\boldsymbol \theta_{0})$ is given by~\eqref{eq:Sigman}.
\end{proof}

\newpage
\section{Tables of the simulation results} \label{tab:sim}

\setlength{\tabcolsep}{1pt}
\renewcommand{\arraystretch}{1.5}
\begin{table}[h]
\caption{Empirical prediction properties (Bias, SD, and RMSE) and empirical selection properties (TPR, and FPR in $\%$) based on 500 replications of Strauss and Geyer models using cBIC, and cERIC for the Lasso penalty function and the regularized pseudo-likelihood function. The mean number of points $n$ under each model is provided.}
\label{table:lasso} 
\centering
\begin{tabular}{@{\extracolsep{1pt}}c l c c | ccccc | ccccc @{}}
\hline
\hline 
 \multicolumn{1}{c}{Spatial} & \multicolumn{1}{c}{Model} & \multicolumn{1}{c}{Interaction} & \multicolumn{1}{c}{Av. number} & \multicolumn{5}{c}{cBIC} & \multicolumn{5}{c}{cERIC} \\ 
 \cline{5-9} \cline{9-14}
domain &  & parameter & of points (n) & Bias & SD & RMSE & FPR & TPR & Bias & SD & RMSE & FPR & TPR  \\ 
  \hline
  \hline
  & \multicolumn{11}{c}{Scenario~\ref{sce1}}\\
\hline
  \multirow{3}{*}{$W_1$} & Strauss & $\gamma=0.2$ & 101 & 2.07 & 0.77 & 2.21 & 2 & 24 & 1.92 & 0.76 & 2.06 & 2 & 42\\ 
   & Strauss & $\gamma=0.5$ & 138 & 1.66 & 1.11 &  2 & 11 & 43 & 1.66 & 0.67 & 1.79 & 3 & 61\\ 
   & Geyer  & $\gamma=1.5$ & 750 & 0.22 & 0.69 &  0.72 & 51 & 100 & 0.41 & 0.57 & 0.7  & 33 & 100\\ 
\hline
  \multirow{3}{*}{$W_2$} & Strauss & $\gamma=0.2$ & 395 & 0.46 & 0.82 & 0.94  & 36 & 100 & 0.79 & 0.41 & 0.89 & 2 & 100\\ 
   & Strauss & $\gamma=0.5$ &  540& 0.45 & 0.77 & 0.89  & 40 & 100 & 0.77 &  0.33& 0.84 & 3 & 100\\ 
  & Geyer  & $\gamma=1.5$ & 2968 & 0.29 & 0.27 & 0.4  & 18 & 100 & 0.33 & 0.22 & 0.4 & 6 & 100\\ 
   \hline
  \multirow{3}{*}{$W_3$} & Strauss & $\gamma=0.2$ & 1137 & 0.5 & 0.46 & 0.68 & 15 & 100 & 0.56 & 0.27 & 0.62 & 2 & 100\\ 
   & Strauss & $\gamma=0.5$ &  1484& 0.42 & 0.46 & 0.62 & 24 & 100 & 0.52 & 0.22 & 0.56 & 2 & 100\\ 
  & Geyer  & $\gamma=1.5$ &  5630& 0.26 & 0.2 & 0.33  & 15 & 100 & 0.29 & 0.14 & 0.32 & 2 & 100\\ 
   \hline
   & \multicolumn{11}{c}{Scenario~\ref{sce2}} \\
   \hline
   \multirow{3}{*}{$W_1$} & Strauss & $\gamma=0.2$ &  101& 1.61 & 5.4 & 5.63 & 12 & 38 & 1.61 & 5.14 & 5.39 & 11 & 46\\ 
   & Strauss & $\gamma=0.5$ &  138& 1.86 & 1.75 & 2.55 & 1 & 36 & 1.76 & 0.34 & 1.79 & 0 & 55\\ 
   & Geyer  & $\gamma=1.5$ &  749& 1.34 & 2.91 & 3.2 & 39 & 43 & 1.27 & 2.7 & 2.98 & 38& 51\\ 
\hline
  \multirow{3}{*}{$W_2$} & Strauss & $\gamma=0.2$ & 394 & 1.63 & 0.44 & 1.69 & 0 & 55 & 1.64 & 0.4 & 1.69 & 0 & 55\\ 
   & Strauss & $\gamma=0.5$ & 538 & 1.58 & 2.32 & 2.81  & 4 & 57 & 1.66 & 0.41 & 1.71 & 0 & 56\\ 
  & Geyer  & $\gamma=1.5$ &2964 & 1.48 & 0.83 & 1.7 & 2 & 44 & 1.1 & 0.48 & 1.2 & 3 & 95\\ 
   \hline
  \multirow{3}{*}{$W_3$} & Strauss & $\gamma=0.2$ & 1137& 1.63 & 1.42 & 2.16& 3 & 55 & 1.7 & 0.38 & 1.74& 0& 53\\ 
   & Strauss & $\gamma=0.5$ & 1485& 1.61 & 1.52 & 2.21& 7 & 56 & 1.73 & 0.51 & 1.8 & 0 & 53\\ 
  & Geyer  & $\gamma=1.5$ & 5623& 1.35 & 0.95 & 1.65 & 1 & 43 & 0.82 & 0.41 & 0.92 & 2 & 97\\ 
   \hline
 \end{tabular}
\end{table}

\newpage

\setlength{\tabcolsep}{1pt}
\renewcommand{\arraystretch}{1.5}
\begin{table}[!ht]
\caption{Empirical prediction properties (Bias, SD, and RMSE) and empirical selection properties (TPR, and FPR in $\%$) based on 500 replications of Strauss and Geyer models using cBIC, and cERIC for the Ridge penalty function and the regularized pseudo-likelihood function. The mean number of points $n$ under each model is provided.}
\label{table:ridge} 
\centering
\begin{tabular}{@{\extracolsep{1pt}}c l c c | ccccc | ccccc @{}}
\hline
\hline 
 \multicolumn{1}{c}{Spatial} & \multicolumn{1}{c}{Model} & \multicolumn{1}{c}{Interaction} & \multicolumn{1}{c}{Av. number} & \multicolumn{5}{c}{cBIC} & \multicolumn{5}{c}{cERIC} \\ 
 \cline{5-9} \cline{9-14}
domain &  & parameter & of points (n) & Bias & SD & RMSE & FPR & TPR & Bias & SD & RMSE & FPR & TPR  \\ 
  \hline
  \hline
  & \multicolumn{11}{c}{Scenario~\ref{sce1}}\\
\hline
  \multirow{3}{*}{$W_1$} & Strauss & $\gamma=0.2$ & 101 & 1.33 & 0.91 & 1.61 & 100 & 100 & 2.4 & 0.19 & 2.41 & 100 & 100\\ 
   & Strauss & $\gamma=0.5$ & 138 & 1.35 & 0.72 & 1.53 & 100 & 100 & 2.27 & 0.06 & 2.27 & 100 & 100\\ 
   & Geyer  & $\gamma=1.5$ & 750 & 1.43 & 0.39 & 1.48 & 100 & 100 & 1.46 & 0.39 & 1.51 & 100 & 100\\ 
\hline
  \multirow{3}{*}{$W_2$} & Strauss & $\gamma=0.2$ & 395& 1.33& 0.55& 1.44& 100 & 100 & 1.83 & 0.63 & 1.94 & 100& 100\\ 
   & Strauss & $\gamma=0.5$ & 540& 1.37& 0.44& 1.44& 100 & 100 & 1.65 & 0.43 & 1.71 &100  & 100\\ 
  & Geyer  & $\gamma=1.5$ & 2968& 1.44 & 0.2& 1.45 & 100 & 100 & 1.44 & 0.2 & 1.45& 100& 100\\ 
   \hline
  \multirow{3}{*}{$W_3$} & Strauss & $\gamma=0.2$ & 1137 & 1.37& 0.38 & 1.42& 100 & 100 & 1.52 &0.31  & 1.55 & 100& 100\\ 
   & Strauss & $\gamma=0.5$ &1484 & 1.39& 0.31& 1.42& 100 & 100 & 1.45 & 0.3 & 1.48 & 100 &100 \\ 
  & Geyer  & $\gamma=1.5$ &5630 & 1.44& 0.15& 1.45& 100 &100  & 1.44 & 0.15 & 1.45&100 & 100\\ 
   \hline
   & \multicolumn{11}{c}{Scenario~\ref{sce2}} \\
   \hline
   \multirow{3}{*}{$W_1$} & Strauss & $\gamma=0.2$ &101& 0.67& 1.06 & 1.25 & 100 & 100 & 2.4 & 0.19 & 2.41 & 100 & 100  \\ 
   & Strauss & $\gamma=0.5$ &138 & 0.72& 0.75& 1.04& 100& 100 & 2.27 & 0.06 & 2.27& 100&100 \\ 
   & Geyer  & $\gamma=1.5$ &749 & 1.04& 0.36& 1.1& 100& 100&1.05 & 0.37 & 1.11 & 100& 100\\ 
\hline
  \multirow{3}{*}{$W_2$} & Strauss & $\gamma=0.2$ &394 & 0.73& 0.56& 0.92 & 100 & 100 & 2.41 &0.09  & 2.41& 100&100 \\ 
   & Strauss & $\gamma=0.5$ & 538& 0.82& 0.39& 0.91& 100&100 &2.38 & 0.03& 2.28& 100& 100\\ 
  & Geyer  & $\gamma=1.5$ & 2964& 1.03& 0.17& 1.04& 100& 100&1.03 & 0.17 &1.04  &100 & 100\\ 
   \hline
  \multirow{3}{*}{$W_3$} & Strauss & $\gamma=0.2$ & 1137& 0.82& 0.32& 0.88&100 & 100& 1.16 &0.25  & 1.19 & 100 & 100\\ 
   & Strauss & $\gamma=0.5$ & 1485& 0.9&0.23 & 0.93&100 &100  & 0.94 &0.22  &0.97 &100& 100\\ 
  & Geyer  & $\gamma=1.5$ & 5623& 1.05& 0.1& 1.05& 100&100 &1.05 & 0.1 & 1.05 &100  & 100\\ 
   \hline
 \end{tabular}
\end{table}

\newpage

\setlength{\tabcolsep}{1pt}
\renewcommand{\arraystretch}{1.5}
\begin{table}[!ht]
\caption{Empirical prediction properties (Bias, SD, and RMSE) and empirical selection properties (TPR, and FPR in $\%$) based on 500 replications of Strauss and Geyer models using cBIC, and cERIC for the Elastic net penalty function and the regularized pseudo-likelihood function. The mean number of points $n$ under each model is provided.}
\label{table:enet} 
\centering
\begin{tabular}{@{\extracolsep{1pt}}c l c c | ccccc | ccccc @{}}
\hline
\hline 
 \multicolumn{1}{c}{Spatial} & \multicolumn{1}{c}{Model} & \multicolumn{1}{c}{Interaction} & \multicolumn{1}{c}{Av. number} & \multicolumn{5}{c}{cBIC} & \multicolumn{5}{c}{cERIC} \\ 
 \cline{5-9} \cline{9-14}
domain &  & parameter & of points (n) & Bias & SD & RMSE & FPR & TPR & Bias & SD & RMSE & FPR & TPR  \\ 
  \hline
  \hline
  & \multicolumn{11}{c}{Scenario~\ref{sce1}}\\
\hline
  \multirow{3}{*}{$W_1$} & Strauss & $\gamma=0.2$ & 101 & 1.89 & 0.5 & 1.96 & 8 & 70 & 1.85 & 0.4 & 1.89 & 10 & 82\\ 
   & Strauss & $\gamma=0.5$ & 138 & 1.76 & 0.35 &  1.79 & 10 & 92 & 1.68 & 0.26 & 1.7 & 12 & 99 \\ 
   & Geyer  & $\gamma=1.5$ & 750 & 1.39 & 0.32 & 1.43 & 12 & 100 & 1.38 & 0.31 & 1.41 & 12 & 100\\
\hline
  \multirow{3}{*}{$W_2$} & Strauss & $\gamma=0.2$ & 395 & 1.56& 0.23 & 1.58&8 & 100 & 1.49 & 0.23 & 1.51 & 10 & 100\\ 
   & Strauss & $\gamma=0.5$ & 540 & 1.53 & 0.18 & 1.54 & 8& 100 & 1.44 & 0.18 & 1.45 & 11 & 100\\ 
  & Geyer  & $\gamma=1.5$ & 2968 & 1.2& 0.2 & 1.22 & 9& 100 & 1.18 & 0.17 & 1.19 & 9 & 100\\ 
   \hline
  \multirow{3}{*}{$W_3$} & Strauss & $\gamma=0.2$ & 1137 & 1.43& 0.15&1.44 & 5& 100 & 1.37 & 0.15 & 1.38 & 7& 100\\ 
   & Strauss & $\gamma=0.5$ &  1484& 1.39& 0.13 & 1.4& 6& 100 & 1.32& 0.13&1.33 & 9&100 \\ 
  & Geyer  & $\gamma=1.5$ & 5630 & 1.14 &0.11  & 1.15 & 6& 100 & 1.12& 0.11& 1.13& 7&100 \\ 
   \hline
   & \multicolumn{11}{c}{Scenario~\ref{sce2}} \\
   \hline
   \multirow{3}{*}{$W_1$} & Strauss & $\gamma=0.2$ & 101& 1.74& 4.23& 4.57& 15& 34 & 1.72 & 3.95& 4.31&15&54 \\ 
   & Strauss & $\gamma=0.5$ & 138& 2.17&1.79& 2.81&1 &17 & 2.08& 0.15& 2.09 & 2 & 49\\ 
   & Geyer  & $\gamma=1.5$& 749& 1.17& 2.77& 3.01& 50& 50& 1.15& 2.53&2.78&50 &53 \\ 
\hline
  \multirow{3}{*}{$W_2$} & Strauss & $\gamma=0.2$ &394 & 2.32 & 0.81 & 2.46& 1 & 12 & 2.14& 0.12& 2.14 & 1 & 50\\ 
   & Strauss & $\gamma=0.5$ &538 & 1.29 & 5.5& 5.65 & 42& 52& 2.05 & 0.07 & 2.05 &2  &50 \\ 
  & Geyer  & $\gamma=1.5$ & 2964& 2.1 & 0.34 & 2.13 & 1&  2& 1.91 & 0.32 & 1.94 & 5 & 51\\ 
   \hline
  \multirow{3}{*}{$W_3$} & Strauss & $\gamma=0.2$ & 1137&2.06& 1.73 & 2.69& 7 &53 &1.9 & 0.73&2.04 &3 & 51\\ 
   & Strauss & $\gamma=0.5$ &1485 & 1.87&1.83&2.62&12 &56&1.61&1.75&2.38&15 & 57\\ 
  & Geyer  & $\gamma=1.5$ &5623&2.13&0.21&2.14 & 0&1& 1.93& 0.32&1.96 & 3& 52\\ 
   \hline
 \end{tabular}
\end{table}

\newpage

\setlength{\tabcolsep}{1pt}
\renewcommand{\arraystretch}{1.5}
\begin{table}[!ht]
\caption{Empirical prediction properties (Bias, SD, and RMSE) and empirical selection properties (TPR, and FPR in $\%$) based on 500 replications of Strauss and Geyer models using cBIC, and cERIC for the Adaptive lasso penalty function and the regularized pseudo-likelihood function. The mean number of points $n$ under each model is provided.}
\label{table:alasso} 
\centering
\begin{tabular}{@{\extracolsep{1pt}}c l c c | ccccc | ccccc @{}}
\hline
\hline 
 \multicolumn{1}{c}{Spatial} & \multicolumn{1}{c}{Model} & \multicolumn{1}{c}{Interaction} & \multicolumn{1}{c}{Av. number} & \multicolumn{5}{c}{cBIC} & \multicolumn{5}{c}{cERIC} \\ 
 \cline{5-9} \cline{9-14}
domain &  & parameter & of points (n) & Bias & SD & RMSE & FPR & TPR & Bias & SD & RMSE & FPR & TPR  \\ 
  \hline
  \hline
  & \multicolumn{11}{c}{Scenario~\ref{sce1}}\\
\hline
  \multirow{3}{*}{$W_1$} & Strauss & $\gamma=0.2$ & 101 & 1.94 & 1.05 & 2.21 & 2 & 23 & 1.93 & 0.61 & 2.02 & 1 & 53\\ 
   & Strauss & $\gamma=0.5$ & 138 & 1.58 & 1.2 & 1.98 & 8 & 33 & 1.66 & 0.48 & 1.73 & 0 & 60\\ 
   & Geyer  & $\gamma=1.5$ & 750 & 0.07 & 0.48 &  0.49 & 33 & 100 & 0.45 & 0.32 & 0.55  & 1 & 100\\ 
\hline
  \multirow{3}{*}{$W_2$} & Strauss & $\gamma=0.2$ &395 & 0.08& 0.53& 0.54 & 27& 100 & 0.61& 0.27 & 0.67 & 0 & 100\\ 
   & Strauss & $\gamma=0.5$ & 540& 0.06& 0.45 & 0.45& 32 & 100 & 0.57 & 0.16 & 0.59 &0  &100 \\ 
  & Geyer  & $\gamma=1.5$ & 2968& 0.06 & 0.18& 0.19 & 2& 100 & 0.18 & 0.16& 0.24 & 0 & 100\\ 
   \hline
  \multirow{3}{*}{$W_3$} & Strauss & $\gamma=0.2$ & 1137 & 0.1& 0.23& 0.25&9  & 100 & 0.38 & 0.12 & 0.4 & 0& 100\\ 
   & Strauss & $\gamma=0.5$ &1484 & 0.06 & 0.19 & 0.2 & 10 & 100 & 0.36 & 0.09 & 0.37 & 0 & 100\\ 
  & Geyer  & $\gamma=1.5$ & 5630& 0.05& 0.09 & 0.1 & 0 & 100 & 0.15 & 0.09 & 0.17 & 0 & 100\\ 
   \hline
   & \multicolumn{11}{c}{Scenario~\ref{sce2}} \\
   \hline
   \multirow{3}{*}{$W_1$} & Strauss & $\gamma=0.2$ & 101&1.09& 2.3& 2.55 & 5 & 53 & 0.82 & 2.15 & 2.3 & 7 & 78\\ 
   & Strauss & $\gamma=0.5$ & 138 & 0.48& 0.68 & 0.83 & 0 & 86 & 0.48& 0.6 & 0.77 & 3 & 63\\ 
   & Geyer  & $\gamma=1.5$ & 749& 0.09 & 0.35 & 0.36 & 0 & 100 & 0.37 & 0.3 & 0.48  & 0 & 100\\ 
\hline
  \multirow{3}{*}{$W_2$} & Strauss & $\gamma=0.2$ & 394 & 0.15 & 0.25 & 0.29  & 0 & 100 & 0.34 &  0.28& 0.44&1 & 100\\ 
   & Strauss & $\gamma=0.5$ & 538 & 0.1 & 0.18 & 0.21  & 0 &100  & 0.27 & 0.22 & 0.35 & 1 & 100\\ 
  & Geyer  & $\gamma=1.5$ & 2964 & 0.04 & 0.17 & 0.17 & 0 & 100 & 0.18 & 0.15 & 0.23 & 0 & 100\\ 
   \hline
  \multirow{3}{*}{$W_3$} & Strauss & $\gamma=0.2$ &1137&0.08&0.13&0.15&0&100&0.25&0.15&0.29&0 & 100\\ 
   & Strauss & $\gamma=0.5$ &1485&0.06& 0.1& 0.12&0&100&0.22&0.12&0.25&0&100 \\ 
  & Geyer  & $\gamma=1.5$ &5623&0.04&0.08&0.09&0&100&0.16&0.08&0.18&0&100 \\ 
   \hline
 \end{tabular}
\end{table}

\newpage

\setlength{\tabcolsep}{1pt}
\renewcommand{\arraystretch}{1.5}
\begin{table}[!ht]
\caption{Empirical prediction properties (Bias, SD, and RMSE) and empirical selection properties (TPR, and FPR in $\%$) based on 500 replications of Strauss and Geyer models using cBIC, and cERIC for the Adaptive Elastic Net penalty function and the regularized pseudo-likelihood function. The mean number of points $n$ under each model is provided.}
\label{table:aenet} 
\centering
\begin{tabular}{@{\extracolsep{1pt}}c l c c | ccccc | ccccc @{}}
\hline
\hline 
 \multicolumn{1}{c}{Spatial} & \multicolumn{1}{c}{Model} & \multicolumn{1}{c}{Interaction} & \multicolumn{1}{c}{Av. number} & \multicolumn{5}{c}{cBIC} & \multicolumn{5}{c}{cERIC} \\ 
 \cline{5-9} \cline{9-14}
domain &  & parameter & of points (n) & Bias & SD & RMSE & FPR & TPR & Bias & SD & RMSE & FPR & TPR  \\ 
  \hline
  \hline
  & \multicolumn{11}{c}{Scenario~\ref{sce1}}\\
\hline
  \multirow{3}{*}{$W_1$} & Strauss & $\gamma=0.2$ & 101 & 1.6 & 0.94 & 1.86 & 3 & 53 & 2.17 & 0.48 & 2.22 & 1 & 54\\ 
   & Strauss & $\gamma=0.5$ & 138 & 1.1 & 0.63 &  1.27 & 4 & 90 & 2.07 & 0.36 & 2.1 & 1 & 54\\ 
   & Geyer  & $\gamma=1.5$ & 750 & 0.58 & 0.43 & 0.72 & 4 & 100 & 0.9 & 0.35 & 0.97 & 3 & 100\\
\hline
  \multirow{3}{*}{$W_2$} & Strauss & $\gamma=0.2$ &395 & 0.77& 0.33 & 0.84 & 2 & 100 & 0.93 & 0.42& 1.02 & 3 & 98\\ 
   & Strauss & $\gamma=0.5$ & 540& 0.73 & 0.3& 0.79 & 3 & 100 & 0.85 & 0.31 & 0.9 & 4 & 100\\ 
  & Geyer  & $\gamma=1.5$ & 2968 & 0.32 & 0.23 & 0.39 & 4 & 100 & 0.65 & 0.22 & 0.69 & 2 & 100\\ 
   \hline
  \multirow{3}{*}{$W_3$} & Strauss & $\gamma=0.2$ & 1137 & 0.61& 0.25 & 0.66 & 2 & 100 & 0.82 & 0.22 & 0.85 & 2 & 100\\ 
   & Strauss & $\gamma=0.5$ &1484 & 0.53 & 0.22 & 0.57 & 2&100  & 0.74 & 0.22 & 0.77 & 2 & 100\\ 
  & Geyer  & $\gamma=1.5$ & 5630 & 0.25 & 0.16 & 0.3 & 2 & 100& 0.5 &0.17  & 0.53 & 2 & 100\\ 
   \hline
   & \multicolumn{11}{c}{Scenario~\ref{sce2}} \\
   \hline
   \multirow{3}{*}{$W_1$} & Strauss & $\gamma=0.2$ & 101& 1.32 & 1.54 & 2.03& 5& 50 & 0.68 & 1.28 & 1.45 & 9 & 90\\ 
   & Strauss & $\gamma=0.5$ & 138& 0.82& 0.69 & 1.07 & 0 & 82 & 0.47 & 0.56& 0.73 & 6 & 98\\ 
   & Geyer  & $\gamma=1.5$ & 749& 0.31 & 0.39& 0.5  & 0 & 100& 0.47 & 0.31 & 0.56  & 0 & 100\\
\hline
  \multirow{3}{*}{$W_2$} & Strauss & $\gamma=0.2$ & 394& 0.36& 0.3 & 0.47 & 0 & 99 & 0.34 & 0.31 &0.46  & 3 & 100\\ 
   & Strauss & $\gamma=0.5$ & 538& 0.27& 0.23& 0.35 & 0&100 & 0.26 & 0.23& 0.35& 3&100 \\ 
  & Geyer  & $\gamma=1.5$ &2964 & 0.1 & 0.17 & 0.2& 0 & 100& 0.2 & 0.15 & 0.25 & 0 & 100\\ 
   \hline
  \multirow{3}{*}{$W_3$} & Strauss & $\gamma=0.2$ &1137&0.17&0.15&0.23&0&100&0.24&0.15&0.28&1&100 \\ 
   & Strauss & $\gamma=0.5$ &1485&0.13&0.11&0.17&0&100&0.2&0.13&0.24&1&100 \\ 
  & Geyer  & $\gamma=1.5$ &5623&0.07&0.09&0.11&0&100&0.17&0.07&0.18&0&100 \\ 
   \hline
 \end{tabular}
\end{table}

\newpage

\setlength{\tabcolsep}{1pt}
\renewcommand{\arraystretch}{1.5}
\begin{table}[!ht]
\caption{Empirical prediction properties (Bias, SD, and RMSE) and empirical selection properties (TPR, and FPR in $\%$) based on 500 replications of Strauss and Geyer models using cBIC, and cERIC for the SCAD penalty function and the regularized pseudo-likelihood function. The mean number of points $n$ under each model is provided.}
\label{table:scad} 
\centering
\begin{tabular}{@{\extracolsep{1pt}}c l c c | ccccc | ccccc @{}}
\hline
\hline 
 \multicolumn{1}{c}{Spatial} & \multicolumn{1}{c}{Model} & \multicolumn{1}{c}{Interaction} & \multicolumn{1}{c}{Av. number} & \multicolumn{5}{c}{cBIC} & \multicolumn{5}{c}{cERIC} \\ 
 \cline{5-9} \cline{9-14}
domain &  & parameter & of points (n) & Bias & SD & RMSE & FPR & TPR & Bias & SD & RMSE & FPR & TPR  \\ 
  \hline
  \hline
  & \multicolumn{11}{c}{Scenario~\ref{sce1}}\\
\hline
  \multirow{3}{*}{$W_1$} & Strauss & $\gamma=0.2$ & 101 & 2.16 & 0.86 & 2.32 & 2 & 12 & 2.31 & 0.61 & 2.39 & 1 & 4\\ 
   & Strauss & $\gamma=0.5$ & 138 & 0.42 & 1.6 & 1.65 & 65 & 83 & 0.91 & 1.45 & 1.71 & 33 & 71\\ 
   & Geyer  & $\gamma=1.5$ & 750 & 0.23 & 0.6 &  0.64 & 81 & 100 & 0.29 & 0.71 & 0.77 & 68 & 100\\
\hline
  \multirow{3}{*}{$W_2$} & Strauss & $\gamma=0.2$ & 395 & 0.1 & 0.87 & 0.88 & 60& 98 & 2.16 & 0.55 & 2.23 &  1& 19\\ 
   & Strauss & $\gamma=0.5$ &  540& 0.06 & 0.8 & 0.8 & 79 & 99 & 0.95 &0.97 & 1.36 & 4 & 77\\ 
  & Geyer  & $\gamma=1.5$ & 2968& 0.29 & 0.25 & 0.38  & 34 & 100 & 0.3 & 0.21 &0.37  & 4 & 100\\ 
   \hline
  \multirow{3}{*}{$W_3$} & Strauss & $\gamma=0.2$ & 1137& 0.06& 0.36 &0.36  & 12& 100&0.12  &0.48  & 0.49 & 0 & 98\\ 
   & Strauss & $\gamma=0.5$ & 1484& 0.04& 0.35 & 0.35 &14 & 100& 0.08& 0.42 & 0.43& 1&100 \\ 
  & Geyer  & $\gamma=1.5$ & 5630& 0.27& 0.15 & 0.31 & 28& 100 & 0.27 & 0.11 & 0.29 & 1 & 100\\ 
   \hline
   & \multicolumn{11}{c}{Scenario~\ref{sce2}} \\
   \hline
   \multirow{3}{*}{$W_1$} & Strauss & $\gamma=0.2$ & 101& 0.9& 5.4& 5.47& 7 & 57 & 1.53 & 5.41 & 5.62& 7 & 35\\ 
   & Strauss & $\gamma=0.5$ & 138& 0.16& 0.65& 0.67  & 0& 93 & 0.75& 0.91& 1.18 & 0 & 75\\ 
   & Geyer  & $\gamma=1.5$ & 749& 0.28 & 0.34 & 0.44 & 0& 100 & 0.28 & 0.35& 0.45& 0 & 99\\ 
\hline
  \multirow{3}{*}{$W_2$} & Strauss & $\gamma=0.2$ & 394& 0.02& 0.25& 0.25 & 0& 100& 0.36 & 0.67 &0.76  & 0 & 80\\ 
   & Strauss & $\gamma=0.5$ & 538& 0.01& 0.17& 0.17 & 0& 100 & 0.03 & 0.19& 0.19& 0&100 \\ 
  & Geyer  & $\gamma=1.5$ & 2964 & 0.25 & 0.15 & 0.29& 0& 100 & 0.26 & 0.15 & 0.3 & 0& 100\\ 
   \hline
  \multirow{3}{*}{$W_3$} & Strauss & $\gamma=0.2$ & 1137&0.04&0.13&0.14&0&100&0.04&0.13&0.14&0&100 \\ 
   & Strauss & $\gamma=0.5$ &1485&0.01&0.1&0.1&0&100&0.02&0.1&0.1&0&100 \\ 
  & Geyer  & $\gamma=1.5$ & 5623&0.25&0.07&0.26&0&100&0.26&0.07&0.27 &0&100 \\ 
   \hline
 \end{tabular}
\end{table}

\newpage

\setlength{\tabcolsep}{1pt}
\renewcommand{\arraystretch}{1.5}
\begin{table}[!ht]
\caption{Empirical prediction properties (Bias, SD, and RMSE) and empirical selection properties (TPR, and FPR in $\%$) based on 500 replications of Strauss and Geyer models using cBIC, and cERIC for the MC+ penalty function and the regularized pseudo-likelihood function. The mean number of points $n$ under each model is provided.}
\label{table:mcp} 
\centering
\begin{tabular}{@{\extracolsep{1pt}}c l c c | ccccc | ccccc @{}}
\hline
\hline 
 \multicolumn{1}{c}{Spatial} & \multicolumn{1}{c}{Model} & \multicolumn{1}{c}{Interaction} & \multicolumn{1}{c}{Av. number} & \multicolumn{5}{c}{cBIC} & \multicolumn{5}{c}{cERIC} \\ 
 \cline{5-9} \cline{9-14}
domain &  & parameter & of points (n) & Bias & SD & RMSE & FPR & TPR & Bias & SD & RMSE & FPR & TPR  \\ 
  \hline
  \hline
  & \multicolumn{11}{c}{Scenario~\ref{sce1}}\\
\hline
  \multirow{3}{*}{$W_1$} & Strauss & $\gamma=0.2$ & 101 & 2.1 & 0.92 & 2.29 & 2 & 11 & 2.28 & 0.61 & 2.36 & 1& 4\\ 
   & Strauss & $\gamma=0.5$ & 138 & 0.11 & 1.51 & 1.51 & 80 & 91 & 0.54 & 1.35 & 1.45 & 38 & 76\\ 
   & Geyer  & $\gamma=1.5$ & 750 & 0.21 & 0.56 &  0.6 & 82 & 100 & 0.22 & 0.53 & 0.57  & 70 & 100\\ 
\hline
  \multirow{3}{*}{$W_2$} & Strauss & $\gamma=0.2$ &395  & 0.02 & 0.76 & 0.76  & 61 & 100 & 2.25 & 0.64 & 2.34 & 0 & 7\\ 
   & Strauss & $\gamma=0.5$ &540 & 0.02 & 0.72& 0.72  & 79& 100 & 0.12 & 0.33 &0.35  & 1 & 99\\ 
  & Geyer  & $\gamma=1.5$ & 2968& 0.26 & 0.24 & 0.35 & 46 & 100 & 0.25& 0.16 & 0.3& 3 & 100\\ 
   \hline
  \multirow{3}{*}{$W_3$} & Strauss & $\gamma=0.2$ & 1137 & 0.02 & 0.23 & 0.23& 10& 100 & 0.03 &0.12  & 0.12 & 0 &100 \\ 
   & Strauss & $\gamma=0.5$ & 1484 & 0.01 & 0.23& 0.23 & 14 & 100 & 0.02& 0.09 & 0.09& 0 &100 \\ 
  & Geyer  & $\gamma=1.5$ & 5630& 0.25 & 0.13&  0.28& 30 & 100 & 0.26& 0.07 & 0.27 & 0 & 100\\ 
   \hline
   & \multicolumn{11}{c}{Scenario~\ref{sce2}} \\
   \hline
   \multirow{3}{*}{$W_1$} & Strauss & $\gamma=0.2$ & 101& 0.82& 6.1 & 6.15 & 9& 58 & 1.37 & 6.16 & 6.31 & 9 & 37\\ 
   & Strauss & $\gamma=0.5$ & 138& 0.1& 0.55 & 0.56 & 0 & 92 & 0.7 & 0.75 & 1.03 & 0 & 68\\ 
   & Geyer  & $\gamma=1.5$ & 749& 0.27& 0.32 & 0.42 & 0& 100 & 0.29& 0.38 & 0.48  & 0 & 99\\ 
\hline
  \multirow{3}{*}{$W_2$} & Strauss & $\gamma=0.2$ & 394& 0.02& 0.24 & 0.24&0 & 100& 0.4 & 0.62& 0.74 & 0 & 77\\ 
   & Strauss & $\gamma=0.5$ & 538 & 0.01 & 0.17& 0.17  & 0 & 100 & 0.04 & 0.19 & 0.19 & 0 & 100\\ 
  & Geyer  & $\gamma=1.5$ & 2964& 0.25 & 0.15 & 0.29 & 0&100  & 0.26 & 0.15 &0.3 & 0& 100\\ 
   \hline
  \multirow{3}{*}{$W_3$} & Strauss & $\gamma=0.2$ &1137&0.04&0.13&0.14&0&100&0.04&0.13&0.14&0&100 \\ 
   & Strauss & $\gamma=0.5$ &1485&0.01&0.1&0.1&0&100&0.02&0.1&0.1&0&100 \\ 
  & Geyer  & $\gamma=1.5$ & 5623&0.25&0.07&0.26&0&100&0.26&0.07&0.27&0&100 \\ 
   \hline
 \end{tabular}
\end{table}
\newpage


\bibliographystyle{imsart-nameyear}
\bibliography{SelectionGPP_arxiv2}




\end{document}